\documentclass[12pt,reqno]{amsart}
\usepackage{moreverb}

\usepackage{amsmath,amssymb}
\usepackage{fullpage}
\usepackage{graphicx,psfrag, epstopdf, subfigure}
\usepackage{color}
\usepackage{ifthen}

\def\A{\mathbf{A}}       
\def\b{\mathbf{b}}       
\def\r{c}                

\def\subsection#1
{
 \bigskip

 \refstepcounter{subsection}
 {\bf\arabic{section}.\arabic{subsection}.~#1.~}
}

\def\subsubsection#1{\bigskip \noindent{\bf#1.}}

\newtheorem{theorem}{Theorem}
\newtheorem{lemma}[theorem]{Lemma}
\newtheorem{corollary}[theorem]{Corollary}
\newtheorem{proposition}[theorem]{Proposition}
\newtheorem{algorithm}[theorem]{Algorithm}
\newtheorem{remark}[theorem]{Remark}

\def\eps{\varepsilon}

\def\set#1#2{\big\{#1\,:\,#2\big\}}

\newcommand{\norm}[3][]{#1\|#2#1\|_{#3}}

\newcommand{\enorm}[2][]{#1|\hspace*{-.3mm}#1|\hspace*{-.3mm}#1|#2#1|\hspace*{-.3mm}#1|\hspace*{-.3mm}#1|}

\def\div{{\rm div}}

\def\osc{ {\rm osc} }

\newcounter{constantsnumber}
\def\setc#1{
  \ifthenelse{\equal{#1}{rel}}{C_{\rm rel}}{ 
  \ifthenelse{\equal{#1}{reduction}}{C_{\rm est}}{ 
  \ifthenelse{\equal{#1}{invest}}{C_{\rm Vinv}}{ 
  \ifthenelse{\equal{#1}{trace}}{C_{\rm trace}}{ 
  \ifthenelse{\equal{#1}{caccioppoli}}{C_{\rm cacc}}{ 
  \ifthenelse{\equal{#1}{farfield}}{C_{\rm far}}{ 
  \ifthenelse{\equal{#1}{nearfield}}{C_{\rm near}}{ 
  \ifthenelse{\equal{#1}{nearfield2}}{C_{\rm near2}}{ 
  \ifthenelse{\equal{#1}{nvb}}{C_{\rm nvb}}{ 
  \ifthenelse{\equal{#1}{optimal}}{C_{\rm opt}}{ 
  \ifthenelse{\equal{#1}{eff}}{C_{\rm eff}}{ 
  \ifthenelse{\equal{#1}{osc}}{C_{\rm osc}}{ 
  \ifthenelse{\equal{#1}{galerkin}}{C_{\rm gal}}{ 
  \ifthenelse{\equal{#1}{dlr}}{C_{\rm dlr}}{ 
  \ifthenelse{\equal{#1}{diff1}}{C_{\A,1}}{ 
  \ifthenelse{\equal{#1}{diff2}}{C_{\A,2}}{ 
  \ifthenelse{\equal{#1}{convreact}}{C_{\b\r}}{ 
  \ifthenelse{\equal{#1}{bil}}{C_{\rm bil}}{ 
  \refstepcounter{constantsnumber}
  \label{const#1}C_{\theconstantsnumber}}}}}}}}}}}}}}}}}}}}

\def\c#1{
  \ifthenelse{\equal{#1}{rel}}{C_{\rm rel}}{ 
  \ifthenelse{\equal{#1}{reduction}}{C_{\rm est}}{ 
  \ifthenelse{\equal{#1}{invest}}{C_{\rm Vinv}}{ 
  \ifthenelse{\equal{#1}{trace}}{C_{\rm trace}}{ 
  \ifthenelse{\equal{#1}{caccioppoli}}{C_{\rm cacc}}{ 
  \ifthenelse{\equal{#1}{farfield}}{C_{\rm far}}{ 
  \ifthenelse{\equal{#1}{nearfield}}{C_{\rm near}}{ 
  \ifthenelse{\equal{#1}{nearfield2}}{C_{\rm near2}}{ 
  \ifthenelse{\equal{#1}{nvb}}{C_{\rm nvb}}{ 
  \ifthenelse{\equal{#1}{optimal}}{C_{\rm opt}}{ 
  \ifthenelse{\equal{#1}{eff}}{C_{\rm eff}}{ 
  \ifthenelse{\equal{#1}{galerkin}}{C_{\rm gal}}{ 
  \ifthenelse{\equal{#1}{osc}}{C_{\rm osc}}{ 
  \ifthenelse{\equal{#1}{dlr}}{C_{\rm dlr}}{ 
  \ifthenelse{\equal{#1}{diff1}}{C_{\A,1}}{ 
  \ifthenelse{\equal{#1}{diff2}}{C_{\A,2}}{ 
  \ifthenelse{\equal{#1}{convreact}}{C_{\b\r}}{ 
  \ifthenelse{\equal{#1}{bil}}{C_{\rm bil}}{ 
  C_{\ref{const#1}}}}}}}}}}}}}}}}}}}}}

\def\<{\langle\hspace*{-.9mm}\langle}
\def\>{\rangle\hspace*{-.9mm}\rangle}

\def\product#1#2{(#1\hspace*{.5mm},#2)}

\def\N{{\mathbb N}}

\def\R{{\mathbb R}}
\def\T{{\mathbb T}}

\def\nf{{\mathbf n}}

\def\AA{{\mathcal A}}  
\def\CC{{\mathcal C}}

\def\II{{\mathcal I}}
\def\MM{{\mathcal M}}
\def\NN{{\mathcal N}}
\def\OO{{\mathcal O}}
\def\PP{{\mathcal P}}
\def\RR{{\mathcal R}}

\def\SS{{\mathcal S}}
\def\TT{{\mathcal T}}

\def\ZZ{{\mathcal Z}}

\def\refine{{\tt refine}}

\def\EEt{\EE_{T}}               

\def\[{[\hspace*{-0.8mm}[}
\def\]{]\hspace*{-0.8mm}]}
\newcommand{\jump}[2][]{%
 \ifthenelse{\equal{#1}{}}
 {[\hspace*{-0.8mm}[#2]\hspace*{-0.8mm}]}%
 {\left[\hspace*{-1.6mm}\left[#2\right]\hspace*{-1.6mm}\right]}%
}

\def\diam{{\rm diam}}

\def\Cmark{C_{\rm mark}}
\def\Clin{C_{\rm lin}}
\def\qlin{q_{\rm lin}}
\def\Cest{C_{\rm est}}
\def\qest{q_{\rm est}}
\def\Copt{C_{\rm opt}}
\def\Crel{C_{\rm rel}}

\def\Ctot{C_{\rm tot}}
\def\Cmns{C_{\rm MNS}}%

\def\UU{\mathcal U}
\def\refine{{\tt refine}}

\def\EE{{\mathcal F}}
\def\patch{\Omega}
\def\facet{F}
\def\star{\times}

\title[Adaptive finite volume methods]{Adaptive vertex-centered finite volume methods\\ with convergence rates}
\author{Christoph Erath}
\address{TU Darmstadt, Department of Mathematics, Dolivostra\ss{}e 15, 64293 Darmstadt, Germany}
\email{Erath@mathematik.tu-darmstadt.de\quad\rm(corresponding author)}

\author{Dirk Praetorius}
\address{TU Wien, Institute for Analysis and Scientific Computing,
Wiedner Hauptstra\ss{}e 8-10, 1040 Wien, Austria}
\email{Dirk.Praetorius@tuwien.ac.at}

\thanks{C. Erath (corresponding author): TU Darmstadt, Germany; Erath@mathematik.tu-darmstadt.de}
\thanks{D. Praetorius: TU Wien, Austria; Dirk.Praetorius@tuwien.ac.at}
\thanks{The second author acknowledges support through the research project
\emph{Optimal adaptivity for BEM and FEM-BEM Coupling} funded by the Austrian Science Fund (FWF) under
grant P27005}

\date{\bf\color{red}\today}
\begin{document}

\begin{abstract}
We consider the vertex-centered finite volume method with first-order conforming ansatz functions. The adaptive mesh-refinement is driven by the local contributions of the weighted-residual error estimator. We prove that the adaptive algorithm leads to linear convergence with generically optimal algebraic rates for the error estimator and the sum of energy error plus data oscillations. While similar results have been derived for finite element methods and boundary element methods, the present work appears to be the first for adaptive finite volume methods, where the lack of the classical Galerkin orthogonality leads to new challenges.
\end{abstract}

\maketitle

\section{Introduction}

\subsection{Finite volume method}
A classical finite volume method (FVM) describes numerically a conservation 
law of an underlying model problem, which
might be described by a partial differential equation (PDE). 
In particular, it naturally preserves local conservation of the numerical fluxes. Therefore, FVMs are
well-established in the engineering community (fluid mechanics). 
Even though the FVM has a wide range of applications the numerical analysis is less developed 
than for the more prominent finite element method (FEM).
There exist different versions of the FVM like the 
cell-centered FVM, which basically yields to 
a piecewise constant approximation of the unknown solution on a primal mesh.
For more details we refer to~\cite{Eymard:2000-book}. 
The so-called vertex-centered FVM (finite volume element method, box method) belongs to the other big
family of FVMs, where one usually introduces an additional dual mesh around the nodes for the approximation.
In this work, we focus on the lowest-order vertex-centered finite volume method (from now on only FVM) 
for some elliptic model problem in $\R^d$, $d=2,3$.
The first relevant mathematical analysis of this method 
started with the works~\cite{Bank:1987-1, Hackbusch:1989-1,Cai:1991-1}.

\subsection{A~posteriori error estimation and adaptive mesh-refinement}
Accurate {\sl a~posteriori} error estimation and related adaptive mesh-refinement is one fundamental column of modern scientific computing. On the one hand, the {\sl a~posteriori} error estimator allows to monitor whether a numerical approximation is sufficiently accurate, even though the exact solution is unknown. On the other hand, it allows to adapt the discretization to resolve possible singularities most effectively. 
Over the last few years, the mathematical understanding of adaptive mesh-refinement has matured. It has been proved that adaptive procedures for the finite element method (FEM) as well as for the boundary element method (BEM) lead to optimal convergence behavior of the numerical scheme; see, 
e.g.,~\cite{doerfler96,mns00,bdd04,stevenson07,ckns,ffp14} for FEM, \cite{fkmp,part1,part2,gantumur} for BEM, and~\cite{axioms} for some general framework.

In this work, we analyze an adaptive mesh-refining algorithm of the type
\begin{align}\label{eq:semr}
\boxed{\tt SOLVE} 
\quad\Longrightarrow\quad
\boxed{\tt ESTIMATE}
\quad\Longrightarrow\quad
\boxed{\tt MARK} 
\quad\Longrightarrow\quad
\boxed{\tt REFINE}
\end{align}
in the frame of the FVM (Algorithm~\ref{algorithm:mns}). 
Given a conforming triangulation $\TT_\ell$, the module \texttt{SOLVE} uses FVM to compute a discrete approximation $u_\ell$ to the solution $u$ of the PDE. For the ease of presentation, we assume that the linear system is solved exactly, although, in the spirit of~\cite[Section 7]{axioms}, stopping criteria for iterative solvers can be included into our analysis. The module \texttt{ESTIMATE} employs a weighted-residual error estimator $\eta_\ell$ 
from~\cite{Carstensen:2005-1,xzz06} which is also well-studied in the context of adaptive finite element methods~\cite{stevenson07,ckns,ffp14}. The module \texttt{MARK} uses the D\"orfler marking criterion introduced in~\cite{doerfler96}, to mark elements for refinement, where the local error appears to be large. Unlike common algorithms for FEM and BEM, we follow~\cite{mns00} and also mark elements with respect to the data oscillations to overcome the lack of the Galerkin orthogonality. Finally, the module \texttt{REFINE} employs newest vertex bisection (NVB) to refine the marked elements and to generate a new conforming triangulation $\TT_{\ell+1}$ which better resolves the present singularities.

\subsection{Contributions of the present work}
Iteration of the adaptive loop~\eqref{eq:semr} provides a sequence of successively refined triangulations $\TT_\ell$ together with the corresponding FVM solutions $u_\ell$ and the {\sl a~posteriori} error estimators $\eta_\ell$. Theorem~\ref{theorem:mns} below proves that this adaptive iteration leads to linear convergence in the sense of
\begin{align}\label{eq:intro1}
 \eta_{\ell+n} \le Cq^n\,\eta_\ell
 \quad\text{for all }\ell,n\in\N_0
\end{align}
with some independent constants $C>0$ and $0<q<1$. Under an additional assumption on the marking which can be monitored {\sl a~posteriori}, we prove optimal convergence behavior 
\begin{align}\label{eq:intro2}
 \eta_\ell \le C\,(\#\TT_\ell-\#\TT_0)^{-s}
\end{align}
for each ``possible'' algebraic rate $s>0$ (in the sense of certain nonlinear approximation classes which are defined in Section~\ref{section:main} below), where $\#\TT_\ell$ denotes the number of elements in $\TT_\ell$.
These results can be equivalently stated with respect to the sum of energy error plus data oscillations, which is usually done in the FEM literature~\cite{stevenson07,ckns,ffp14},
since
\begin{align}\label{eq:intro3}
 C^{-1}\,\eta_\ell \le \min_{v_\ell}\big(\enorm{u-v_\ell}+\osc_\ell(v_\ell)\big)
 \le \enorm{u-u_\ell}+\osc_\ell(u_\ell)
 \le C\,\eta_\ell;
\end{align}%
see Theorem~\ref{theorem:cea} below.
We note that~\eqref{eq:intro3} in particular provides a generalized C\'ea lemma which states that the FVM solution $u_\ell$ is quasi-optimal with respect to the so-called total error, i.e., the sum of energy error plus data oscillations. Since~\eqref{eq:intro3} is also known for 
the FEM (see, e.g.,~\cite[Lemma~5.1]{ffp14}), this reveals 
that FEM and FVM lead to equivalent errors in the sense of 
\begin{align}\label{eq:intro4}
 C^{-1}\,\big(\enorm{u-u_\ell}+\osc_\ell(u_\ell)\big)
 \le \enorm{u-u_\ell^{\rm FEM}} +\osc_\ell(u_\ell^{\rm FEM})
 \le C\,\big(\enorm{u-u_\ell}+\osc_\ell(u_\ell)\big),
\end{align}
where $u_\ell^{\rm FEM}$ is the FEM solution with respect to the FVM space.
This complements recent results which compare the total errors of different FEM discretizations~\cite{MR3022243,MR3349689}.

Unlike the results for FEM and BEM, the novel C\'ea-type estimate~\eqref{eq:intro3} as well as our result~\eqref{eq:intro1}--\eqref{eq:intro2} on adaptive FVM requires the additional assumption that the initial triangulation $\TT_0$ is sufficiently fine. We note, however, that such an assumption is also required to prove well-posedness of the FVM in general and thus appears naturally.
%

Prior to this work, {\sl a~posteriori} error estimates for the FVM for elliptic model problems are derived in~\cite{Carstensen:2005-1,xzz06,Zou:2009-1}; 
see also~\cite[Remark 6.1]{Erath:2013-1} and~\cite[Conclusions]{Erath:2013-1} 
for estimates which are robust with respect to the lower-order convection and reaction terms.
To the best of the authors' knowledge, convergence of an adaptive 2D FVM has only been analyzed in the yet unpublished preprint~\cite{xzz06}. 
The latter is concerned with convergence only and 
the analysis follows~\cite{mns00} and relies on a discrete efficiency estimate and hence on the so-called interior node property of the mesh-refinement. Contrary to~\cite{xzz06}, our analysis extends the ideas of~\cite{ckns} and provides a contraction property for the weighted sum of energy error, weighted-residual error estimator, and data oscillations. Therefore, our analysis 
covers in particular standard NVB, where marked elements are refined by one bisection.

We finally note that residual error estimators have also been developed for the cell-centered finite volume method~\cite{Nicaise:2005-1,Erath:2008-1,Vohralik:2008-1}.
These a~posteriori estimators rely on an interpolatory post-processing
of the original piecewise constant cell-centered finite volume approximation.
Thus, a thorough adaptive convergence analysis requires additional ideas
to extend and adapt the analysis presented below.
\subsection{General notation}
%
We use $\lesssim$ to abbreviate $\le$ up to some (generic) multiplicative constant which is clear from the context. Moreover, $\simeq$ abbreviates that both estimates $\lesssim$ and $\gtrsim$ hold.
Throughout, the mesh-dependence of (discrete) quantities is explicitly stated by use of appropriate indices, e.g., $u_\star$ is the FVM solution for the triangulation $\TT_\star$ and 
$\eta_\ell$ is the error estimator with respect to the triangulation $\TT_\ell$.

\section{Model problem \& main results}
\label{section:results}%

\subsection{Model problem}
\label{section:modelproblem}%
Let $\Omega\subset \R^d$, $d=2,3$, be a bounded and connected Lipschitz domain with boundary $\Gamma:=\partial\Omega$.
As model problem, we consider the following stationary diffusion problem: Given $f\in L^2(\Omega)$, find $u\in H^1(\Omega)$ such that
\begin{subequations}
\label{eq:model}
\begin{align}
 -\div \A \nabla u  &= f \quad \text{in }\Omega,\\
                               u &= 0\quad \text{on }\Gamma.
\end{align}
\end{subequations}
We suppose that the diffusion matrix $\A=\A(x)\in\R^{d\times d}$ is bounded, symmetric, and uniformly positive definite, i.e., there exist constants $\lambda_{\rm min},\lambda_{\rm max}>0$ such that
\begin{align}\label{eq:A}
 \lambda_{\rm min}\,|\mathbf{v}|^2\leq \mathbf{v}^T\A(x)\mathbf{v}\leq \lambda_{\rm max}\,|\mathbf{v}|^2
 \quad\text{for all } \mathbf{v}\in \R^d \text{ and almost all }x\in\Omega.
\end{align}
For convergence of our FVM and well-posedness of the residual error estimator, 
we additionally require that $\A(x)$ is piecewise Lipschitz continuous, i.e., 
\begin{align}\label{eq3:A}
 \A\in W^{1,\infty}(T)^{d\times d}
 \quad\text{for all}\quad T\in\TT_0, 
\end{align} 
where $\TT_0$ is some given initial triangulation of $\Omega$; see Section~\ref{section:estimator} below.

The weak formulation of the model problem~\eqref{eq:model} reads: Find $u\in H^1_0(\Omega)$ such that 
\begin{align}
\label{eq:weakform}
 \AA(u,v):= \product{\A\nabla u}{\nabla v}_{\Omega} = \product{f}{v}_{\Omega}
 \quad\text{for all } v\in H^1_0(\Omega),
\end{align}
where $\product{\phi}{\psi}_\Omega = \int_\Omega \phi(x)\psi(x)\,dx$ denotes the $L^2$-scalar product.
According to our assumptions~\eqref{eq:A} on $\A$, the bilinear form $\AA(\cdot,\cdot)$ is continuous and elliptic on $H^1_0(\Omega)$. Therefore, 
existence and uniqueness of the solution $u\in H^1_0(\Omega)$ of~\eqref{eq:weakform} follow from the Lax-Milgram theorem. Moreover, $\enorm{v}^2:=\AA(v,v)$ defines the so-called energy norm which is an equivalent norm on $H^1_0(\Omega)$. We shall use the notation $\enorm{v}_\omega^2:=\int_\omega\A\nabla v\cdot\nabla v$ for the energy norm on subdomains $\omega\subseteq\Omega$, i.e., $\enorm{v}=\enorm{v}_\Omega$. According to~\eqref{eq:A}, it holds $\enorm{v}_\omega \simeq \norm{\nabla v}{L^2(\omega)}$.

\subsection{Triangulation}
\label{section:triangulation}%
Throughout, $\TT_\star$ denotes a conforming triangulation of $\Omega$ into non-degenerated closed simplices $T\in\TT_\star$ (i.e., triangles for $d=2$, tetrahedra for $d=3$),
$\NN_\star$ is the corresponding set of nodes, and 
$\EE_\star$ is the corresponding set of facets (i.e., edges for $d=2$ and triangular faces for $d=3$). We suppose that $\TT_\star$ is $\sigma$-shape regular, i.e.,
\begin{align}\label{eq:shaperegular}
 \max_{T\in\TT_\star}\frac{\diam(T)^d}{|T|} \le \sigma < \infty.
\end{align}
Here, $\diam(T) := \max\set{|x-y|}{x,y\in T}$ denotes the Euclidean diameter and $|T|$ is the area of $T$.  Additionally, we assume that the triangulation $\TT_\star$ is aligned with the discontinuities of the coefficient matrix $\A$, i.e.,~\eqref{eq3:A} holds with $\TT_0$ replaced by $\TT_\star$. We note that this follows from~\eqref{eq3:A} and the mesh-refinement used; see Section~\ref{section:main}.
Associated with $\TT_\star$ is the local mesh-size function $h_\star\in L^\infty(\Omega)$ which is defined by $h_\star|_T:=h_T:=|T|^{1/d}$. Note that $\sigma$-shape regularity~\eqref{eq:shaperegular}
 yields $h_T \simeq \diam(T)$.

For the nodes $\NN_\star$, we introduce the partition $\NN_\star=\NN_\star^\Gamma\cup\NN_\star^\Omega$ into all boundary nodes $\NN_\star^\Gamma := \NN_\star\cap\Gamma$ and all interior nodes $\NN_\star^\Omega := \NN_\star \backslash \NN_\star^\Gamma$. 

For the facets $\EE_\star$, we introduce the partition $\EE_\star=\EE_\star^\Gamma\cup\EE_\star^\Omega$ into all boundary facets $\EE_\star^\Gamma:=\set{\facet\in\EE_\star}{\facet\subset\Gamma}$ and all interior facets $\EE_\star^\Omega := \EE_\star \backslash \EE_\star^\Gamma$. 
Finally, for an element $T\in\TT_\star$, we denote by $\EEt:=\set{\facet\in\EE_\star}{\facet\subset \partial T}\subseteq\EE_\star$ the set of all facets of $T$.

\begin{figure}[t]
  \centering
  \psfrag{v1}[c][c]{\small $V_{1}$}\psfrag{v2}[c][c]{\small $V_{2}$}\psfrag{v3}[c][c]{\small $V_{3}$}
  \psfrag{v4}[c][c]{\small $V_{4}$}\psfrag{v5}[c][c]{\small $V_{5}$}\psfrag{v6}[c][c]{\small $V_{6}$}
  \psfrag{v7}[c][c]{\small $V_{7}$}
  \psfrag{v}[c][c]{\small $V$}
  \subfigure[Dual mesh $\TT_\star^*$.]{\label{fig:primaldual1}
	\includegraphics[width=0.4\textwidth]{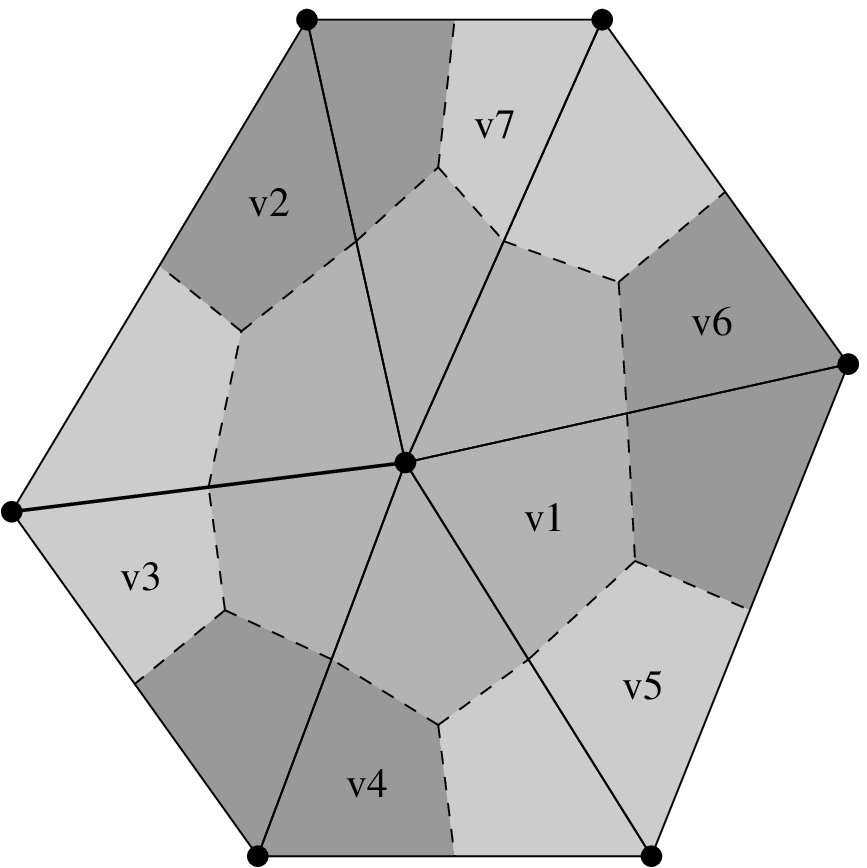}}
  \hspace*{0.1\textwidth}
  \psfrag{V}[c][c]{\small$V$}
  \subfigure[Edges of $\EE_{V,\star}$.]
  {\label{fig:edgesV}\includegraphics[width=0.4\textwidth]{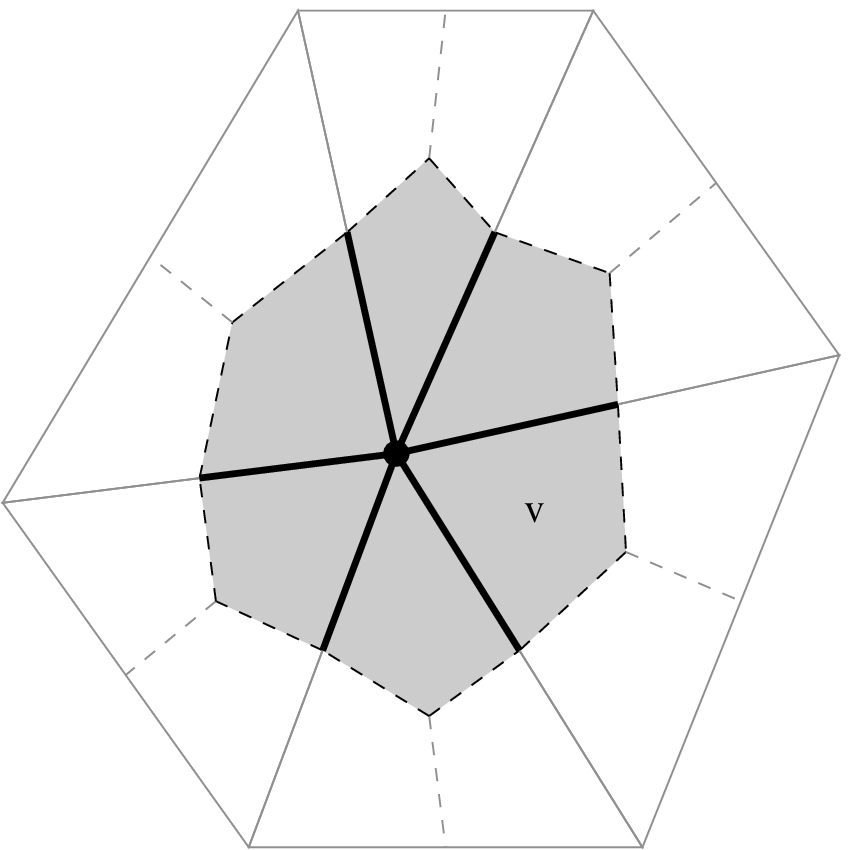}}
  \caption{Local construction of the dual mesh $\TT^*_\star$ from the 
  primal mesh $\TT_\star$ in 2D: 
  The dashed lines are the boundaries of the induced control volumes $V_i\in\TT^*_\star$, 
  which are associated with the nodes $a_i\in\NN_\star$ of $\TT_\star$ (left).
	For $V\in\TT_\star^*$, the set $\EE_{V,\star}$ consists of the bold lines which are parts of 
	edges in $\EE_\star$ (right).}
  \label{fig:primaldual}
\end{figure}
\begin{figure}[t]
  \centering
  \psfrag{T}[c][c]{\small $T$}
  \psfrag{cT}[c][c]{\small $c_T$}\psfrag{ai}[c][c]{\small $a_{i}$}
  \psfrag{Q1}[c][c]{\small $Q_{1}$}\psfrag{Q2}[c][c]{\small $Q_{2}$}
  \psfrag{Q3}[c][c]{\small $Q_{3}$}
  \psfrag{cE1}[c][c]{\small $c_{\facet_1}$}\psfrag{cE2}[c][c]{\small $c_{\facet_2}$}
  \psfrag{cE3}[c][c]{\small $c_{\facet_3}$}
  \psfrag{m1}[c][c]{\small $m_1$}\psfrag{m2}[c][c]{\small $m_2$}
  \psfrag{m3}[c][c]{\small $m_3$}
   \subfigure[Dual mesh $\TT^*_\star$.]{\label{fig:dual3dconst}
  \includegraphics[width=0.4\textwidth]{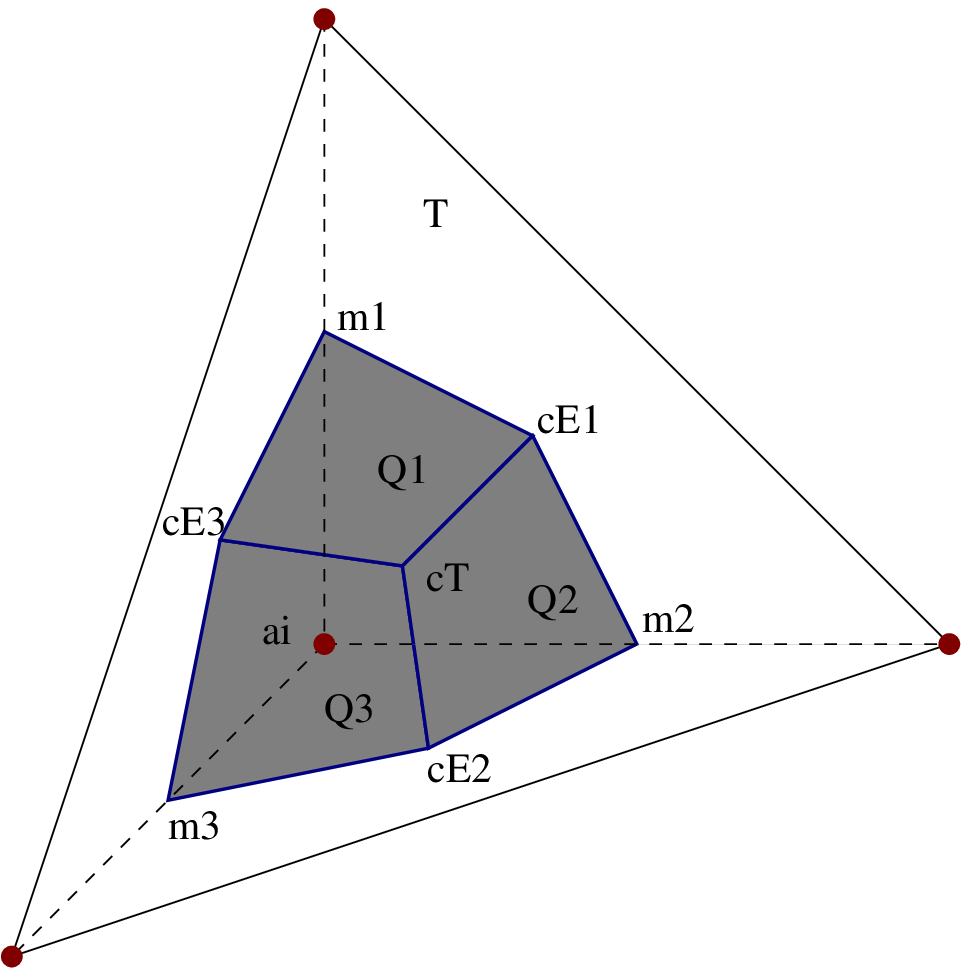}}
  \hspace*{0.1\textwidth}
  \psfrag{z1}[c][c]{\small $\zeta_1$}\psfrag{z2}[c][c]{\small $\zeta_2$}
  \psfrag{z3}[c][c]{\small $\zeta_3$}
   \subfigure[Some faces $\zeta_j$ of $\EE_{V_i,\star}$]{\label{fig:dual3dEEV}
  \includegraphics[width=0.4\textwidth]{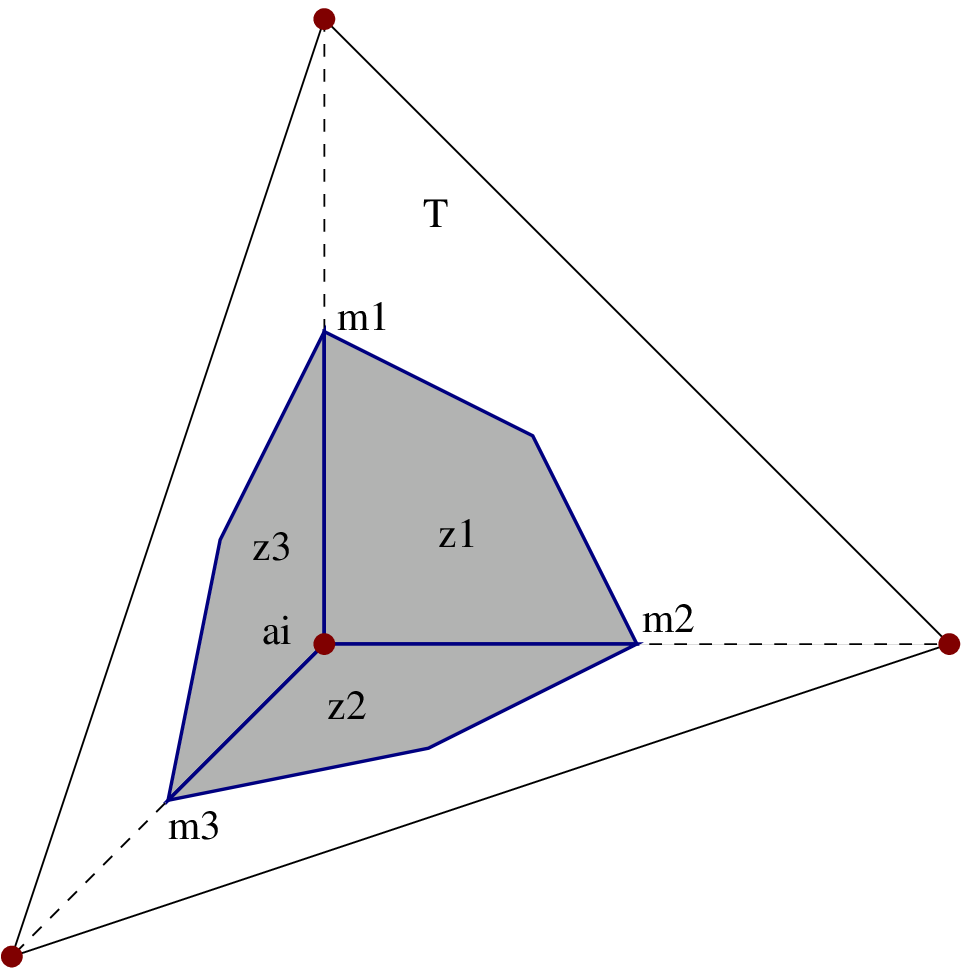}}
  \caption{Local construction of the dual mesh $\TT^*_\star$
  from the 
  primal mesh $\TT_\star$ in 3D:
  For a node $a_i\in\NN_\star$ of $T$, the center of gravity 
  $c_T$ of $T$ is connected  with the centers of gravity $c_{F_j}$ of the three adjacent faces $\facet\in\EEt$. Moreover, these centers are connected to the midpoints $m_k$ of the three edges which meet in $a_i$. Together with the edges
  from these midpoints to $a_i$, we get the cuboid $V_i\cap T\not=\emptyset$ (left).
  The three dark-gray faces $Q_1$, $Q_2$, and $Q_3$ are part of
  the boundary $\partial V_i$ of the box $V_i$ (left).
  The light-gray faces $\zeta_1$, $\zeta_2$, and $\zeta_3$ belong to the set $\EE_{V_i,\star}$ and are part of faces in $\EE_\star$ (right).
  }
  \label{fig:dual3d}
\end{figure}
\subsection{Dual mesh}
\label{subsec:dualmesh}
In contrast to standard FEM, our FVM discretization additionally needs the so-called 
dual mesh $\TT^*_\star$ 
which is built from the conforming triangulation $\TT_\star$.
In 2D, connecting the center of gravity of an element $T\in\TT_\star$
with the (edge) midpoint of $\facet\in\EEt$, we obtain $\TT^*_\star$ 
whose boxes (elements)
$V\in\TT^*_\star$ are non-degenerate closed polygons; 
see Figure~\ref{fig:primaldual}\subref{fig:primaldual1}.
In 3D, we connect the center of gravity of an element $T\in\TT_\star$
with the centers of gravity of the four faces $\facet\in\EEt$. Furthermore, each center of 
gravity of a face $\facet\in\EEt$ is connected by straight lines to the midpoints of the 
edges of the face $\facet$. Figure~\ref{fig:dual3d}\subref{fig:dual3dconst}
shows the contribution of some element $T\in\TT_\star$ with node $a_i$ to
the box $V_i\in\TT^*_\star$.

Note that there is a unique correspondence between the nodes $a_i\in\NN_\star$ of the primal mesh $\TT_\star$ and the boxes $V_i\in\TT^*_\star$ of the dual mesh, namely $V_i\cap\NN_\star = \{a_i\}$.
Furthermore, we define $\EE_{V,\star}:=\set{\facet\cap V}{\facet\in\EE_\star}$ for all $V\in\TT^*_\star$; see
Figure~\ref{fig:primaldual}\subref{fig:edgesV} for 2D.

For 3D, Figure~\ref{fig:dual3d}\subref{fig:dual3dEEV} shows three faces $\zeta_1$, $\zeta_2$,
and $\zeta_3$ of $\EE_{V_i,\star}$, $V_i\in\TT^*_\star$.
Note that 
\begin{align}
\label{eq:meshrelation}
\bigcup_{T\in\TT_\star}T = \overline\Omega = \bigcup_{V\in\TT_\star^*}V
\quad\text{and}\quad
\bigcup_{\facet\in\EE_\star}\facet=\bigcup_{V\in\TT^*_\star}\bigcup_{\facet\in\EE_{V,\star}}\facet.
\end{align}

\subsection{Vertex-centered finite volume method (FVM)}
\label{section:fvm}%
Given the conforming triangulation $\TT_\star$ and the corresponding dual mesh $\TT^*_\star$, we define the space of all $\TT_\star$-piecewise affine and globally continuous functions
\begin{align*}
\SS^1(\TT_\star):=\set{v\in\CC(\Omega)}{v|_T \text{ affine for all } T\in\TT_\star}\subset H^1(\Omega)
\end{align*}
as well as the space of all $\TT_\star^*$-piecewise constant functions
\begin{align*}
\PP^0(\TT^*_\star):=\set{v\in L^2(\Omega)}{v|_V \text{ constant for all } V\in\TT^*_\star}.
\end{align*}
For the FVM discretization, we consider the subspaces which respect the homogeneous Dirichlet conditions of~\eqref{eq:model}, i.e.,
\begin{align*}
 \SS^1_0(\TT_\star) := \set{v\in\SS^1(\TT_\star)}{v|_\Gamma=0}\subset H^1_0(\Omega)
 \text{ and }
 \PP^0_0(\TT_\star^*) := \set{v\in\PP^0(\TT_\star^*)}{v|_\Gamma=0}.
\end{align*}%
The formal idea of the FVM reads as follows:
If we integrate the strong form~\eqref{eq:model} over each dual element $V\in\TT^*_\star$ and
apply the divergence theorem, we get a balance equation for the model problem. The FVM approximates $u\in H^1_0(\Omega)$ by some conforming approximation $u_\star\in\SS^1_0(\TT_\star)$ of the balance equation. 
With the aid of test functions in $\PP^0_0(\TT_\star^*)$, we formalize this with the bilinear form
\begin{align}
  \label{eq:fvembilinear}
  \begin{split}
  \AA_\star(v_\star,v_\star^*):=-\sum_{a_i\in\NN_\star^\Omega}v_\star^*|_{V_i}
  \int_{\partial V_i}\A \nabla v_\star\cdot \nf\,ds
  \quad\text{for all }v_\star\in\SS^1_0(\TT_\star)
  \text{ and }v_\star^*\in\PP^0_0(\TT_\star^*).
  \end{split}
\end{align}
The right-hand side reads
\begin{align*}
  \sum_{a_i\in\NN_\star^\Omega}v_\star^*|_{V_i}
  \int_{V_i} f\,dx
  = \product{f}{v_\star^*}_\Omega
  \quad\text{for all }v_\star^*\in\PP^0_0(\TT_\star^*).
\end{align*}
Throughout, if $\nf$ appears in a boundary integral, it denotes the unit normal vector to the boundary pointing outward the respective domain.
Now, the FVM discretization reads: Find $u_\star\in\SS^1_0(\TT_\star)$ such that
\begin{align}
\label{eq:disc_systemfvem}
  \AA_\star(u_\star,v_\star^*)&= \product{f}{v_\star^*}_\Omega\quad \text{for all } v_\star^*\in\PP^0_0(\TT_\star^*).
\end{align}
It is well-known that there exists a constant $H>0$ such that~\eqref{eq:disc_systemfvem} admits a unique solution $u_\star\in\SS^1_0(\TT_\star)$ provided that $\TT_\star$ is sufficiently fine, i.e., $\norm{h_\star}{L^\infty(\Omega)}\le H$;
see Lemma~\ref{lem:bilinearcomp} below.
%
The convergence of the FVM is usually proved under certain regularity assumptions, e.g., $u\in H^1_0(\Omega)\cap H^{1+\eps}(\Omega)$ for some $\eps>0$; see, e.g.,~\cite[Theorem~3.3.]{Ewing:2002-1}. As a side result of our analysis, Theorem~\ref{corollary:cea} below proves convergence of the total error (i.e., energy error plus data oscillations) without any regularity assumptions.

\subsection{Weighted-residual error estimator}
\label{section:estimator}%
With $\div_\star$ denoting the $\TT_\star$-piecewise divergence operator,
we define the volume residual by
\begin{align}
  \label{eq:residuum}
  R_\star(v_\star)|_T=
	(f+\div_\star\A \nabla v_\star)|_T \quad\text{for all } T\in\TT_\star
	\text{ and all }v_\star\in\SS^1_0(\TT_\star).
\end{align}
Throughout, we abbreviate $\div_\star\A \nabla v_\star:=\div_\star(\A \nabla v_\star)$ to ease the readability.
Let $\jump\cdot$ denote the normal jump across an interior facet $\facet=T\cap T'\in\EE_\star^\Omega$, i.e., $\jump{\boldsymbol{g}}|_{F} =\boldsymbol{g}|_T\cdot\nf_T+  \boldsymbol{g}|_{T'}\cdot\nf_{T'}$, where, e.g., $\boldsymbol{g}|_T$ denotes the trace of $\boldsymbol{g}$ from $T$ onto $\facet$ and $\nf_T$ is the outer normal of $T$ on $\facet$.
Then, we define the facet residual or normal jump by
\begin{align}
  \label{eq:edgeJ}
  J_\star(v_\star)|_{F}=\jump{\A\nabla v_\star}|_{F} 
  \quad \text{ for all }\facet\in\EE_\star^\Omega
  	\text{ and all }v_\star\in\SS^1_0(\TT_\star).
\end{align}
For all $v_\star\in\SS^1_0(\TT_\star)$, we define the weighted-residual error estimator as for the FEM
\begin{align}\label{eq:eta}
 \eta_\star(v_\star)^2 = \eta_\star(\TT_\star,v_\star)^2
 \quad\text{with}\quad
 \eta_\star(\UU_\star,v_\star)^2 = \sum_{T\in\UU_\star}\eta_\star(T,v_\star)^2
 \quad\text{for all }\UU_\star\subseteq\TT_\star,
\end{align}
where 
\begin{align}\label{eq:etaT}
 \eta_\star(T,v_\star)^2 = h_T^2\,\norm{f+\div_\star\A\nabla v_\star}{L^2(T)}^2
 + h_T\,\norm{\jump{\A\nabla v_\star}}{L^2(\partial T\backslash\Gamma)}^2;
\end{align}
cf., e.g.,~\cite{ao00,verfuerth}. For $v_\star=u_\star$ being the discrete FVM solution, we abbreviate the notation and omit this argument, e.g., $\eta_\star := \eta_\star(u_\star)$ and $\eta_\star(T) := \eta(T,u_\star)$.

Let $\Pi_\star$ denote the elementwise or facetwise integral mean operator, i.e.,
\begin{align}
 (\Pi_\star v)|_\tau = \frac{1}{|\tau|}\int_\tau v\,dx
 \quad\text{for all }\tau\in\TT_\star\cup\EE_\star
 \text{ and all }v\in L^2(\tau).
\end{align}
Recall that $\Pi_\star$ is the elementwise $L^2$-orthogonal projection onto the constants, i.e.,
\begin{align}\label{eq:L2proj}
 \norm{v-\Pi_\star v}{L^2(\tau)} = \min_{c\in\R}\norm{v-c}{L^2(\tau)}
 \le \norm{v}{L^2(\tau)}
 \quad\text{for all }\tau\in\TT_\star\cup\EE_\star
 \text{ and all }v\in L^2(\tau).
\end{align}
With $\Pi_\star$, we define the data oscillations
\begin{align*}
 \osc_\star(v_\star)^2 = \osc_\star(\TT_\star,v_\star)^2
 \text{ with }
 \osc_\star(\UU_\star,v_\star)^2 = \sum_{T\in\UU_\star}\osc_\star(T,v_\star)^2
 \text{ for all }\UU_\star\subseteq\TT_\star,
\end{align*}
where
\begin{align}\label{eq:oscT}
 \osc_\star(T,v_\star)^2 = h_T^2\,\norm{(1-\Pi_\star)(f+\div_\star\A\nabla v_\star)}{L^2(T)}^2
 + h_T\,\norm{(1-\Pi_\star)\jump{\A\nabla v_\star}}{L^2(\partial T\backslash\Gamma)}^2.
\end{align}
Again, we abbreviate the notation for $v_\star = u_\star$ being the FVM solution, e.g., $\osc_\star := \osc_\star(u_\star)$ and $\osc_\star(T) := \osc_\star(T,u_\star)$.
Moreover, we stress the elementwise estimate 
\begin{align}\label{eq:osc-eta}
 \osc_\star(T,v_\star)\le\eta_\star(T,v_\star)
 \quad\text{for all }T\in\TT_\star
 \text{ and all }v_\star\in\SS^1_0(\TT_\star)
\end{align}%
which immediately follows from~\eqref{eq:L2proj}.
The following result is proved in~\cite[Theorem~2.4]{xzz06} and~\cite[Theorem~2.6]{xzz06}; see also~\cite[Theorem~3.1]{Carstensen:2005-1} and \cite[Theorem~3.3]{Carstensen:2005-1}.

\begin{proposition}[reliability and efficiency]
The estimator $\eta_\star$ satisfies reliability
\begin{align}\label{eq:reliable}
 \enorm{u-u_\star}^2 \le \c{rel}\,\eta_\star^2
\end{align}
as well as efficiency
\begin{align}\label{eq:efficient}
 \c{eff}^{-1}\,\eta_\star^2 \le \enorm{u-u_\star}^2 + \osc_\star^2.
\end{align}
The constants $\setc{rel},\setc{eff}>0$ depend only on $\sigma$-shape regularity  of $\TT_\star$ and on the assumptions~\eqref{eq:A}--\eqref{eq3:A} on $\A$.
\qed
\end{proposition}

The first contribution of the present work is the following C\'ea-type quasi-optimality of FVM with respect to the total error (i.e., sum of energy error plus data oscillations). 
In particular, this implies 
that the total errors of FVM and FEM are equivalent, see~\eqref{eq:intro4}. The proof of the theorem is given in Section~\ref{section:cea}.

\begin{theorem}\label{theorem:cea}
There exists $H>0$ such that the following statement is valid provided that $\TT_\star$ is sufficiently fine, i.e., $\norm{h_\star}{L^\infty(\Omega)}\le H$: There is a constant $\Ctot>0$ such that
\begin{align}\label{eq:totalerror}
 \Ctot^{-1}\,\eta_\star 
 \le \min_{v_\star\in\SS^1_0(\TT_\star)} \big(\enorm{u-v_\star} + \osc_\star(v_\star)\big) 
 \le \enorm{u-u_\star} + \osc_\star
 \le \Ctot\,\eta_\star.
\end{align}
Moreover, if $u_\star^{\rm FEM}\in\SS^1_0(\TT_\star)$ denotes the FEM solution of
\begin{align}\label{eq:fem}
 \AA(u_\star^{\rm FEM},v_\star) = (f,v_\star)_\Omega
 \quad\text{for all }v_\star\in\SS^1_0(\TT_\star),
\end{align}
it holds
\begin{align}\label{eq:equivalence}
 \Ctot^{-1}\,\big(\enorm{u-u_\star} + \osc_\star\big)
 \le \enorm{u-u_\star^{\rm FEM}} + \osc_\star(u_\star^{\rm FEM})
 \le  \Ctot\,\big(\enorm{u-u_\star} + \osc_\star\big).
\end{align}
The constant $\Ctot>0$ depends only on $\Omega$, $H$, the $\sigma$-shape regularity of $\TT_\star$, and on the  assumptions~\eqref{eq:A}--\eqref{eq3:A} on $\A$.
\end{theorem}

For the sake of completeness and as an application of Theorem~\ref{theorem:cea}, we note 
the following {\sl a~priori} estimate for the total error. Note that~\eqref{eq:cor:cea}
does not require any additional regularity assumption on $u$.
The proof is given in Section~\ref{section:cea2}.

\begin{theorem}\label{corollary:cea}
There exists $H>0$ such that the following statement is valid provided that $\TT_\star$ is sufficiently fine, i.e., $\norm{h_\star}{L^\infty(\Omega)}\le H$: There is a constant $C>0$ such that
\begin{align}\label{eq:cor:cea}
 C^{-1} \big(\enorm{u-u_\star} + \osc_\star\big) 
 \le  \norm{h_\star(1\!-\!\Pi_\star)f}{L^2(\Omega)} + \!\!\!\min_{v_\star\in\SS^1_0(\TT_\star)} 
 \!\!\big(\enorm{u-v_\star} + \norm{h_\star\nabla v_\star}{L^2(\Omega)}\big).
\end{align}
In particular, this proves convergence 
\begin{align}\label{eq3:cor:cea}
 \enorm{u-u_\star} + \osc_\star \to0
 \text{ as }
 \norm{h_\star}{L^\infty(\Omega)}\to0.
\end{align}
Provided that $u\in H^1_0(\Omega)\cap H^2(\Omega)$, there even holds
\begin{align}\label{eq2:cor:cea}
 \enorm{u-u_\star} + \osc_\star
 = \OO(\norm{h_\star}{L^\infty(\Omega)}).
\end{align}
The constant $C>0$ depends only on $\Omega$, $H$, the $\sigma$-shape regularity of $\TT_\star$, and on the  assumptions~\eqref{eq:A}--\eqref{eq3:A} on $\A$, and~\eqref{eq3:cor:cea}--\eqref{eq2:cor:cea} require uniform $\sigma$-shape regularity of the considered family $\TT_\star$.
%
\end{theorem}

\begin{figure}[t]
 \centering
 \psfrag{T0}{}
 \psfrag{T1}{}
 \psfrag{T2}{}
 \psfrag{T3}{}
 \psfrag{T4}{}
 \psfrag{T12}{}
 \psfrag{T34}{}
 \includegraphics[width=35mm]{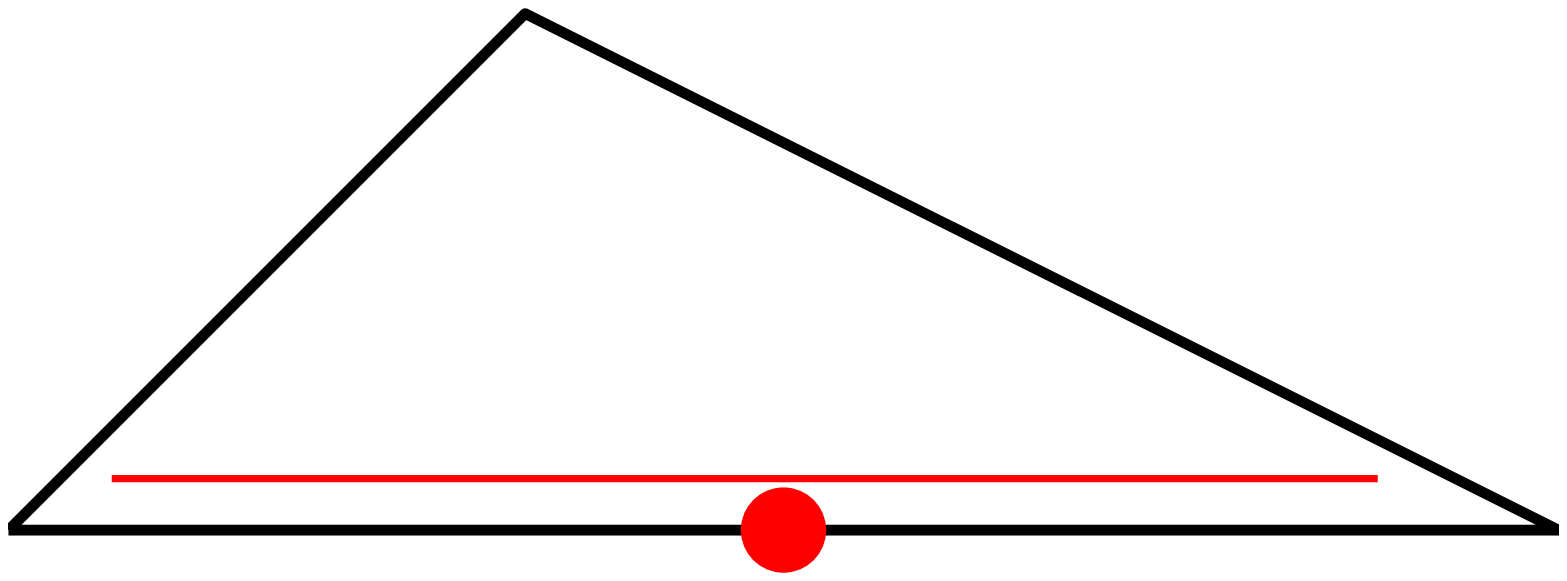} \quad
 \includegraphics[width=35mm]{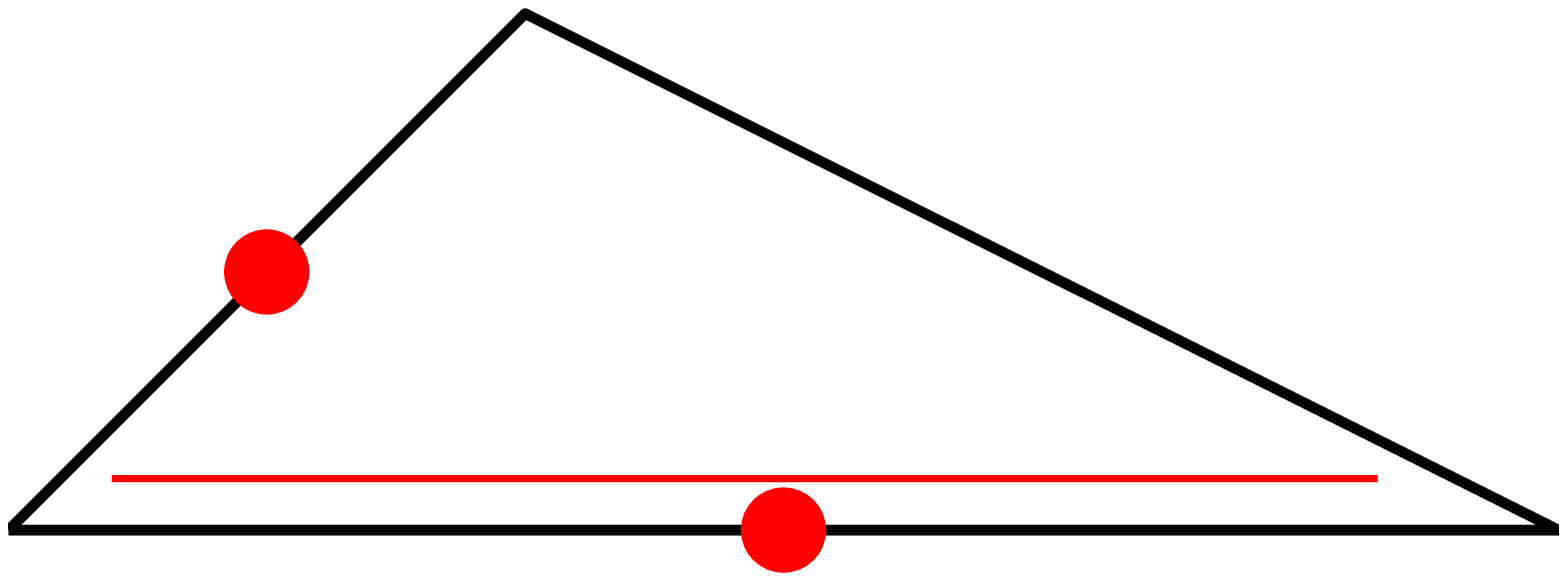} \quad
 \includegraphics[width=35mm]{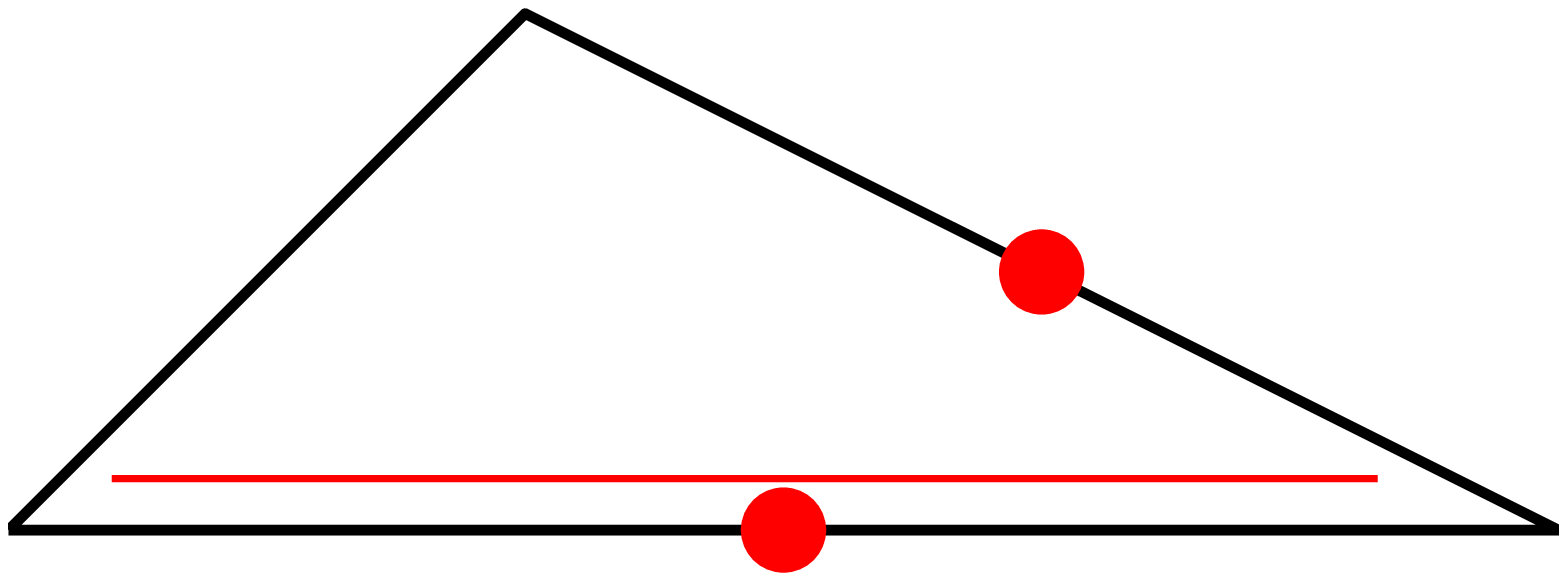} \quad
 \includegraphics[width=35mm]{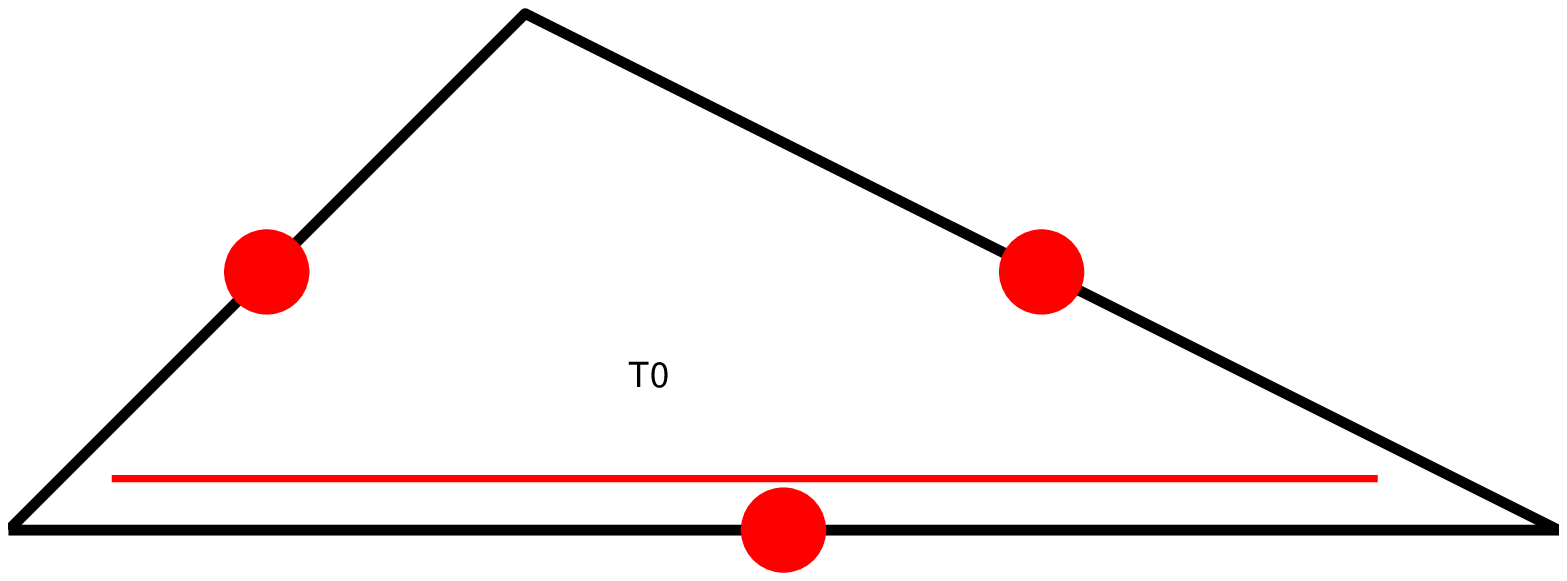} \\
 \includegraphics[width=35mm]{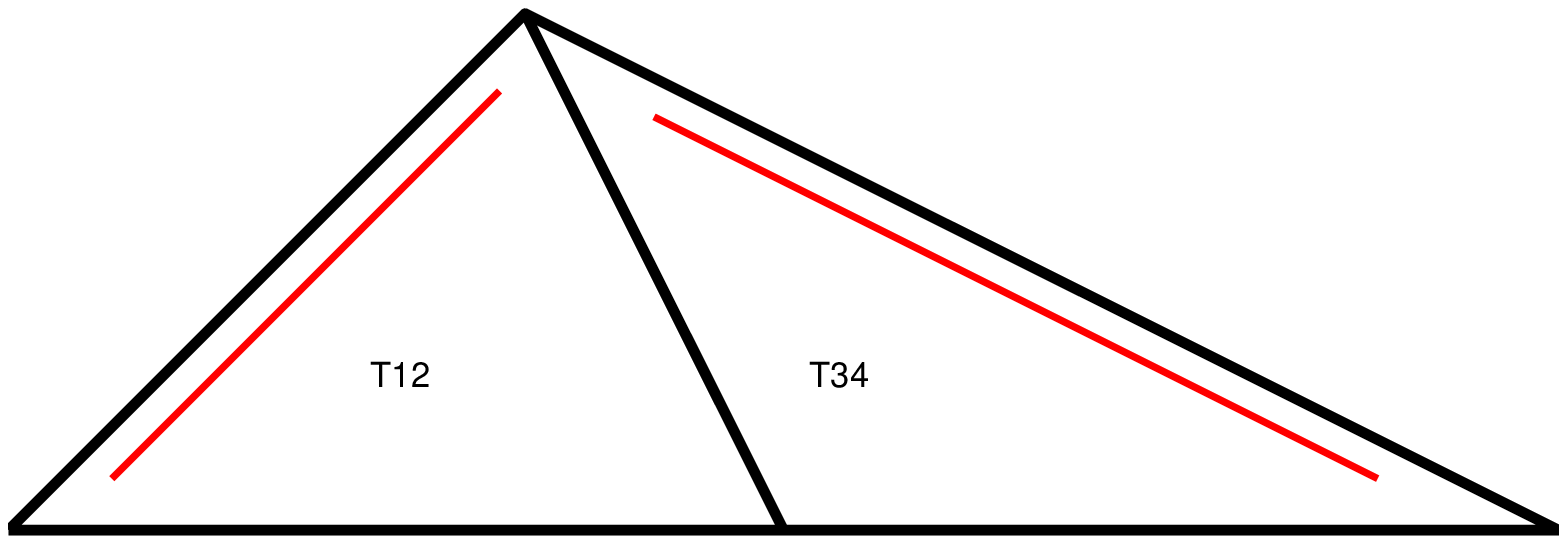} \quad
 \includegraphics[width=35mm]{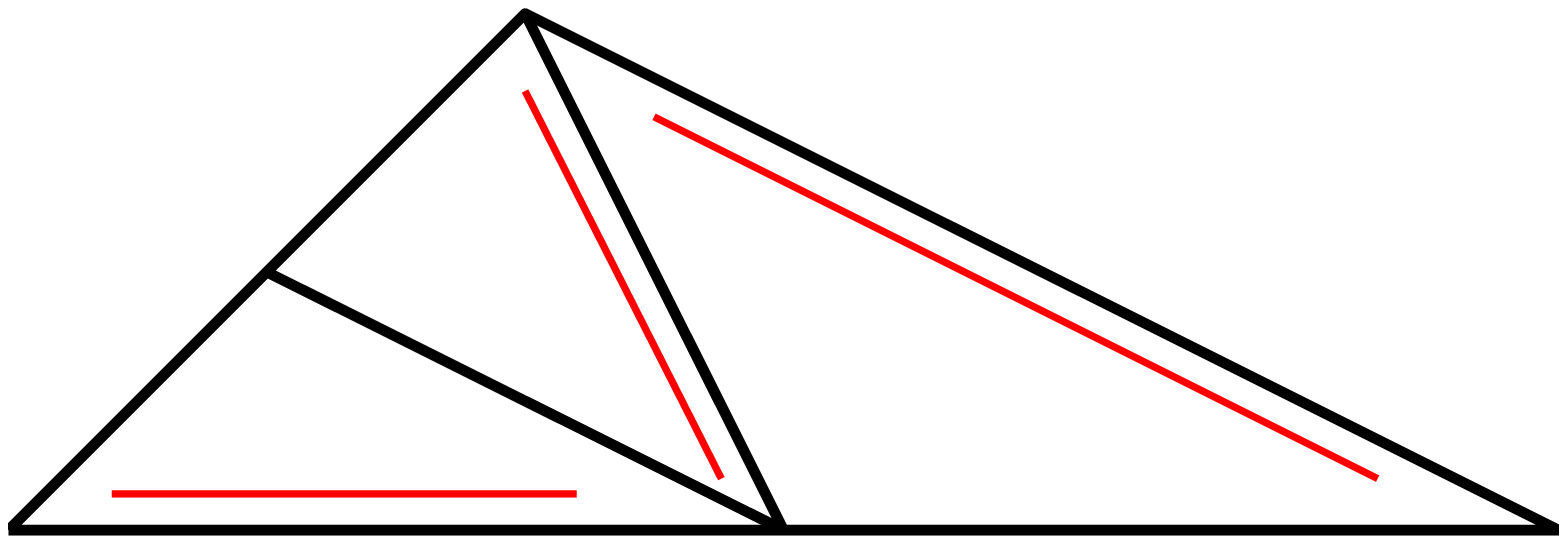}\quad
 \includegraphics[width=35mm]{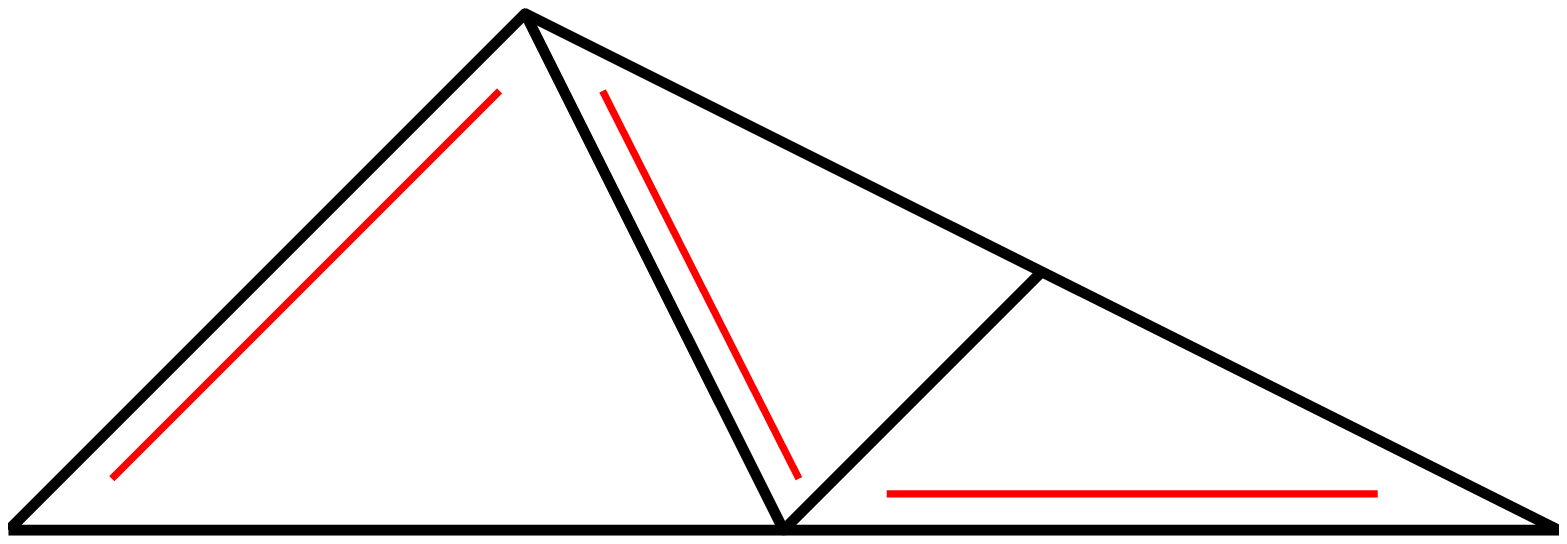}\quad
 \includegraphics[width=35mm]{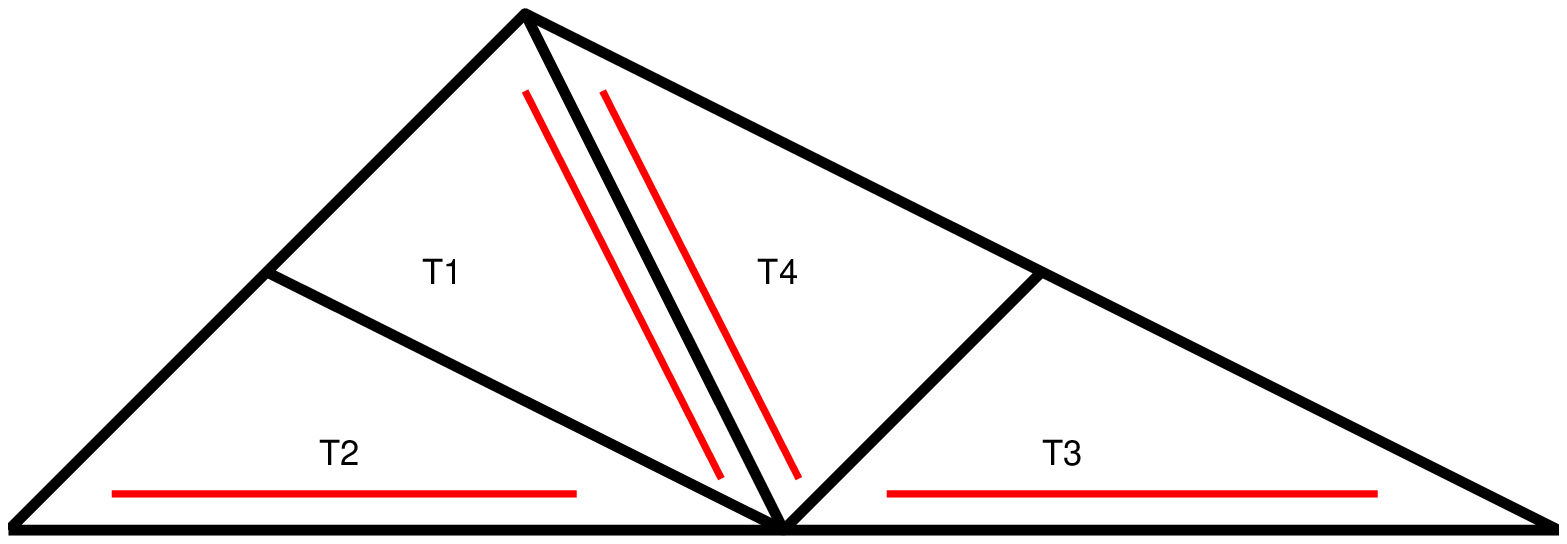}
 \caption{2D newest vertex bisection:
 For each triangle $T\in\TT$, there is one fixed \emph{reference edge},
 indicated by the double line (left, top). Refinement of $T$ is done by bisecting
 the reference edge, where its midpoint becomes a new node. The reference
 edges of the son triangles are opposite to this newest vertex (left, bottom).
 To avoid hanging nodes, one proceeds as follows:
 We assume that certain edges of $T$, but at least the reference edge,
 are marked for refinement (top).
 Using iterated newest vertex bisection, the element is then split into
 2, 3, or 4 son triangles (bottom).}
 \label{fig:nvb}
\end{figure}

\subsection{Adaptive algorithm \& main result}
\label{section:main}%
As for adaptive finite element methods~\cite{doerfler96,mns00,stevenson07,ckns,ffp14}, we consider the following adaptive algorithm which specifies the adaptive loop~\eqref{eq:semr}. Unlike the common algorithms in the context of adaptive FEM and BEM~\cite{axioms}, our algorithm does not only employ D\"orfler marking with respect to the error indicators $\eta_\ell(T)$, but also for the local contributions $\osc_\ell(T)$ of the data oscillations. This additional marking step is necessary to control the lack of Galerkin orthogonality~\eqref{eq:galerkin} and thus allows to prove (linear) convergence~\eqref{eq:mns:linear} of the adaptive algorithm.

For the mesh-refinement in step~(v) of Algorithm~\ref{algorithm:mns}, we employ newest vertex bisection (NVB); see, e.g.,~\cite{kpp13,stevenson08} for general dimension $d\ge2$ 
and Figure~\ref{fig:nvb} for an illustration for $d=2$.  
For a conforming triangulation $\TT$ and a set of marked elements $\MM\subseteq\TT$, let $\TT':=\refine(\TT,\MM)$ be the coarsest conforming triangulation generated by NVB such that all marked elements $T\in\MM$ have been refined, i.e., $\MM\subseteq\TT\backslash\TT'$.

\begin{algorithm}\label{algorithm:mns}
{\bfseries Input:} Let $0<\theta'\le\theta\le1$ and $\Cmark\ge1$ be given adaptivity parameters. Let
$\TT_0$ be a conforming triangulation of $\Omega$ which resolves possible discontinuities of $\A$ in the sense of~\eqref{eq3:A}. \\[1mm]
{\bfseries Then:} For $\ell=0,1,2,\dots$ iterate the following steps~{\rm(i)--(v)}:
\begin{itemize}
\item[\rm(i)] Solve~\eqref{eq:disc_systemfvem} to compute the discrete solution $u_\ell\in\SS^1_0(\TT_\ell)$ corresponding to $\TT_\ell$.
\item[\rm(ii)] Compute the refinement indicators $\eta_\ell(T)$ from~\eqref{eq:etaT} 
and the data oscillations from~\eqref{eq:oscT} for all $T\in\TT_\ell$.
\item[\rm(iii)] Construct a subset $\MM_\ell^\eta\subseteq\TT_\ell$ of up to the multiplicative factor $\Cmark$ minimal cardinality 
which satisfies the D\"orfler marking criterion 
\begin{align}\label{eq:doerfler:eta}
 \theta\,\eta_\ell^2 \le \eta_\ell(\MM_\ell^\eta)^2.
\end{align}
\item[\rm(iv)] Construct a subset $\MM_\ell\subseteq\TT_\ell$ of up to the multiplicative factor $\Cmark$ minimal cardinality which satisfies $\MM_\ell^\eta\subseteq\MM_\ell$ as well as the D\"orfler marking criterion 
\begin{align}\label{eq:doerfler:osc}
 \theta'\,\osc_\ell^2 \le \osc_\ell(\MM_\ell)^2.
\end{align}
\item[\rm(v)] Define $\TT_{\ell+1} := \refine(\TT_\ell,\MM_\ell)$.
\end{itemize}
\smallskip
{\bfseries Output:} Adaptively refined triangulations $\TT_\ell$, corresponding discrete solutions $u_\ell$, estimators $\eta_\ell$, and data oscillations $\osc_\ell$ for $\ell\ge0$.
\end{algorithm}

\begin{remark}
{\rm(i)} For $\Cmark=1$, the construction of the set $\MM_\ell^\eta$ in step~{\rm(iii)} of Algorithm~\ref{algorithm:mns} requires to sort the error indicators and thus results in logarithmic-linear complexity. Instead, for $\Cmark=2$, an approximate sorting based on binning allows to construct $\MM_\ell$ in linear complexity~\cite{stevenson07}.
The same applies for $\MM_\ell$ in step~{\rm(iv)} of Algorithm~\ref{algorithm:mns}.\\
{\rm(ii)} There exists a constant $H>0$ such that~\eqref{eq:disc_systemfvem} has a unique solution provided that $\norm{h_\ell}{L^\infty(\Omega)}\le H$; see Lemma~\ref{lem:bilinearcomp} below. Since NVB guarantees $\norm{h_\ell}{L^\infty(\Omega)} \le \norm{h_0}{L^\infty(\Omega)}$, it is sufficient to suppose that the initial triangulation $\TT_0$ is sufficiently fine.\\
{\rm(iii)} In step~{\rm(v)} of Algorithm~\ref{algorithm:mns}, one may use any variant of NVB which applies at most $n$ bisections per marked element, where $n\ge1$ is a fix constant.
\qed
\end{remark}

Next, we define certain nonlinear approximation classes, which are needed to prove optimal convergence behavior~\eqref{eq:intro2}.
To this end,
we write $\TT'\in\refine(\TT)$, if there exists some $n\in\N_0$, triangulations $\TT'_j$, and marked elements $\MM'_j\subseteq\TT'_j$ such that $\TT=\TT'_0$, $\TT'=\TT'_n$, and $\TT'_j = \refine(\TT'_{j-1},\MM'_{j-1})$ for all $j=1,\dots,n$.
Given $\TT_0$ from Algorithm~\ref{algorithm:mns}, 
we note that NVB ensures that all triangulations $\TT_\star\in\refine(\TT_0)$ are uniformly $\sigma$-shape regular~\eqref{eq:shaperegular}, where $\sigma$ depends only on $\TT_0$.

For $N>0$, we abbreviate $\T_N:=\set{\TT_\star\in\refine(\TT_0)}{\#\TT_\star-\#\TT_0\le N}$, where $\#\TT_\star$ denotes the number of elements in $\TT_\star$.
For all $s>0$, we define the approximability measure
\begin{align*}
 \norm{u}{\mathbb A_s} 
 := \sup_{N>0} \min_{\substack{\TT_\star\in\T_N}} (N+1)^s\eta_\star,
\end{align*}
where $\eta_\star$ denotes the weighted-residual error estimator~\eqref{eq:eta} associated with the optimal triangulation $\TT_\star$.
Note that $\norm{u}{\mathbb A_s} < \infty$ means that an algebraic decay $\eta_\star = \OO(N^{-s})$ is theoretically possible if for each $N>0$ the optimal triangulations $\TT_\star\in\T_N$ are chosen.

As a corollary of Theorem~\ref{theorem:cea}, we obtain that the corresponding approximation class (of all $u$ which satisfy $\norm{u}{\mathbb A_s}<\infty$) can equivalently be characterized by the so-called total error (i.e., energy error plus data oscillations) and hence coincides with the approximation classes from the FEM literature; see, e.g.,~\cite{stevenson07,ckns,ffp14}.

\begin{corollary}
There exists $H>0$ such that the following equivalence is valid if the initial triangulation $\TT_0$ is sufficiently fine, i.e., $\norm{h_0}{L^\infty(\Omega)}\le H$:
For all $s>0$, it holds
\begin{align*}
 \norm{u}{\mathbb A_s} < \infty
 \quad\Longleftrightarrow\quad
 \sup_{N>0}\min_{\substack{\TT_\star\in\T_N}}\inf_{v_\star\in\SS^1_0(\TT_\star)} 
 (N+1)^s\big(\enorm{u-v_\star} + \osc_\star(v_\star)\big)< \infty.
\end{align*}
\end{corollary}%
\begin{proof}
Note that all triangulations $\TT_\star\in\refine(\TT_0)$ are uniformly $\sigma$-shape regular and satisfy $\norm{h_\star}{L^\infty(\Omega)}\le \norm{h_0}{L^\infty(\Omega)}$. Therefore, the claim follows from~\eqref{eq:totalerror}.
\end{proof}

Besides Theorem~\ref{theorem:cea}, the following theorem is the main result~\eqref{eq:intro1}--\eqref{eq:intro2} of our work. Unlike~\cite{stevenson07,ckns}, we follow~\cite{axioms,ffp14} and formulate the result with respect to the error estimator as this is the natural goal quantity of Algorithm~\ref{algorithm:mns}. In view of~\eqref{eq:totalerror}, the theorem can equivalently be formulated with respect to the total error. Its proof is given in Section~\ref{section:theorem} below.

\begin{theorem}\label{theorem:mns}
There is a constant $H>0$ such that the following statements {\rm(i)--(ii)} are valid provided that the initial triangulation $\TT_0$ is sufficiently fine, i.e., $\norm{h_0}{L^\infty(\Omega)}\le H$:
\begin{itemize}
\item[\rm(i)] For all $0<\theta'\le\theta\le1$, there exist constants $\Clin>0$ and $0<\qlin<1$ such that the adaptive 
Algorithm~\ref{algorithm:mns} guarantees linear convergence of the estimator in the sense of 
\begin{align}\label{eq:mns:linear}
 \eta_{\ell+n}^2\le\Clin\qlin^n\,\eta_\ell^2
 \quad\text{for all }\ell,n\in\N_0.
\end{align}
\item[\rm(ii)] There  exists a bound $0<\theta_{\rm opt}\le1$ such that for all $0<\theta<\theta_{\rm opt}$, the following holds:
Provided that there is a constant $\Cmns\ge1$ such that $\#\MM_\ell\le\Cmns\#\MM_\ell^\eta$ for all $\ell\in\N_0$,
there is a constant $\Copt>0$ such that for all $s>0$, it holds
\begin{align}\label{eq:mns:optimal}
 \norm{u}{\mathbb A_s} < \infty
 \quad\Longleftrightarrow\quad
 \eta_\ell \le \frac{\Copt^{1+s}}{(1-\qlin^{1/s})^s}\,\norm{u}{\mathbb A_s}\,(\#\TT_\ell-\#\TT_0)^{-s},
\end{align}
i.e., the adaptive algorithm leads asymptotically to each possible algebraic decay $s>0$ of the error estimator.
\end{itemize}
The constant $\theta_{\rm opt}$ depends only on $\Omega$, $H$, and uniform $\sigma$-shape regularity of the triangulations $\TT_\star\in\refine(\TT_0)$, the constants $\Clin$ and $\qlin$ depend additionally on $\theta$ and $\theta'$, while the constant $\Copt$ depends also on the use of NVB and on $\Cmark$ and $\Cmns$.
\end{theorem}

\begin{remark}
{\rm(i)} 
The additional assumption in Theorem~\ref{theorem:mns} {\rm(ii)} assumes that marking~\eqref{eq:doerfler:osc} of the data oscillations is negligible with respect to the overall number of marked elements. We note that $\theta'>0$ can be chosen arbitrarily small so that, in practice,~\eqref{eq:doerfler:eta} already implies~\eqref{eq:doerfler:osc}.\\
{\rm(ii)} Instead of the additional marking step {\rm(iv)} in Algorithm~\ref{algorithm:mns}, one can also define $\MM_\ell:=\MM_\ell^\eta$ and monitor {\sl a~posteriori} if
\begin{align}
 \sup_{\ell_0\in\N_0}\inf_{\ell\ge\ell_0}\frac{\osc_\ell(\MM_\ell)^2}{\osc_\ell^2} =: \theta' > 0.
\end{align}
In this case, linear convergence~\eqref{eq:mns:linear} with optimal rates~\eqref{eq:mns:optimal} follows. However, for $\theta'=0$, even convergence remains mathematically open, so that we favor the present form of Algorithm~\ref{algorithm:mns} which guarantees~\eqref{eq:mns:linear}, while~\eqref{eq:mns:optimal} requires an additional assumption.
\qed
\end{remark}

\section{Proofs}
\label{section:proofs}%

\subsection{Axioms of adaptivity}
\label{section:axioms}%
In~\cite[Theorem~4.1]{axioms}, it is proved in a general framework 
that the following set of four axioms is sufficient (and partially even necessary) to guarantee linear convergence with optimal algebraic rates in the sense of Theorem~\ref{theorem:mns}. 
In particular, the model problem, the discretisation, and the estimator
enter only through the proof of these axioms.
Implicitly, we assume that given $\TT_k\in\refine(\TT_0)$, the corresponding FVM solution $u_k\in\SS^1_0(\TT_k)$ is well-defined. With this convention, the axioms read:
\begin{itemize}
\item[\bf(A1)] \textbf{stability on non-refined elements}:
There exists a constant $C>0$ such that for all 
$\TT_\diamond\in\refine(\TT_0)$ and all $\TT_\star\in\refine(\TT_\diamond)$, it holds
\begin{align*}
 |\eta_\star(\TT_\star\cap\TT_\diamond)-\eta_\diamond(\TT_\star\cap\TT_\diamond)|
 \le C\,\enorm{u_\star-u_\diamond}.
\end{align*}

\item[\bf(A2)] \textbf{reduction on refined elements}:
There exist constants $0<q<1$ and $C>0$ such that
for all $\TT_\diamond\in\refine(\TT_0)$ and all $\TT_\star\in\refine(\TT_\diamond)$, it holds
\begin{align*}
 \eta_\star(\TT_\star\backslash\TT_\diamond)^2
 \le q\,\eta_\diamond(\TT_\diamond\backslash\TT_\star)^2
 + C\,\enorm{u_\star-u_\diamond}^2.
\end{align*}

\item[\bf(A3)] \textbf{general quasi-orthogonality}:
There exists $C>0$ such that for all $\ell\in\N_0$, it holds
\begin{align*}
 \sum_{k=\ell}^\infty \enorm{u_{k+1}-u_k}^2
 \le C\,\eta_\ell^2.
\end{align*}

\item[\bf(A4)] \textbf{discrete reliability}:
There exists a constant $C>0$ such that for all $\TT_\diamond\in\refine(\TT_0)$ 
and all $\TT_\star\in\refine(\TT_\diamond)$, there exists some set 
$\RR_\diamond\subseteq\TT_\diamond$ with $\TT_\diamond\backslash\TT_\star\subseteq\RR_\diamond$ and
\begin{align*}
 \#\RR_\diamond \le C\,\#(\TT_\diamond\backslash\TT_\star)
 \quad\text{as well as}\quad
 \enorm{u_\star-u_\diamond} \le C\,\eta_\diamond(\RR_\diamond).
\end{align*}

\end{itemize}

The subsequent analysis proves that Algorithm~\ref{algorithm:mns} for our adaptive
FVM guarantees 
the validity of (A1)--(A4) if the initial triangulation $\TT_0$ is sufficiently fine. 

\subsection{Stability \& reduction of error estimator}
\label{section:A1-A2}%
The following lemma is stated without a proof, since the details are implicitly found in~\cite[Section~3.1]{ckns}. 
Moreover, (A1)--(A2) do not only hold for the FVM solutions $u_\star$ and $u_\diamond$ (or for FEM solutions in~\cite{ckns}), but for all $v_\star\in\SS^1_0(\TT_\star)$ and $v_\diamond\in\SS^1_0(\TT_\diamond)$.
The reader may simplify the proof of Lemma~\ref{lemma:osc} below, which provides slightly sharper estimates for the data oscillations.

\begin{lemma}\label{lemma:A1-A2}
The residual error estimator satisfies the following two properties~{\rm(A1')--(A2')}:
\begin{itemize}
\item[\bf(A1')]
There exists a constant $C>0$ such that for all $\TT_\diamond\in\refine(\TT_0)$, all $\TT_\star\in\refine(\TT_\diamond)$, and all $v_\star\in\SS^1_0(\TT_\star)$, $v_\diamond\in\SS^1_0(\TT_\diamond)$, it holds
\begin{align*}
 |\eta_\star(\TT_\star\cap\TT_\diamond,v_\star)-\eta_\diamond(\TT_\star\cap\TT_\diamond,v_\diamond)|
 \le C\,\Big(\sum_{T\in\patch_\star(\TT_\star\cap\TT_\diamond)}\enorm{v_\star-v_\diamond}^2_{T}\Big)^{1/2}.
\end{align*}
\item[\bf(A2')]
There exist constants $0<q<1$ and $C>0$ such that
for all $\TT_\diamond\in\refine(\TT_0)$, all $\TT_\star\in\refine(\TT_\diamond)$, and all $v_\star\in\SS^1_0(\TT_\star)$, $v_\diamond\in\SS^1_0(\TT_\diamond)$, it holds
\begin{align*}
 \eta_\star(\TT_\star\backslash\TT_\diamond,v_\star)^2
 \le q\,\eta_\diamond(\TT_\diamond\backslash\TT_\star,v_\diamond)^2
 + C\,\Big(\sum_{T\in\patch_\star(\TT_\star\backslash\TT_\diamond)}\enorm{v_\star-v_\diamond}_{T}^2\Big)^{1/2}.
\end{align*}
\end{itemize}
Here, $\patch_\star(\UU_\star):=\set{T\in\TT_\star}{\exists T'\in\UU_\star\quad T\cap T'\neq\emptyset}$ denotes the patch of $\UU_\star\subseteq\TT_\star$ in $\TT_\star$.
The constants $C$ and $q$ depend only on uniform $\sigma$-shape regularity of all $\TT_\star\in\refine(\TT_0)$ and the assumptions~\eqref{eq:A}--\eqref{eq3:A} on $\A$. In particular, this implies~{\rm(A1)--(A2)}.\qed
\end{lemma}%

\subsection{Stability \& reduction of data oscillations}
\label{section:B1B2}%
Our proof of linear convergence~\eqref{eq:mns:linear}
in Section~\ref{section:A3} requires to control the data oscillations which arise in some quasi-Galerkin orthogonality~\eqref{eq:galerkin}. This is done by means of two additional axioms which structurally follow~(A1')--(A2'), but have an additional factor $\norm{h_\star}{L^\infty(\Omega)}$ in the perturbation term.
Essentially, the following lemma is a sharper variant of the proofs in~\cite[Lemma~5.2]{xzz06}
and~\cite[Section~3.1]{ckns}:

\begin{lemma}\label{lemma:osc}
The data oscillations satisfy the following two properties~{\rm(B1')--(B2')}:
\begin{itemize}
\item[\bf(B1')]
There exists a constant $C>0$ such that for all $\TT_\diamond\in\refine(\TT_0)$, all $\TT_\star\in\refine(\TT_\diamond)$, and all $v_\star\in\SS^1_0(\TT_\star)$, $v_\diamond\in\SS^1_0(\TT_\diamond)$, it holds
\begin{align*}
 |\osc_\star(\TT_\star\cap\TT_\diamond,v_\star)-\osc_\diamond(\TT_\star\cap\TT_\diamond,v_\diamond)|
 \le C\,\Big(\sum_{T\in\patch_\star(\TT_\star\cap\TT_\diamond)}h_{T}^2\enorm{v_\star-v_\diamond}^2_{T}\Big)^{1/2}.
\end{align*}

\item[\bf(B2')]
There exist constants $0<q<1$ and $C>0$ such that
for all $\TT_\diamond\in\refine(\TT_0)$, all $\TT_\star\in\refine(\TT_\diamond)$, and all $v_\star\in\SS^1_0(\TT_\star)$, $v_\diamond\in\SS^1_0(\TT_\diamond)$, it holds
\begin{align*}
 \osc_\star(\TT_\star\backslash\TT_\diamond,v_\star)^2
 \le q\,\osc_\diamond(\TT_\diamond\backslash\TT_\star,v_\diamond)^2
 + C\,\Big(\sum_{T\in\patch_\star(\TT_\star\backslash\TT_\diamond)}h_{T}^2\enorm{v_\star-v_\diamond}_{T}^2\Big)^{1/2}
\end{align*}
\end{itemize}
Here, $\patch_\star(\UU_\star):=\set{T\in\TT_\star}{\exists T'\in\UU_\star\quad T\cap T'\neq\emptyset}$ denotes the patch of $\UU_\star\subseteq\TT_\star$ in $\TT_\star$.
The constants $C$ and $q$ depend only on uniform $\sigma$-shape regularity of the triangulations $\TT_\star\in\refine(\TT_0)$ and the assumptions~\eqref{eq:A}--\eqref{eq3:A} on $\A$.
\end{lemma}

\begin{proof}
{\bf Step~1.} For all $w_\star\in\SS^1_0(\TT_\star)$ and all $\UU_\star\subseteq\TT_\star$, it holds
\begin{align}\label{eq:osc:step1}
 \sum_{T\in\UU_\star}h_T^2\norm{(1-\Pi_\star)\,\div_\star \A\nabla w_\star}{L^2(T)}^2
 \le C\,\sum_{T\in\UU_\star}h_{T}^2\norm{\nabla w_\star}{L^2(T)}^2,
\end{align}
where $C>0$ depends only on $\max_{T_0\in\TT_0}\norm{\A}{W^{1,\infty}(T_0)}$:
For $T\in\UU_\star$, it holds
\begin{align*}
 h_T\norm{(1-\Pi_\star)\,\div_\star \A\nabla w_\star}{L^2(T)}
 \le h_T\norm{\,\div_\star \A\nabla w_\star}{L^2(T)}
 \lesssim h_T\norm{\A}{W^{1,\infty}(T)} \norm{\nabla w_\star}{L^2(T)},
\end{align*}
since $\nabla w_\star$ is constant on $T$. All elements $T\in\TT_\star$ satisfy $T\subseteq T_0$ for some $T_0\in\TT_0$, i.e., 
$\norm{\A}{W^{1,\infty}(T)}\leq \max_{T_0\in\TT_0}\norm{\A}{W^{1,\infty}(T_0)}$.
Summing this estimate over all $T\in\UU_\star$, we thus obtain~\eqref{eq:osc:step1}.

{\bf Step~2.} For all $w_\star\in\SS^1_0(\TT_\star)$ and all $\UU_\star\subseteq\TT_\star$, it holds
\begin{align}\label{eq:osc:step2}
 \sum_{T\in\UU_\star}h_T\norm{(1-\Pi_\star)J_\star(w_\star) }{L^2(\partial T\backslash\Gamma)}^2
 \le C\,\sum_{T\in\patch_\star(\UU_\star)}h_{T}^2\norm{\nabla w_\star}{L^2(T)}^2,
\end{align}
where $C>0$ depends only on $\max_{T_0\in\TT_0}\norm{\A}{W^{1,\infty}(T_0)}$ and $\sigma$-shape regularity of $\TT_\star$: Let $T\in\UU_\star$ and
$\facet\in\EE_T\cap\EE_\star^\Omega$ be a facet of $T$ which is not on the boundary $\Gamma$. Let $\overline\A=(1/|T|)\int_T\A\,dx$, i.e., piecewise integral means of the entries in $\A$. 
Note that $\nabla w_\star$ as well as the outer normal vector $\nf_T$ of $T$ are constant on $\facet$. The uniform continuity of $\A|_T$, the Poincar\'e inequality in $W^{1,\infty}(T)$, and a scaling argument show
\begin{align*}
 h_T^{1/2}\norm{(1-\Pi_\star)\big(\A\nabla w_\star\cdot \nf_T\big)}{L^2(\facet)}
 &\le h_T^{1/2}\norm{(\A-\overline{\A})\nabla w_\star\cdot\nf_T}{L^2(\facet)}
 \\&
 \le \norm{\A-\overline{\A}}{L^\infty(T)}\,h_T^{1/2}\,\norm{\nabla w_\star}{L^2(\facet)}
 \\&
 \lesssim h_T\,\norm{\A}{W^{1,\infty}(T)}\,\norm{\nabla w_\star}{L^2(T)}.
\end{align*}
Let $T'\in\TT_\star$ be the unique element with $\facet=T\cap T'$. 
Then, the definition of the facet residual~\eqref{eq:edgeJ} on $\facet$ leads to
\begin{align*}
\begin{split}
 h_T\norm{(1-\Pi_\star)J_\star(w_\star)}{L^2(\facet)}^2
 \lesssim
 \,h_T^2\,\norm{\nabla w_\star}{L^2(T)}^2 + h_{T'}^2\,\norm{\nabla w_\star}{L^2(T')}^2.
\end{split}
\end{align*}
Summing this over all interior facets $\facet\in\EE_T\cap\EE_\star^\Omega$ of the elements $T\in\UU_\star$, we obtain~\eqref{eq:osc:step2}.

{\bf Step~3.}
For all $v_\star,v_\star'\in\SS^1_0(\TT_\star)$ and all $\UU_\star\subseteq\TT_\star$, it holds
\begin{align}\label{eq:osc:step3}
 |\osc_\star(\UU_\star,v_\star)-\osc_\star(\UU_\star,v_\star')|
 \le C\,\Big(\sum_{T\in\patch_\star(\UU_\star)}h_{T}^2\norm{\nabla(v_\star-v_\star')}{L^2(T)}^2\Big)^{1/2},
\end{align}
where $C>0$ depends only on $\max_{T_0\in\TT_0}\norm{\A}{W^{1,\infty}(T_0)}$ and $\sigma$-shape regularity of $\TT_\star$: 
The inverse triangle inequality for square-summable sequences in the Banach space $\ell_2$ gives
\begin{align*}
 |\osc_\star(\UU_\star,v_\star)-\osc_\star(\UU_\star,v_\star')|
 \le &\Big(\sum_{T\in\UU_\star}h_T^2\,\norm{(1-\Pi_\star)\,\div_\star\A\nabla(v_\star-v_\star') }{L^2(T)}^2\\
 &\qquad+ \sum_{T\in\UU_\star}h_T\,\norm{(1-\Pi_\star)J_\star(v_\star-v_\star')}{L^2(\partial T\backslash\Gamma)}^2\Big)^{1/2}.
\end{align*}
Using~\eqref{eq:osc:step1}--\eqref{eq:osc:step2} for $w_\star:=v_\star-v_\star'$, we obtain~\eqref{eq:osc:step3}.%

{\bf Step~4: Proof of (B1').}
For $v_\star\in\SS^1_0(\TT_\star)$ and $v_\diamond\in\SS^1_0(\TT_\diamond)$, apply~\eqref{eq:osc:step3} with $v_\star':=v_\diamond$ and $\UU_\star:=\TT_\star\cap\TT_\diamond$. Note that $\osc_\diamond(\TT_\star\cap\TT_\diamond,v_\diamond) = \osc_\star(\TT_\star\cap\TT_\diamond,v_\diamond)$.
With $\norm{\nabla \cdot}{L^2(T)}\simeq\enorm{\cdot}_{T}$ this yields
\begin{align*}
 |\osc_\star(\TT_\star\cap\TT_\diamond,v_\star)-\osc_\diamond(\TT_\star\cap\TT_\diamond,v_\diamond)|
 &\lesssim 
\Big(\sum_{T\in\patch_\star(\TT_\star\cap\TT_\diamond)}
h_{T}^2\enorm{v_\star-v_\diamond}^2_{T}\Big)^{1/2}.
\end{align*}%
The hidden constant depends only on $\max_{T_0\in\TT_0}\norm{\A}{W^{1,\infty}(T_0)}$, on $\sigma$-shape regularity of $\TT_\star$, and on the assumptions~\eqref{eq:A}--\eqref{eq3:A} on $\A$. This concludes the proof of~(B1'). 

{\bf Step~5: Proof of (B2').}
For $v_\star\in\SS^1_0(\TT_\star)$ and $v_\diamond\in\SS^1_0(\TT_\diamond)$, we apply~\eqref{eq:osc:step3} with $v_\star':=v_\diamond$ and $\UU_\star:=\TT_\star\backslash\TT_\diamond$. 
With $\norm{\nabla \cdot}{L^2(T)}\simeq\enorm{\cdot}_{T}$ this shows
\begin{align*}
 \osc_\star(\TT_\star\backslash\TT_\diamond,v_\star)
 \le \osc_\star(\TT_\star\backslash\TT_\diamond,v_\diamond)
 + C\,\Big(\sum_{T\in\patch_\star(\TT_\star\backslash\TT_\diamond)}h_{T}^2\enorm{v_\star-v_\diamond}_{T}^2\Big)^{1/2}.
\end{align*}
For all $\delta>0$, the Young inequality $(a+b)^2 \le (1+\delta)\,a^2+(1+\delta^{-1})\,b^2$ proves
\begin{align*}
 \osc_\star(\TT_\star\backslash\TT_\diamond,v_\star)^2
 \le (1+\delta)\,\osc_\star(\TT_\star\backslash\TT_\diamond,v_\diamond)^2
 + (1+\delta^{-1})\,C^2\sum_{T\in\patch_\star(\TT_\star\backslash\TT_\diamond)}h_{T}^2\enorm{v_\star-v_\diamond}_{T}^2.
\end{align*}%
Let $T\in\TT_\diamond\backslash\TT_\star$. Let $\TT_\star|_T:=\set{T'\in\TT_\star}{T'\subsetneqq T}$ 
be the set of its successors in $\TT_\star$. Let $T'\in\TT_\star|_T$. Recall that bisection ensures $|T'|\le|T|/2$. With $0<\widetilde q:=2^{-1/d}<1$, it follows
\begin{align*}
 \osc_\star(T',v_\diamond)^2 
 \le \widetilde q\,\big(h_T^2\norm{(1-\Pi_\diamond)(f+\div_\diamond\A\nabla v_\diamond)}{L^2(T')}^2 
 + h_{T}\,\norm{(1-\Pi_\diamond)\jump{\A\nabla v_\diamond}}{L^2((\partial T'\cap\partial T)\backslash\Gamma)}^2\big),
\end{align*}%
since $\A\nabla v_\diamond$ is smooth inside of $T$ so that all normal jumps inside of $T$ vanish. 
This leads to
\begin{align*}
 \osc_\star(\TT_\star\backslash\TT_\diamond,v_\diamond)^2
 &= 
 \!\!\!\sum_{T\in\TT_\diamond\backslash\TT_\star}\sum_{T'\in\TT_\star|_T}\osc_\star(T',v_\diamond)^2
 \le\widetilde q\!\!\sum_{T\in\TT_\diamond\backslash\TT_\star}\osc_\diamond(T,v_\diamond)^2
 = \widetilde q\,\osc_\diamond(\TT_\diamond\backslash\TT_\star,v_\diamond)^2.
\end{align*}
Choosing $\delta>0$ such that $0<q:=(1+\delta)\,\widetilde q<1$, we conclude the proof of~(B2').
\end{proof}

\subsection{General quasi-orthogonality \& linear convergence}
\label{section:A3}%
The following proposition proves~\eqref{eq:mns:linear} in Theorem~\ref{theorem:mns}\rm (i) and shows, in particular, that the general quasi-orthogonality~(A3) is satisfied.

\begin{proposition}\label{prop:mns:linear}
There is a constant $H>0$ such that the following statement is valid provided that  $\norm{h_0}{L^\infty(\Omega)}\le H$:
For all $0<\theta'\le\theta\le1$ Algorithm~\ref{algorithm:mns} guarantees linear convergence 
\rm in the sense of Theorem~\ref{theorem:mns}\rm (i).
Moreover, together with reliability~\eqref{eq:reliable}, estimate~\eqref{eq:mns:linear} also implies the general quasi-ortho\-go\-nality~{\rm(A3)}.
%
\end{proposition}

Our proof relies on the following quasi-Galerkin orthogonality property from~\cite{xzz06}.

\begin{lemma}
\label{lemma:orth}
Let $\TT_\star\in\refine(\TT_\ell)$. Then, the corresponding discrete solutions
satisfy
\begin{align}\label{eq:galerkin}
 \enorm{u-u_\star}^2
 \le \enorm{u-u_\ell}^2 - (1-\delta)\,\enorm{u_\star-u_\ell}^2
 + \delta^{-1}\c{galerkin}\,\osc_{\star}^2
\end{align}
for all $0<\delta<1$. The constant $\setc{galerkin}>0$ depends only on
$\sigma$-shape regularity of $\TT_{\star}$.
\end{lemma}

\begin{proof}
According to~\cite[Theorem~5.1]{xzz06}, it holds
\begin{align*}
 |\AA(u-u_\star,v_\star)|
 \le C\,\enorm{v_\star}\,\osc_{\star}
 \quad\text{for all }v_\star\in\SS^1_0(\TT_\star),
\end{align*}
where $C>0$ depends only on $\sigma$-shape regularity of $\TT_{\star}$.
For each $\delta>0$, the symmetry of $\AA(\cdot,\cdot)$, the last estimate, 
and the Young inequality $2ab\le \delta a^2+\delta^{-1}b^2$ yield
\begin{align*}
\enorm{u-u_\star}^2 
&= \enorm{u-u_\ell}^2 - 2\,\AA(u-u_\star,u_\star-u_\ell) - \enorm{u_\star-u_\ell}^2
 \\&\le\enorm{u-u_\ell}^2 -(1-\delta)\, \enorm{u_\star-u_\ell}^2 + C^2\delta^{-1}\,\osc_{\star}^2.
\end{align*}%
This concludes the proof with $\c{galerkin}=C^2$.
\end{proof}

\begin{proof}[\bfseries Proof of Proposition~\ref{prop:mns:linear}]
{\bf Step 1.}
There exist constants $\Cest>0$ and $0<\qest<1$ which depend only on $0<\theta\le1$ and the constants in~(A1)--(A2), such that
\begin{align}\label{eq:mns:step1}
 \eta_{\ell+1}^2 \le \qest\,\eta_\ell^2 + \Cest\,\enorm{u_{\ell+1}-u_\ell}^2
 \quad\text{for all }\ell\in\N_0:
\end{align}
The combination of~(A1)--(A2) yields for all $\eps>0$ that 
\begin{align*}
 \eta_{\ell+1}^2
 \le q\,\eta_\ell(\TT_{\ell}\backslash\TT_{\ell+1})^2
 + (1+\eps)\,\eta_\ell(\TT_{\ell}\cap\TT_{\ell+1})^2
 + (C+(1+\eps^{-1})C^2)\,\enorm{u_{\ell+1}-u_\ell}^2.
\end{align*}%
Note that $\eta(\MM_\ell^\eta)\le\eta_\ell(\TT_{\ell}\backslash\TT_{\ell+1})$. Therefore, the D\"orfler marking~\eqref{eq:doerfler:eta} yields
\begin{eqnarray*}
 q\,\eta_\ell(\TT_{\ell}\backslash\TT_{\ell+1})^2
 + (1+\eps)\,\eta_\ell(\TT_{\ell}\cap\TT_{\ell+1})^2
 &=& (1+\eps)\,\eta_\ell^2-(1+\eps-q)\,\eta_\ell(\TT_{\ell}\backslash\TT_{\ell+1})^2
 \\
 &\stackrel{\eqref{eq:doerfler:eta}}\le&
 (1+\eps-\theta\,(1+\eps-q))\,\eta_\ell^2.
\end{eqnarray*}
For sufficiently small $\eps>0$, we see $0<\qest:=1+\eps-\theta\,(1+\eps-q)<1$ 
and conclude~\eqref{eq:mns:step1}.

{\bf Step 2.} 
There exist constants $\Cest>0$ and $0<\qest<1$ which depend only on $0<\theta'\le1$ and the constants in~(B1')--(B2'), such that
\begin{align}\label{eq:mns:step2}
 \osc_{\ell+1}^2 \le \qest\,\osc_\ell^2 + \Cest\,\norm{h_{\ell+1}}{L^\infty(\Omega)}^2\enorm{u_{\ell+1}-u_\ell}^2
 \quad\text{for all }\ell\in\N_0:
\end{align}
The proof follows verbatim to that of~\eqref{eq:mns:step1}, but now involves (B1')--(B2')
in combination with the D\"orfler marking~\eqref{eq:doerfler:osc} for the data oscillations.

{\bf Step 3.}
Without loss of generality, we may assume that the constants $\Cest>0$ and $0<\qest<1$ in~\eqref{eq:mns:step1}--\eqref{eq:mns:step2} are the same.
With free parameters $\gamma,\mu>0$ which are fixed later, 
we define 
\begin{align*}
 \Delta_\star := \enorm{u-u_{\star}}^2 + \gamma\,\eta_\star^2 + \mu\,\osc_\star^2.
\end{align*}
We claim that there are constants $\gamma,\mu,C>0$ and $0<\qlin<1$ such that
\begin{align}\label{eq:mns:step3}
 \Delta_{\ell+1} \le \qlin\Delta_\ell 
 - \big(1/4-C\,\norm{h_{\ell+1}}{L^\infty(\Omega)}^2\big)\,\enorm{u_{\ell+1}-u_\ell}^2,
\end{align}
where $\gamma,\mu,C,\qlin$ depend only on $\theta$, $\theta'$,
 uniform $\sigma$-shape regularity of the triangulations $\TT_\star\in\refine(\TT_0)$, and the assumptions~\eqref{eq:A}--\eqref{eq3:A} on $\A$: To prove this, we use the quasi-Galerkin orthogonality~\eqref{eq:galerkin} with $\delta=1/2$ and the estimates~\eqref{eq:mns:step1}--\eqref{eq:mns:step2} to see
\begin{eqnarray*}
 \Delta_{\ell+1}
 &\stackrel{\eqref{eq:galerkin}}\le&
 \enorm{u-u_\ell}^2
 + \gamma\,\eta_{\ell+1}^2 + (\mu+2\c{galerkin})\,\osc_{\ell+1}^2
 - (1/2)\,\enorm{u_{\ell+1}-u_\ell}^2
 \\
&\le&
 \enorm{u-u_\ell}^2 
 + \qest\,\gamma\,\eta_\ell^2 
 + \qest\,\frac{\mu+2\c{galerkin}}{\mu}\,\mu\,\osc_{\ell}^2
\\&& - \big(1/2-\gamma\Cest-(\mu+2\c{galerkin})\Cest\,\norm{h_{\ell+1}}{L^\infty(\Omega)}^2\big)\,\enorm{u_{\ell+1}-u_\ell}^2
\end{eqnarray*}
For all $\eps>0$, reliability~\eqref{eq:reliable} implies
\begin{align*}
 \enorm{u-u_\ell}^2 + \qest\,\gamma\,\eta_\ell^2
 \le (1-\eps)\,\enorm{u-u_\ell}^2 + (\qest + \gamma^{-1}\Crel\eps)\,\gamma\,\eta_\ell^2.
\end{align*}
We choose $\gamma>0$ sufficiently small such that $\gamma\Cest \le 1/4$.
Additionally, we choose $\eps>0$ sufficiently small and $\mu>0$ sufficiently large such that
\begin{align*}
 0 < q_1 := \qest + \gamma^{-1}\Crel\eps < 1
 \quad\text{and}\quad
 0 < q_2 := \qest\,\frac{\mu+2\c{galerkin}}{\mu} < 1.
\end{align*}%
Combining the latter estimates with $C:=(\mu+2\c{galerkin})\Cest>0$, we arrive at
\begin{align*}
 \Delta_{\ell+1}
 &\le (1-\eps)\,\enorm{u-u_\ell}^2 + q_1\,\gamma\,\eta_\ell^2
 + q_2\,\mu\,\osc_\ell^2 - (1/4 - C\,\norm{h_{\ell+1}}{L^\infty(\Omega)}^2\big)\,\enorm{u_{\ell+1}-u_\ell}^2
 \\&
 \le \qlin\,\Delta_\ell - (1/4 - C\,\norm{h_{\ell+1}}{L^\infty(\Omega)}^2\big)\,\enorm{u_{\ell+1}-u_\ell}^2,
\end{align*}
where $0 < \qlin := \max\{1-\eps,q_1,q_2\} < 1$. This concludes the proof of~\eqref{eq:mns:step3}.

{\bf Step~4.} 
Recall that $\norm{h_{\ell+1}}{L^\infty(\Omega)} \le \norm{h_0}{L^\infty(\Omega)}$. If $C\,\norm{h_0}{L^\infty(\Omega)}\le1/4$, estimate~\eqref{eq:mns:step3} proves
\begin{align*}
 \Delta_{\ell+1} 
 \le \qlin\,\Delta_\ell - \big(1/4-C\,\norm{h_{\ell+1}}{L^\infty(\Omega)}^2\big)\,\enorm{u_{\ell+1}-u_\ell}^2
 \le \qlin\,\Delta_\ell
 \quad\text{for all }\ell\in\N_0.
\end{align*}
With reliability~\eqref{eq:reliable} and $\osc_\ell^2 \le \eta_\ell^2$ from~\eqref{eq:osc-eta}, induction on $n$ proves
\begin{align*}
 \gamma\,\eta_{\ell+n}^2 \le \Delta_{\ell+n} \le \qlin^n\,\Delta_\ell
 \le \qlin^n(\Crel + \gamma + \mu)\,\eta_\ell^2
 \quad\text{for all }\ell,n\in\N_0.
\end{align*}
This proves linear convergence~\eqref{eq:mns:linear} with $\Clin = (\Crel + \gamma + \mu)\gamma^{-1}$.

{\bf Step~5.} Together with the triangle inequality 
$\enorm{u_{k+1}-u_k}^2 \le 2\enorm{u-u_{k+1}}^2 + 2\enorm{u-u_k}^2$, reliability~\eqref{eq:reliable}, and linear convergence~\eqref{eq:mns:linear}, the geometric series yields
\begin{align*}
 \sum_{k=\ell}^N \enorm{u_{k+1}-u_k}^2 
 \le 4\,
 \sum_{k=\ell}^{N+1} \enorm{u-u_k}^2 
 \stackrel{\eqref{eq:reliable}}\lesssim \sum_{k=\ell}^{N+1}\eta_k^2
 \le \sum_{j=0}^\infty\eta_{j+\ell}^2
 \stackrel{\eqref{eq:mns:linear}}\lesssim \eta_\ell^2 \sum_{j=0}^\infty \qlin^j
 \lesssim \eta_\ell^2.
\end{align*}
This concludes the validity of the general quasi-orthogonality~(A3).
\end{proof}

\subsection{Auxiliary results}
\label{subsec:auxialiary}
For the convenience of the reader, this section collects some well-known properties of the FVM which are exploited in the subsequent proofs.

\begin{lemma}
  \label{lem:proplindual}
  With $\chi_i^*\in\PP^0(\TT_\star^*)$ being the characteristic function of $V_i\in\TT_\star^*$, we define the interpolation operator
 \begin{align*}
  \II_\star^*:\CC(\overline\Omega)\to\PP^0(\TT^*_{\star}),\quad
  \II_\star^*v:=\sum_{a_i\in\NN_\star}v(a_i)\chi_i^*.
\end{align*}
Then, for all $T\in\TT_\star$, $\facet\in\EEt$, and $v_\star\in\SS^1(\TT_\star)$, it holds 
  \begin{align}
    \label{eq1:proplindual}
    &\int_T (v_\star-\II_\star^*v_\star)\,dx=0=\int_{\facet}(v_\star-\II_\star^*v_\star)\,ds=0,\\
    \label{eq3:proplindual}
    &\norm{v_\star-\II_\star^*v_\star}{L^2(T)}\leq h_T \norm{\nabla v_\star}{L^2(T)},\\
    \label{eq4:proplindual}
    &\norm{v_\star-\II_\star^*v_\star}{L^2(\facet)}\leq C h_T^{1/2}\norm{\nabla v_\star}{L^2(T)}.
  \end{align}
In particular, it holds $\II_\star^*v_\star \in \PP^0_0(\TT_\star^*)$ for all $v_\star\in\SS^1_0(\TT_\star)$.
The constant $C>0$ depends only on $\sigma$-shape regularity of $\TT_\star$.
\end{lemma}

\begin{proof}
The proof of~\eqref{eq1:proplindual} is 
based on the construction of $\TT^*_\star$ and exploits that $v_\star$ 
is a piecewise linear function on $\TT_\star$. 
A proof for~\eqref{eq3:proplindual} is found in~\cite[Lemma 1.4.2]{Erath:2010-phd},
and~\eqref{eq4:proplindual} follows from~\eqref{eq3:proplindual} and the trace inequality.
\end{proof}

The following lemma is a key observation.
For discrete ansatz and test spaces, it allows to understand the FVM bilinear 
form as a perturbation of the bilinear form of the weak 
formulation. 
The proof is given 
in~\cite{Ewing:2002-1,Chatzipantelidis:2002-1,Erath:2012-1} 
for Lipschitz-continuous $\A$, but transfers directly to the present situation, where $\A$ satisfies~\eqref{eq:A}--\eqref{eq3:A}. 

\begin{lemma}[{{\cite{Ewing:2002-1,Chatzipantelidis:2002-1,Erath:2012-1}}}] 
 \label{lem:bilinearcomp}
 It holds
 \begin{align}
  \label{eq:bilinearcomp}
  |\AA(v_\star,w_\star)&-\AA_\star(v_\star,\II_\star^*w_\star)|
    \leq \c{bil}\sum_{T\in\TT}h_T\,\enorm{v_\star}_T\enorm{w_\star}_T
    \text{ for all }v_\star,w_\star\in\SS^1_0(\TT_\star).
 \end{align} 
 Moreover, there exists some constant $H>0$ such that 
 $\c{bil}\norm{h_\star}{L^\infty(\Omega)}\le H$ implies
 \begin{align}\label{eq:stable}
  \AA_\star(v_\star,\II_\star^*v_\star)
  \ge C_{\rm stab}\,\enorm{v_\star}^2
  \text{ for all }v_\star\in\SS^1_0(\TT_\star).
 \end{align}
 In particular, this proves that the FVM system~\eqref{eq:disc_systemfvem} has a unique solution $u_\star\in\SS^1_0(\TT_\star)$.
 While $H>0$ depends only on the assumptions~\eqref{eq:A}--\eqref{eq3:A} on $\A$, 
 the constants $\c{bil}$ and $C_{\rm stab}$ depend additionally on $\sigma$-shape regularity of $\TT_\star$.\qed
\end{lemma}

\subsection{Proof of Theorem~\ref{theorem:cea}}
\label{section:cea}%
The proof is split in several steps: 

\noindent
{\bf Step~1.}
For arbitrary $v_\star,w_\star\in \SS^1_0(\TT_\star)$ and $w_\star^*:=\II_\star^*w_\star$, we prove
the identity 
\begin{align}\label{eq:cea:step1}
 \hspace*{-1mm}\AA(v_\star,w_\star) \!-\! \AA_\star(v_\star,w_\star^*)
 = \!\!\sum_{T\in\TT_\star}\!\!\Big(\product{\div_\star\A\nabla v_\star}{w_\star^*\!-\!w_\star}_{T}
 - \product{\A\nabla v_\star\cdot\nf}{w_\star^*\!-\!w_\star}_{\partial T\backslash\Gamma}
 \Big):
\end{align}	
First, elementwise integration by parts for the bilinear form $\AA(v_\star,w_\star)$ leads to
\begin{align*}
\AA(v_\star,w_\star)
=\sum_{T\in\TT_\star}\product{\A\nabla v_\star}{\nabla w_\star}_T
=\sum_{T\in\TT_\star}\Big(-\product{\div_\star\A\nabla v_\star}{w_\star}_{T}
+\product{\A\nabla v_\star\cdot\nf}{w_\star}_{\partial T\backslash\Gamma} \Big),
\end{align*}  
since $w_\star|_\Gamma=0$.
Second, we rewrite the FVM bilinear form $\AA_\star(v_\star,w_\star^*)$.
Note that $w_\star^*$ does not jump across facets $\facet\in\EE_\star$.
Therefore,
\begin{align*}
\begin{split}
\AA_\star(v_\star,w_\star^*)
= \sum_{a_i\in\NN_\star^\Omega}w_\star^*|_{V_i}\int_{\partial V_i}(-\A\nabla v_\star)\cdot\nf\,ds
= -\sum_{T\in\TT_\star}\sum_{a_i\in\NN_T\backslash\Gamma}
\product{\A\nabla v_\star\cdot\nf}{w_\star^*}_{T\cap\partial V_i}.
\end{split}
\end{align*}
Note that $\NN_T\backslash\Gamma$ can be replaced by $\NN_T$, since $w_\star^*|_\Gamma=0$.
Integration by parts thus yields
\begin{align*}
\AA_\star(v_\star,w_\star^*)
&=-\sum_{T\in\TT_\star}\Big(\sum_{a_i\in\NN_T}
\product{\A\nabla v_\star\cdot\nf}{w_\star^*}_{\partial(T\cap V_i)}
-\product{\A\nabla v_\star\cdot\nf}{w_\star^*}_{\partial T\backslash\Gamma}\Big)\nonumber\\
&=\sum_{T\in\TT_\star}\Big(-\product{\div_\star\A\nabla v_\star}{w_\star^*}_{T}
+\product{\A\nabla v_\star\cdot\nf}{w_\star^*}_{\partial T\backslash\Gamma}\Big).
\end{align*}
The difference of the above estimates prove~\eqref{eq:cea:step1}.

{\bf Step~2.}
For arbitrary $v_\star,w_\star\in \SS^1_0(\TT_\star)$ and $w_\star^*:=\II_\star^*w_\star$, it holds
\begin{align}\label{eq:cea:step2}
 (f,w_\star^*-w_\star) - \big(\AA_\star(v_\star,w_\star^*) - \AA(v_\star,w_\star)\big)
 \le C\,\osc_\star(v_\star)\,\norm{\nabla w_\star}{L^2(\Omega)},
\end{align}
where $C>0$ depends only on $\sigma$-shape regularity of $\TT_\star$: 
With~\eqref{eq:cea:step1}, the definition of the facet residual~\eqref{eq:edgeJ}, the
$L^2$-orthogonalities~\eqref{eq1:proplindual}
and the Cauchy-Schwarz inequality,
we see
\begin{align*}
 &(f,w_\star^*-w_\star)_\Omega - \big(\AA_\star(v_\star,w_\star^*) - \AA(v_\star,w_\star)\big)\\
&\qquad =\sum_{T\in\TT_\star}\product{(1-\Pi_\star)(f+\div_\star\A\nabla v_\star)}{w_\star^*-w_\star}_{T}
 -\sum_{\facet\in\EE_\star^\Omega}\product{(1-\Pi_\star)J_\star(v_\star)}
 {w_\star^*-w_\star}_{\facet}\\
 &\qquad \leq
 \osc_\star(v_\star)\,\Big(\sum_{T\in\TT_\star}h_T^{-2}\norm{w_\star^*-w_\star}{L^2(T)}^2\Big)^{1/2}
 +\osc_\star(v_\star)\,\Big(\sum_{T\in\TT_\star}h_T^{-1}\norm{w_\star^*-w_\star}{L^2(\partial T\backslash\Gamma)}^2\Big)^{1/2}
\end{align*}%

With~\eqref{eq3:proplindual}--\eqref{eq4:proplindual}, we conclude~\eqref{eq:cea:step2}.

{\bf Step~3.}
The FVM solution $u_\star$ satisfies
\begin{align}\label{eq:cea:step3}
 C^{-1}\,\enorm{u_\star-v_\star} \le \enorm{u-v_\star} + \osc_\star(v_\star)
 \quad\text{for all }v_\star\in\SS^1_0(\TT_\star),
\end{align}
where $C>0$ depends only on $\sigma$-shape regularity of $\TT_\star$ and the assumptions~\eqref{eq:A}--\eqref{eq3:A} on $\A$: 
With $u$ being the weak solution, we first note the identities 
\begin{align*}
 0 &\stackrel{\eqref{eq:weakform}}= (f,w_\star)_\Omega - \AA(u,w_\star)
 = \big[(f,w_\star)_\Omega - \AA(v_\star,w_\star)\big]
 - \AA(u-v_\star,w_\star)
 \quad\text{for all }v_\star\in\SS^1_0(\TT_\star).
\end{align*}
For  sufficiently fine $\TT_\star$, Lemma~\ref{lem:bilinearcomp} applies. 
Choose $w_\star:=u_\star-v_\star$ and $w_\star^*:=\II_\star^*w_\star$. Then,
\begin{align*}
 \enorm{u_\star-v_\star}^2 
 &\stackrel{\eqref{eq:stable}}\lesssim
 \AA_\star(u_\star-v_\star,w_\star^*)
 \\&
 \stackrel{\eqref{eq:disc_systemfvem}}= 
 \big[(f,w_\star^*)_\Omega - \AA_\star(v_\star,w_\star^*)\big] - \big[(f,w_\star)_\Omega - \AA(v_\star,w_\star)\big]
 + \AA(u-v_\star,w_\star).
\end{align*}
Combining this with~\eqref{eq:cea:step2} and norm equivalence $\norm{\nabla w_\star}{L^2(\Omega)}\simeq\enorm{w_\star}$, we obtain
\begin{align*}
 \enorm{u_\star-v_\star}^2 
 \lesssim \osc_\star(v_\star)\,\enorm{w_\star} + \enorm{u-v_\star}\,\enorm{w_\star},
\end{align*}
where the hidden constant depends only on $\sigma$-shape regularity of $\TT_\star$ and the assumptions~\eqref{eq:A}--\eqref{eq3:A} on $\A$.
By choice of $w_\star$, we conclude~\eqref{eq:cea:step3}.

{\bf Step~4.}
Let $v_\star\in\SS^1_0(\TT_\star)$.
We employ (B1') with $\TT_\diamond = \TT_\star$ and $v_\diamond=u_\star$. 
Combining this with the triangle inequality and~\eqref{eq:cea:step3}, we see
\begin{align*}
 \enorm{u-u_\star} + \osc_\star
 \lesssim \enorm{u-v_\star} + \enorm{u_\star-v_\star} + \osc_\star(v_\star)
 \stackrel{\eqref{eq:cea:step3}}\lesssim \enorm{u-v_\star} + \osc_\star(v_\star).
\end{align*}
Altogether, this proves
\begin{align*}
 \enorm{u-u_\star} + \osc_\star
 \lesssim
 \min_{v_\star\in\SS^1_0(\TT_\star)}\big(\enorm{u-v_\star} + \osc_\star(v_\star)\big)
 \le \enorm{u-u_\star} + \osc_\star.
\end{align*}
Reliability~\eqref{eq:reliable} and efficiency~\eqref{eq:efficient} together with~\eqref{eq:osc-eta} imply $\eta_\star\simeq \enorm{u-u_\star} + \osc_\star$. This concludes~\eqref{eq:totalerror}. 
For the equivalence
\begin{align*}
 \enorm{u-u_\star^{\rm FEM}} + \osc_\star(u_\star^{\rm FEM})
\simeq \min_{v_\star\in\SS^1_0(\TT_\star)}\big(\enorm{u-v_\star} + \osc_\star(v_\star)\big),
\end{align*}
the reader is referred to~\cite[Lemma~5.1]{ffp14}. 
This also concludes~\eqref{eq:equivalence}.\qed

\subsection{Proof of Theorem~\ref{corollary:cea}}
\label{section:cea2}%
Let $v_\star\in\SS^1_0(\TT_\star)$. For $T\in\TT_\star$, the
triangle inequality shows
\begin{align*}
 \osc_\star(T,v_\star)^2
 &\lesssim 
 h_T^2\,\norm{(1-\Pi_\star)f}{L^2(T)}^2
 + h_T^2\,\norm{(1-\Pi_\star)\div_\star \A\nabla v_\star}{L^2(T)}^2
 \\&\quad
 + h_T\,\norm{(1-\Pi_\star)J_\star(v_\star)}{L^2(\partial T\backslash\Gamma)}^2.
\end{align*}
With Step~1--2 from the proof of Lemma~\ref{lemma:osc}, we thus see
\begin{align*}
 \osc_\star(v_\star)^2
 &\lesssim \norm{h_\star(1-\Pi_\star)f}{L^2(\Omega)}^2
 + \norm{h_\star\nabla v_\star}{L^2(\Omega)}^2.
\end{align*}
With~\eqref{eq:totalerror}, we obtain
\begin{eqnarray*}
 \enorm{u-u_\star} + \osc_\star 
 &\stackrel{\eqref{eq:totalerror}}\simeq& \min_{v_\star\in\SS^1_0(\TT_\star)}\big(\enorm{u-v_\star} + \osc_\star(v_\star)\big)
 \\&\lesssim& \norm{h_\star(1-\Pi_\star)f}{L^2(\Omega)} + \min_{v_\star\in\SS^1_0(\TT_\star)}\big(\enorm{u-v_\star} + \norm{h_\star\nabla v_\star}{L^2(\Omega)}\big).
\end{eqnarray*}
This proves~\eqref{eq:cor:cea}. In particular, norm equivalence $\enorm{u-v_\star}\simeq\norm{\nabla(u-v_\star)}{L^2(\Omega)}$ implies
\begin{align*}
 \enorm{u-u_\star} + \osc_\star 
 \lesssim \norm{h_\star}{L^\infty(\Omega)}\big(\norm{f}{L^2(\Omega)}+\norm{\nabla u}{L^2(\Omega)}\big)
 + \min_{v_\star\in\SS^1_0(\TT_\star)} \enorm{u-v_\star}.
\end{align*}
From this, we also conclude~\eqref{eq3:cor:cea}--\eqref{eq2:cor:cea}.\qed

\subsection{Discrete reliability}
\label{section:A4}
The main result of this section is the following variant of the discrete reliability~(A4).

\begin{proposition}\label{prop:dlr}
Let $\TT_\star\in\refine(\TT_\diamond)$ be an arbitrary refinement of 
$\TT_\diamond\in\refine(\TT_0)$ 
and suppose that the corresponding discrete solutions $u_\star$ or $u_\diamond$ exist.
Then,
\begin{align}\label{eq:prop:dlr}
 \enorm{u_\star - u_\diamond}^2
 \le\c{bil}\sum_{T\in\TT_{\star}}h_T^2\enorm{u_\star - u_\diamond}_T^2+
 \c{dlr}\,\sum_{T\in\RR_\diamond}\eta_{\diamond}(T)^2,
\end{align}
where
$\RR_\diamond:=\set{T\in\TT_\diamond}{\exists T'\in\TT_\diamond\backslash\TT_{\star}\quad T\cap T'\neq\emptyset}$,
consists of all refined elements $\TT_\diamond\backslash\TT_{\star}$ plus one
additional layer of neighboring elements.
In particular, the discrete reliability~{\rm(A4)} follows provided that $\TT_\star$ is 
sufficiently fine, i.e., $\c{bil}\norm{h_\star}{L^\infty(\Omega)}^2\le1/2$.
The constants $\c{bil},\c{dlr}>0$ depend
only on $\Omega$, the assumptions~\eqref{eq:A}--\eqref{eq3:A} on $\A$, and on
$\sigma$-shape regularity of $\TT_\diamond$.
\end{proposition}
The proof of Proposition~\ref{prop:dlr} relies on two
properties of the volume and facet residual, i.e., an orthogonality
property~\eqref{eq1:lem:discdefect}
and a discrete defect identity~\eqref{eq2:lem:discdefect} of the FVM bilinear form. 
\begin{lemma}
\label{lem:discdefect}
Let $\TT_{\star} \in \refine(\TT_\diamond)$ be an arbitrary refinement of
$\TT_\diamond\in\refine(\TT_0)$ 
and suppose that the corresponding discrete solutions $u_\star$ or $u_\diamond$ exist.
Then, there holds
\begin{align}\label{eq1:lem:discdefect}
    \sum_{T\in\TT_\diamond}\product{R_{\diamond}(u_\diamond)}{v^*_\diamond}_T
    -\sum_{F\in\EE_{\diamond}^\Omega}\product{J_{\diamond}(u_\diamond)}{v^*_\diamond}_F &= 0
  \quad\text{for all }v^*_\diamond\in\PP^0_0(\TT^*_\diamond)\\
\intertext{as well as}
\label{eq2:lem:discdefect}
  \sum_{T\in\TT_\diamond}\product{R_{\diamond}(u_\diamond)}{v^*_{\star}}_T - 
  \sum_{F\in\EE_{\diamond}^\Omega}\product{J_{\diamond}(u_\diamond)}{v^*_{\star}}_F
  &= \AA_\star(u_\star-u_\diamond,v^*_{\star})
  \quad\text{for all }v^*_{\star}\in\PP^0_0(\TT^*_{\star}).
\end{align}
\end{lemma}

\begin{proof}
The proof of~\eqref{eq1:lem:discdefect} is well-known and found, e.g.,
in~\cite{Carstensen:2005-1,Erath:2010-phd,Erath:2013-1}. 
The proof of~\eqref{eq2:lem:discdefect} is adopted from~\cite{Zou:2009-1} 
for an arbitrary refinement $\TT_\star\in\refine(\TT_\diamond)$:
The divergence theorem shows for all boxes $V'\in\TT^*_\star$ that
\begin{align}
	\label{eq:helpdiscdefect}
	\sum_{T'\in\TT_\star}\int_{T'\cap V'}\div_\star\A\nabla u_\diamond\,dx
	=\sum_{\zeta'\in\EE_{V',\star}}\int_{\zeta'\backslash\Gamma}J_\star(u_\diamond)\,ds+
	\int_{\partial V'}\A\nabla u_\diamond\cdot\nf\,ds.
\end{align}
Let $v^*_{\star}\in\PP^0_0(\TT^*_\star)$. We multiply the above equation by $v^*_{\star}|_{V'}$ and sum over all $V'\in\TT^*_\star$. With $\div_\star \A\nabla u_\diamond = \div_\diamond \A\nabla u_\diamond$, the left-hand side then reads
\begin{align}
		\label{eq:helpadisdefect}
 \sum_{V'\in\TT_\star^*}v^*_{\star}|_{V'}\sum_{T'\in\TT_\star}
 \int_{T'\cap V'}\div_\star \A\nabla u_\diamond\,dx
 = \product{\div_\star \A\nabla u_\diamond}{v^*_{\star}}_\Omega
 = \sum_{T\in\TT_\diamond}\product{\div_\diamond \A\nabla u_\diamond}{v^*_{\star}}_T.
\end{align}
Since $\A\nabla u_\diamond$ is continuous in $T\in\TT_\diamond$ and $J_\star(u_\diamond) = J_\diamond(u_\diamond)$ on $F\in\EE^\Omega_\diamond$, it holds
\begin{align}
	\label{eq:helpbdisdefect}
 \sum_{V'\in\TT_\star^*}v^*_{\star}|_{V'} 
 \sum_{\zeta'\in\EE_{V',\star}}\int_{\zeta'\backslash\Gamma}J_\star(u_\diamond)\,ds
  = \sum_{F'\in\EE_\star^\Omega}\product{J_\star(u_\diamond)}{v_\star^*}_{F'}
  = \sum_{F\in\EE^\Omega_\diamond}\product{J_\diamond(u_\diamond)}{v_\star^*}_F.
\end{align}
By definition~\eqref{eq:fvembilinear} of $\AA_\star(\cdot,\cdot)$, the identity~\eqref{eq:helpdiscdefect}
becomes with~\eqref{eq:helpadisdefect} and~\eqref{eq:helpbdisdefect}
\begin{align}
	\label{eq:help1discdefect}
 \sum_{T\in\TT_\diamond}\product{\div_\diamond \A\nabla u_\diamond}{v^*_{\star}}_T
 = \sum_{F\in\EE^\Omega_\diamond}\product{J_\diamond(u_\diamond)}{v_\star^*}_F - \AA_\star(u_\diamond,v^*_{\star}).
\end{align}%
On the other hand the FVM formulation~\eqref{eq:disc_systemfvem} yields
\begin{align}
\label{eq:help2discdefect}
\product{f}{v^*_{\star}}_\Omega = \AA_\star(u_\star,v^*_{\star}).	
\end{align}
Adding~\eqref{eq:help1discdefect}--\eqref{eq:help2discdefect}, we conclude the proof.
\end{proof}
The following Poincar\'e- and trace-type inequalities play a key role to estimate
quantities over the elements of the dual grid.
\begin{lemma}\label{lemma:Pi*}
For each box $V_i\in\TT_\star^*$, let $a_i\in\NN_\star$ be the corresponding node.
Define 
\begin{align*}
\Pi_\star^*:L^2(\Omega)\to\PP^0_0(\TT_\star^*),
\quad(\Pi_\star^* v)|_{V_i} := \begin{cases}
 \frac{1}{|V_i|}\,\int_{V_i} v\,dx,
 \quad&\text{if }a_i\in\NN_\star^\Omega,\\
 0,&\text{if }a_i\in\NN_\star^\Gamma.
 \end{cases}
\end{align*}
Let $V\in\TT_\star^*$ and $\zeta\in\EE_{V,\star}$.
Then, there holds, for all $v\in H^1_0(\Omega)$,
\begin{align}
  \label{eq:poincareV}
  \norm{v-\Pi_\star^*v}{L^2(V)} &\le C\, \diam(V)\,\norm{\nabla v}{L^2(V)},\\
  \label{eq:trace2}
  \norm{v-\Pi_\star^*v}{L^2(\zeta)} &\le C\,\diam(V)^{1/2}\,\norm{\nabla v}{L^2(V)}.
\end{align}
The constant $C>0$ depends only on the $\sigma$-shape regularity of $\TT_\star$.
\end{lemma}

\begin{proof}
The set $\TT_\star|_V := \set{V\cap T}{T\in\TT_\star\text{ with }V\cap T\neq\emptyset}$ is a partition of $V$ into quadrilaterals in 2D and cuboids in 3D, respectively. 
In 2D each quadrilateral can itself be divided into two triangles.
In 3D each cuboid can be divided into three pyramids (with the center of gravity
of $T$ as top).
Note that a quadrilateral $\zeta\in\EE_{V,\star}$ 
builds the base of one pyramid.
This gives rise to a triangulation $\ZZ_{V,\star}$ of $V$;
see Figure~\ref{fig:primaldual} and Figure~\ref{fig:dual3d} for 2D and 3D, respectively. 

Choose $Z\in\ZZ_{V,\star}$ with $\zeta\subset \partial Z$.
Note that $\ZZ_{V,\star}$ is $\sigma'$-shape regular, where $\sigma'$ depends only on $\sigma$, and that the box $V$ is just the node patch of the corresponding node $a\in\NN_\star$ with respect to $\ZZ_{V,\star}$.  If $a\in\NN_\star^\Omega$, let $v_Z:=(1/|Z|)\,\int_Zv\,dx$ denote the piecewise integral mean. If $a\in\NN_\star^\Gamma$, we define $v_Z:=0$, since then $V\cap\Gamma$ has positive measure.
In either case, it holds 
\begin{align*}
 \norm{v-\Pi_\star^*v}{L^2(V)} \le \norm{v-v_Z}{L^2(V)}
 &\lesssim \diam(Z)\,\norm{\nabla v}{L^2(V)},
\end{align*}
where the hidden constant depends only on $\sigma'$ and hence on $\sigma$;
see~\cite{MR559195}. With $\diam(Z)\le\diam(V)$, the Poincar\'e-type inequality~\eqref{eq:poincareV} follows.

The trace inequality, a scaling argument, and $\diam(V)\simeq\diam(Z)$ lead to
\begin{align*}
	\begin{split}
  \norm{v}{L^2(\zeta)}
  &\lesssim \diam(Z)^{-1/2}\norm{v}{L^2(Z)}+
  \diam(Z)^{1/2}\norm{\nabla v}{L^2(Z)}
  \\&\lesssim \diam(V)^{-1/2}\norm{v}{L^2(V)}+\diam(V)^{1/2}\norm{\nabla v}{L^2(V)}.
	\end{split}
\end{align*}
Combining this with the Poincar\'e-type inequality~\eqref{eq:poincareV}, we obtain
\begin{align*}
  \norm{v-\Pi_\ell^*v}{L^2(\zeta)}
	&\lesssim \diam(V)^{1/2}\norm{\nabla v}{L^2(V)}.
\end{align*}%
This concludes the proof.
\end{proof}

\begin{figure}[t]
  \centering
  \subfigure[Mesh $\TT_{\diamond}$ and $\TT_{\star}$ (incl. dashed lines).]
	{\label{fig:refmesh1}\includegraphics[width=0.23\textwidth]{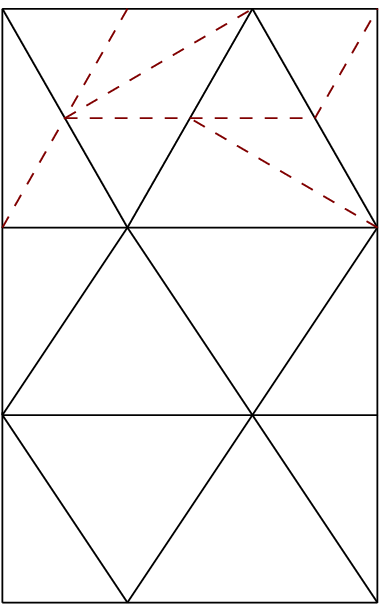}}
  \hspace*{0.01\textwidth}
  \subfigure[Dual mesh $\RR^*_{\diamond}$.]
  {\label{fig:refmesh2}\includegraphics[width=0.23\textwidth]{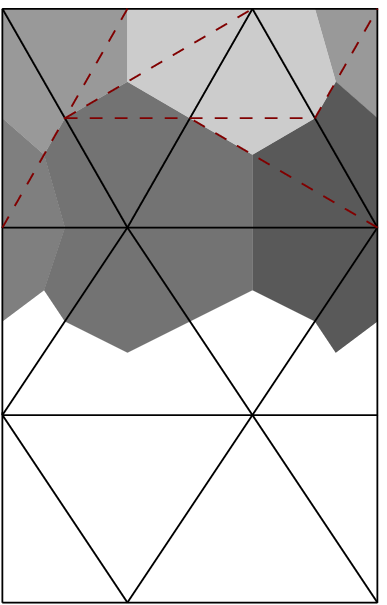}}
  \hspace*{0.01\textwidth}
  \subfigure[Dual mesh $\RR^*_{\star}$.]
  {\label{fig:refmesh3}\includegraphics[width=0.23\textwidth]{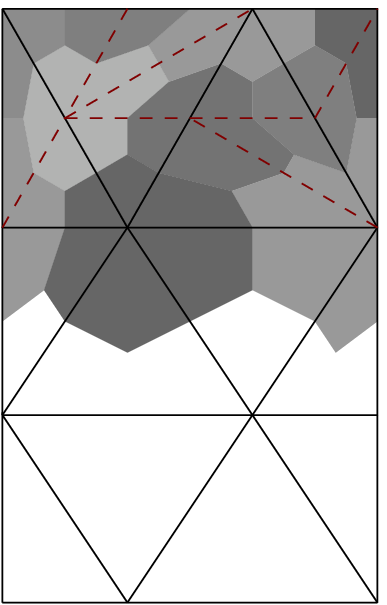}}
  \hspace*{0.01\textwidth}
  \subfigure[Set $\RR_\diamond$.]
  {\label{fig:refmesh4}\includegraphics[width=0.23\textwidth]{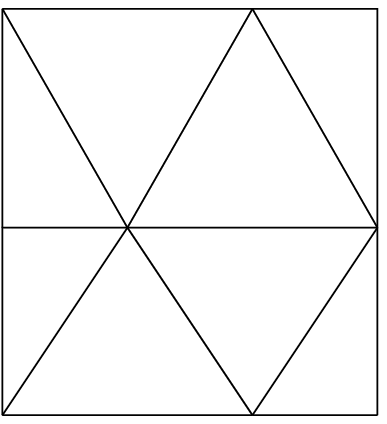}}
  \caption{In~\subref{fig:refmesh1} we see the coarse mesh $\TT_{\diamond}$ for 2D. The dashed lines 
	show the refinement and build the refined mesh $\TT_{\star}$. In~\subref{fig:refmesh2}
	and~\subref{fig:refmesh3} (gray boxes) we see the dual mesh of the refined areas
	notated by $\RR^*_{\diamond}$ and $\RR^*_{\star}$,
	respectively. Finally~\subref{fig:refmesh4} shows the elements $T\in\TT_{\diamond}$ which build
	the set $\RR_\diamond:=\{T\in\omega_a|a\subset\partial(\TT_{\diamond}\backslash\TT_{\star})\}$ in 
	this example.}
  \label{fig:refmesh}
\end{figure}
      
\begin{proof}[{\bfseries Proof of Proposition~\ref{prop:dlr}}]
To abbreviate notation, let $R_\diamond:=R_\diamond(u_\diamond)$ and $J_\diamond:=J_\diamond(u_\diamond)$ denote volume residual~\eqref{eq:residuum} and 
facet residual~\eqref{eq:edgeJ} with respect to the discrete solution $u_\diamond$.
For arbitrary
$v_{\star}\in\SS^1_0(\TT_{\star})$,
 $v^*_{\star}\in\PP^0_0(\TT_{\star})$, and $v^*_{\diamond}\in\PP^0_0(\TT_{\diamond})$,~\eqref{eq2:lem:discdefect}
 and~\eqref{eq1:lem:discdefect} of Lemma~\ref{lem:discdefect} together with the mesh relation~\eqref{eq:meshrelation} show
\begin{align}
	\label{eq:helpdlr1}
  \begin{split}	
  &\AA(u_\star-u_\diamond,v_{\star})\\
  &\qquad=\AA(u_\star-u_\diamond,v_{\star})-\AA_\star(u_\star-u_\diamond,v_{\star}^*)
  +\sum_{T\in\TT_\diamond}\product{R_{\diamond}}{v^*_{\star}-v^*_{\diamond}}_T -
  \sum_{F\in\EE_{\diamond}^\Omega}\product{J_{\diamond}}{v^*_{\star}-v^*_{\diamond}}_{F}\\
  &\qquad=\AA(u_\star-u_\diamond,v_{\star})-\AA_\star(u_\star-u_\diamond,v_{\star}^*)
  +\sum_{V\in\TT^*_\diamond}\Big(\product{R_{\diamond}}{v^*_{\star}-v^*_{\diamond}}_V -
  \sum_{\zeta\in\EE_{V,\diamond}}\product{J_{\diamond}}{v^*_{\star}-v^*_{\diamond}}_{\zeta\backslash\Gamma} \Big);
  \end{split}
\end{align}
see Section~\ref{subsec:dualmesh}, Figure~~\ref{fig:primaldual}\subref{fig:edgesV}
and Figure~\ref{fig:dual3d}\subref{fig:dual3dEEV}
for the definition of $\EE_{V,\diamond}$.

Next, we note that the discrete ansatz spaces are nested, while
the discrete test spaces are not. However, in the non-refined area $\TT_\diamond\cap\TT_\star$
the shape of the dual grid elements is the same.
We use this to truncate the sum of~\eqref{eq:helpdlr1}.
To get the final sum over $\RR_\diamond$
in~\eqref{eq:prop:dlr}, we have to define the functions $v_\star^*$ and $v^*_\diamond$
appropriately to apply Lemma~\ref{lem:proplindual} and Lemma~\ref{lemma:Pi*}, respectively.	
To formalize this, we define 
$\RR^*_\diamond:=\TT^*_\diamond\backslash \TT^*_{\star}$
and $\RR^*_{\star}:=\TT^*_{\star}\backslash\TT^*_\diamond$, 
i.e., the dual mesh of the refined areas; 
see Figure~\ref{fig:refmesh} for a 2D illustration. Note that 
 \begin{align}
	 \label{eq:help1disc}
 	\bigcup_{V\in\RR^*_\diamond}V=\bigcup_{V'\in\RR^*_{\star}}V',
 \end{align}
Consider the transition area $\RR_\diamond\backslash(\TT_\diamond\backslash\TT_{\star}) = \set{T\in \TT_\diamond\cap\TT_\star}{\exists T'\in\TT_\diamond\backslash\TT_{\star}\quad T¯\cap T'\neq\emptyset}$ (second row of
 triangles in Figure~\ref{fig:refmesh}) which consists of all non-refined neighbors of a refined element. For all $T\in\RR_\diamond\backslash(\TT_\diamond\backslash\TT_{\star})$, it holds
\begin{align*}
 \set{V\cap T}{V\in\RR^*_\diamond} = \set{V\cap T}{V\in\RR_\star^*},
\end{align*}
i.e., the shape of $V\in\RR^*_\diamond$ coincides with the shape of some $V'\in\RR^*_{\star}$ in the transition area.

Let $v_{\star}:=u_\star-u_\diamond\in\SS^1_0(\TT_{\star})$.
Choose $v_\star^*:=\II_\star^*v_\star\in\PP^0_0(\TT_\star^*)$. Define $v^*_\diamond\in\PP^0_0(\TT^*_\diamond)$ by
\begin{align*}
 v^*_\diamond|_{V} := \begin{cases}
 (\Pi^*_\diamond v_\star)|_{V} \quad&\text{if }V\in\RR^*_\diamond,\\
 (\II^*_\diamond v_\star)|_{V} \quad&\text{otherwise}.
 \end{cases}
\end{align*}
For $V\in\TT^*_\diamond\backslash\RR^*_\diamond = \TT^*_\diamond\cap\TT_\star^*$, this implies
$v^*_{\star}|_V=v^*_{\diamond}|_V$,
i.e., $v^*_{\star}=v^*_{\diamond}$ within the white area of Figure~\ref{fig:refmesh}\subref{fig:refmesh2} and~\ref{fig:refmesh}\subref{fig:refmesh3}. 
We use this observation to truncate the sum over $\TT^*_\diamond$ in~\eqref{eq:helpdlr1} and replace $\TT^*_\diamond$ by $\RR^*_\diamond$.
Together with~\eqref{eq:bilinearcomp} from Lemma~\ref{lem:bilinearcomp} for the bilinear forms, we get
 \begin{align}\label{eq:dlr:dp1}
 \begin{split}
   \AA(u_\star-u_\diamond,v_{\star})&\leq
   \c{bil}\sum_{T\in\TT_{\star}}h_T\,\enorm{u_\star-u_\diamond}_T \enorm{v_{\star}}_T\\
   &\quad+\sum_{V\in\RR^*_\diamond}\Big(\product{R_{\diamond}}{v^*_{\star}-v^*_{\diamond}}_V -
   \sum_{\zeta\in\EE_{V,\diamond}}\product{J_{\diamond}}{v^*_{\star}-v^*_{\diamond}}_{\zeta\backslash\Gamma}\Big). 
 \end{split}
 \end{align}
 Next, we estimate the sum over $T\in\TT_\star$ by the Cauchy-Schwarz inequality.
  Furthermore, we add $v_\star-v_\star$ and use~\eqref{eq:help1disc} to 
  rewrite the sum over the boxes $V\in\RR^*_\diamond$ in~\eqref{eq:dlr:dp1}:
 \begin{align}\label{eq:dlr:dp2}
	 \begin{split}
  \AA(u_\star-u_\diamond,v_{\star})&\leq
  \Big(\sum_{T\in\TT_{\star}}h_T^2\,\enorm{u_\star-u_\diamond}_T^2\Big)^{1/2}\,\enorm{v_{\star}}\\
   &\qquad+\sum_{V\in\RR^*_\diamond}\Big(\product{R_{\diamond}}{v_{\star}-v^*_{\diamond}}_V -
    \sum_{\zeta\in\EE_{V,\diamond}}\product{J_{\diamond}}{v_{\star}-v^*_{\diamond}}_{\zeta\backslash\Gamma} \Big)\\
    &\qquad+\sum_{V'\in\RR^*_{\star}}\Big(\product{R_{\diamond}}{v^*_{\star}-v_{\star}}_{V'} -
    \sum_{\zeta'\in\EE_{V',\star}
 	 \atop \zeta'\subset F\in\EE_{\diamond}}
	 \product{J_{\diamond}}{v^*_{\star}-v_{\star}}_{\zeta'\backslash\Gamma} \Big).
 	 \end{split}
  \end{align}

 Note that $\EE_{V',\star}$ contains also parts of facets from $\TT_{\star}$ which are not
 needed here and which are avoided by $\zeta'\subset F\in\EE_{\diamond}$. To abbreviate notation, let $h_V:=\diam(V)$ and note that $\sigma$-shape regularity implies $h_V\simeq h_T$ for all $V\in\TT^*_\diamond$ and $T\in\TT_\diamond$ with $V\cap T\neq\emptyset$. 
Next, we estimate the two sums over $\RR^*_\diamond$ and $\RR^*_{\star}$:
First, with~\eqref{eq:poincareV} and~\eqref{eq:trace2} of 
Lemma~\ref{lemma:Pi*} and $v^*_\diamond|_V = \Pi^*_\diamond v_\star|_V$ for all $V\in\RR^*_\diamond$, the Cauchy-Schwarz inequality yields
\begin{align}\label{eq:dlr:dp4}
 \nonumber
 &\sum_{V\in\RR^*_\diamond}\Big(\product{R_{\diamond}}{v_{\star}-v^*_{\diamond}}_V -
 \sum_{\zeta\in\EE_{V,\diamond}}\product{J_{\diamond}}{v_{\star}-v^*_{\diamond}}_{\zeta\backslash\Gamma} \Big)
 \\&\quad
 \nonumber
 \lesssim \bigg[\Big(\sum_{V\in\RR^*_\diamond}h_V^2\norm{R_{\diamond}}{L^2(V)}^2\Big)^{1/2}
 + \Big(\sum_{V\in\RR^*_\diamond}\sum_{\zeta\in\EE_{V,\diamond}}h_V\norm{J_{\diamond}}{L^2(\zeta\backslash\Gamma)}^2\Big)^{1/2}\bigg]\,\Big(\sum_{V\in\RR^*_\diamond}\norm{\nabla v_{\star}}{L^2(V)}^2\Big)^{1/2}.
 \intertext{With $\bigcup_{V\in\RR^*_\diamond}V\subset\bigcup_{T\in\RR_\diamond}T$, we hence obtain}
 &\quad
 \lesssim \bigg[\sum_{T\in\RR_\diamond}\big(h_T^2\,\norm{R_{\diamond}}{L^2(T)}^2 + h_T\,\norm{J_{\diamond}}{L^2(\partial T\backslash\Gamma)}^2\big)\bigg]^{1/2}\,\enorm{v_\star}
 = \Big(\sum_{T\in\RR_\diamond}\eta_\diamond(T)^2\Big)^{1/2}\,\enorm{v_\star}.
\end{align}
Note that $\bigcup_{V'\in\RR^*_{\star}}V'\subset\bigcup_{T\in\RR_\diamond}T$. Then, 
with~\eqref{eq3:proplindual} and~\eqref{eq4:proplindual} of Lemma~\ref{lem:proplindual} and $v_\star^*=\II_\star v_\star$, we get as before
\begin{align}\label{eq:dlr:dp5}
\nonumber
 &\sum_{V'\in\RR^*_{\star}}\Big(\product{R_{\diamond}}{v^*_{\star}-v_{\star}}_{V'} -
   \sum_{\zeta'\in\EE_{V',\star} \atop \zeta'\subset F\in\EE_{\diamond}}
   \product{J_{\diamond}}{v^*_{\star}-v_{\star}}_{\zeta'\backslash\Gamma} \Big)
\nonumber
 \\&\quad
 \lesssim \bigg[\sum_{T\in\RR_\diamond}\big(h_T^2\,\norm{R_{\diamond}}{L^2(T)}^2 + h_T\,\norm{J_{\diamond}}{L^2(\partial T\backslash\Gamma)}^2\big)\bigg]^{1/2}\,\enorm{v_\star}
 = \Big(\sum_{T\in\RR_\diamond}\eta_\diamond(T)^2\Big)^{1/2}\,\enorm{v_\star}.
\end{align}
Combining~\eqref{eq:dlr:dp4}--\eqref{eq:dlr:dp5} with~\eqref{eq:dlr:dp2}, we obtain
\begin{align*}
 \AA(u_\star-u_\diamond,v_{\star})
 \lesssim\bigg[
 \Big(\sum_{T\in\TT_{\star}}h_T^2\,\enorm{u_\star-u_\diamond}_T^2\Big)^{1/2}
 + \Big(\sum_{T\in\RR_\diamond}\eta_\diamond(T)^2\Big)^{1/2}\bigg]
 \,\enorm{v_{\star}}.
\end{align*}
Finally, ellipticity of $\AA(\cdot,\cdot)$ and the choice of $v_\star = u_\star-u_\diamond$ 
show
\begin{align*}
 \enorm{u_\star-u_\diamond}^2
 \lesssim \AA(u_\star-u_\diamond,v_{\star})
 \lesssim \bigg[
 \sum_{T\in\TT_{\star}}h_T^2\,\enorm{u_\star-u_\diamond}_T^2
 + \sum_{T\in\RR_\diamond}\eta_\diamond(T)^2\bigg]^{1/2}
 \,\enorm{u_\star-u_\diamond}.
\end{align*}
This proves~\eqref{eq:prop:dlr} and concludes the proof. 
\end{proof}

\subsection{Proof of Theorem~\ref{theorem:mns}}
\label{section:theorem}%
Suppose that the initial triangulation $\TT_0$ is sufficiently fine such that the following assumptions~(i)--(iii) are satisfied:
\begin{itemize}
\item[\rm(i)] For all $\TT_\star\in\refine(\TT_0)$, the FVM system~\eqref{eq:disc_systemfvem} is well-posed. In particular, Lemma~\ref{lemma:A1-A2} proves that stability~(A1) and reduction~(A2) are satisfied.
\item[\rm(ii)] Proposition~\ref{prop:mns:linear} is valid and, in particular, the general quasi-orthogonality~(A3) is satisfied.
\item[\rm(iii)] The constant $\c{bil}$ from Proposition~\ref{prop:dlr} satisfies $\c{bil}\norm{h_0}{L^\infty(\Omega)}\le1/2$, so that Proposition~\ref{prop:dlr}, in fact, proves the discrete reliability~(A4).
\end{itemize}
Finally, let $\widetilde\MM_\ell\subseteq\TT_\ell$ be a set of minimal cardinality which satisfies the D\"orfler marking~\eqref{eq:doerfler:eta} for the error estimator. Then, the additional assumption of Theorem~\ref{theorem:mns} (ii) and the choice of the marked elements $\MM_\ell^\eta\subseteq\MM_\ell\subseteq\TT_\ell$ in Algorithm~\ref{algorithm:mns} imply that $\#\MM_\ell \le \Cmns\,\#\MM_\ell^\eta \le \Cmns\Cmark\,\#\widetilde\MM_\ell$. Altogether, the assumptions of~\cite[Theorem~4.1]{axioms} are fulfilled, and~\eqref{eq:mns:linear}--\eqref{eq:mns:optimal} follow for our adaptive FVM of Algorithm~\ref{algorithm:mns}.\qed

\begin{figure}
\begin{center}
	\subfigure[\label{subfig:bsp1meshT0}$\TT_0$ (16 elements).]
	{\includegraphics[width=0.3\textwidth]
	{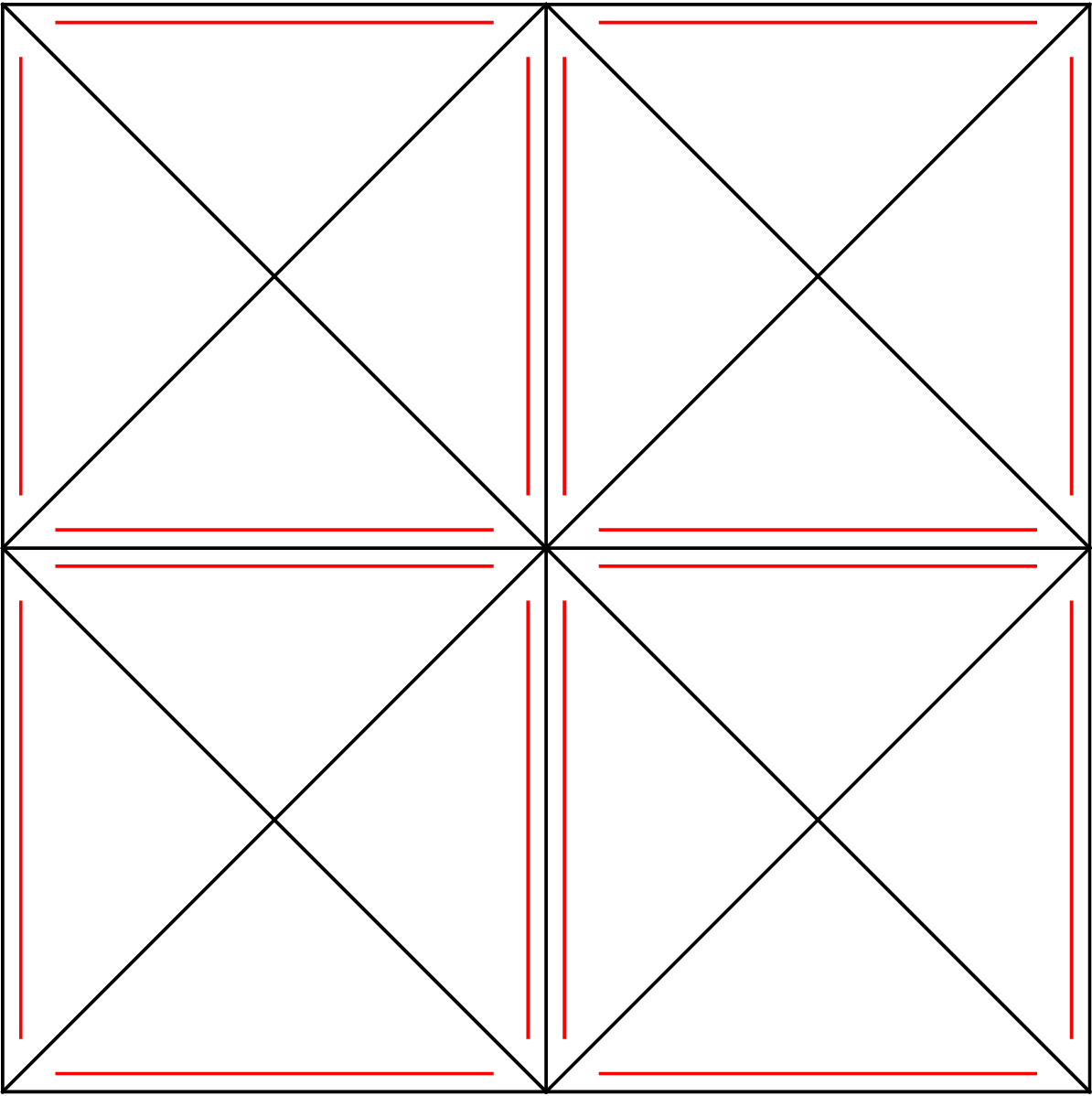}}
	\hspace{0.2\textwidth}
	\subfigure[\label{subfig:bsp1meshT8}$\TT_8$ (216 elements).]
	{\includegraphics[width=0.3\textwidth]
	{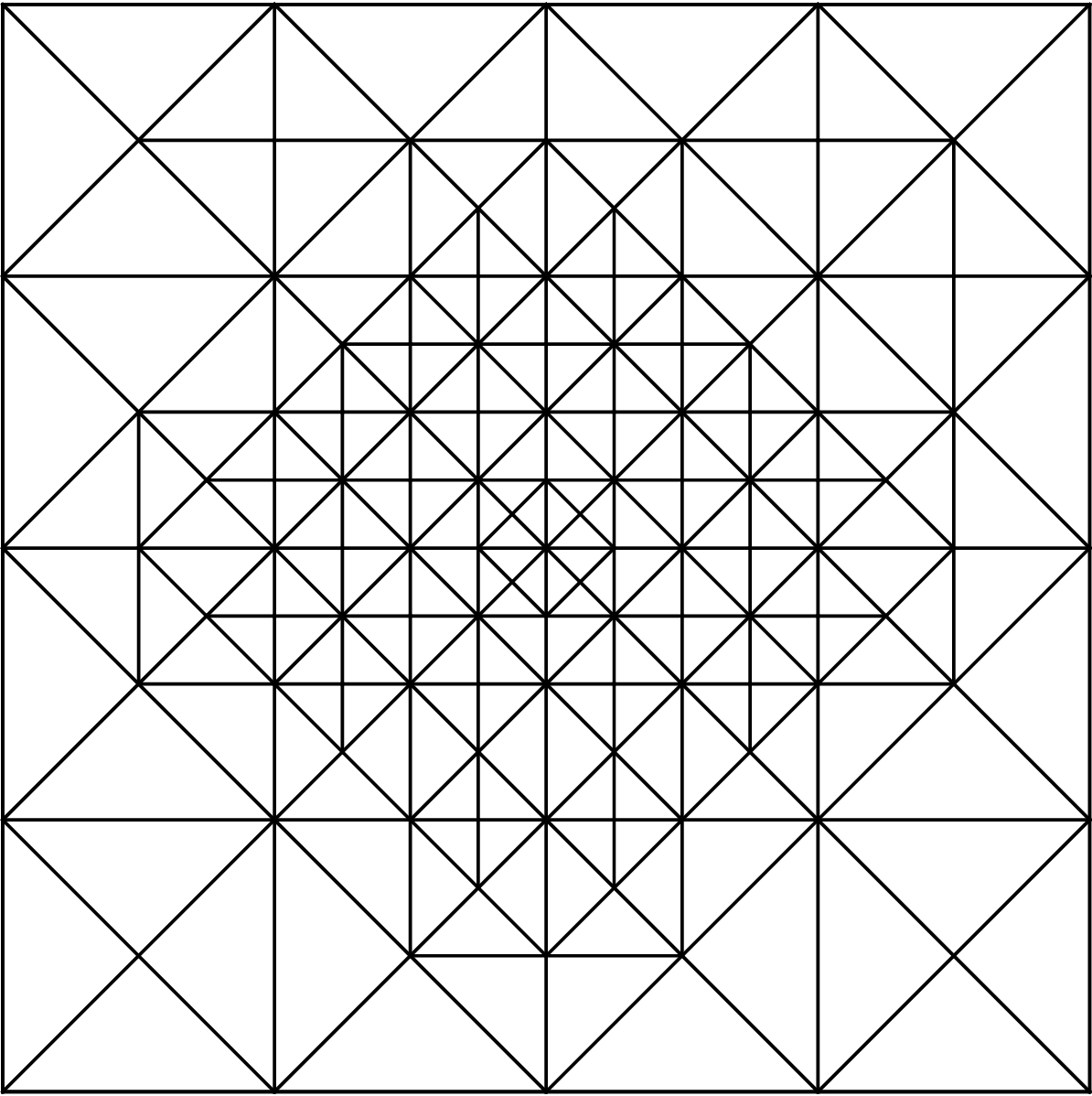}}
	\subfigure[\label{subfig:bsp1meshT16}$\TT_{16}$ (3838 elements).]
	{\includegraphics[width=0.3\textwidth]
	{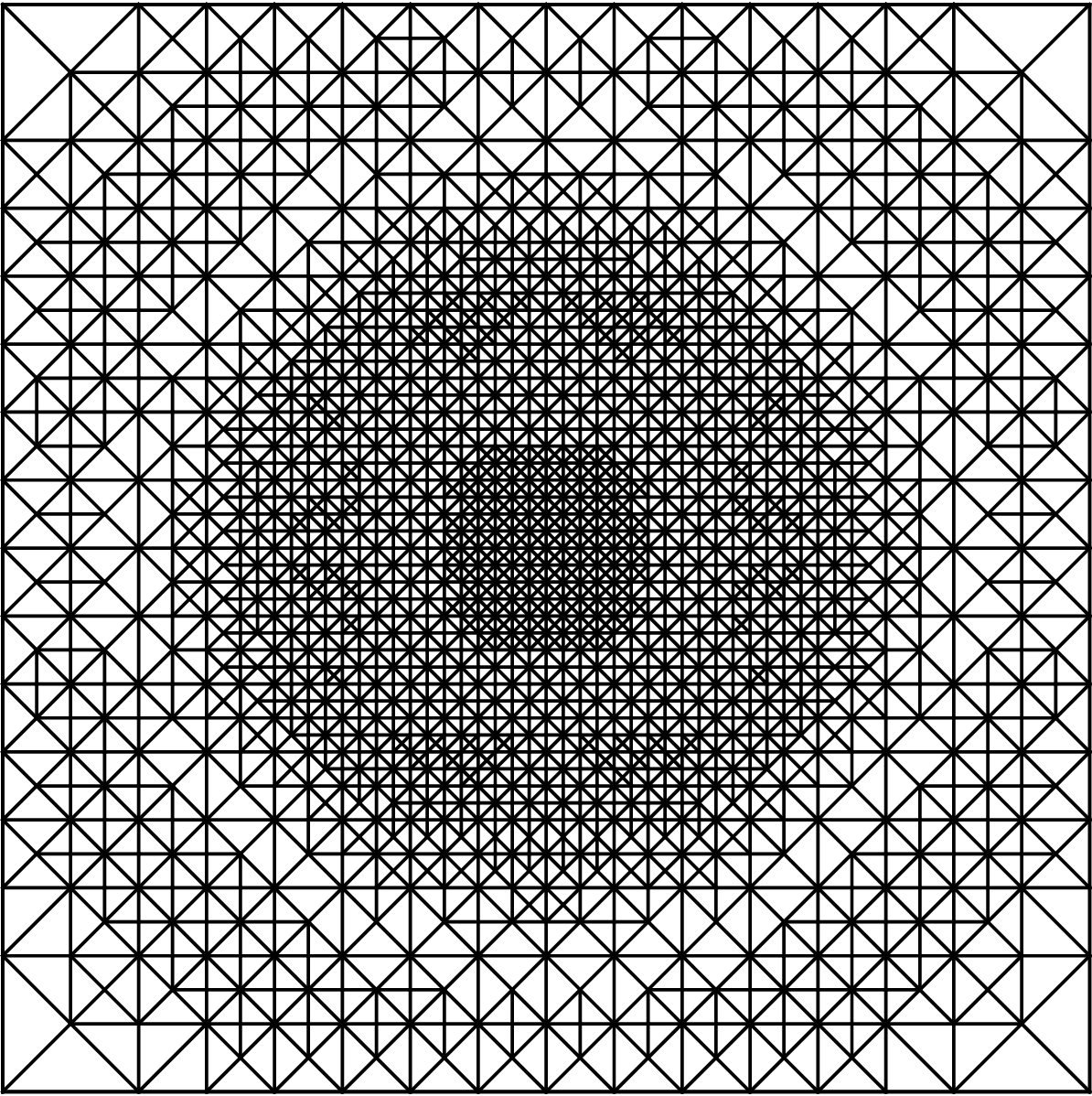}}
	\hspace{0.05\textwidth}
	\subfigure[\label{subfig:bsp1sol}Solution ($\TT_{16}$).]
	{\includegraphics[width=0.46\textwidth]
	{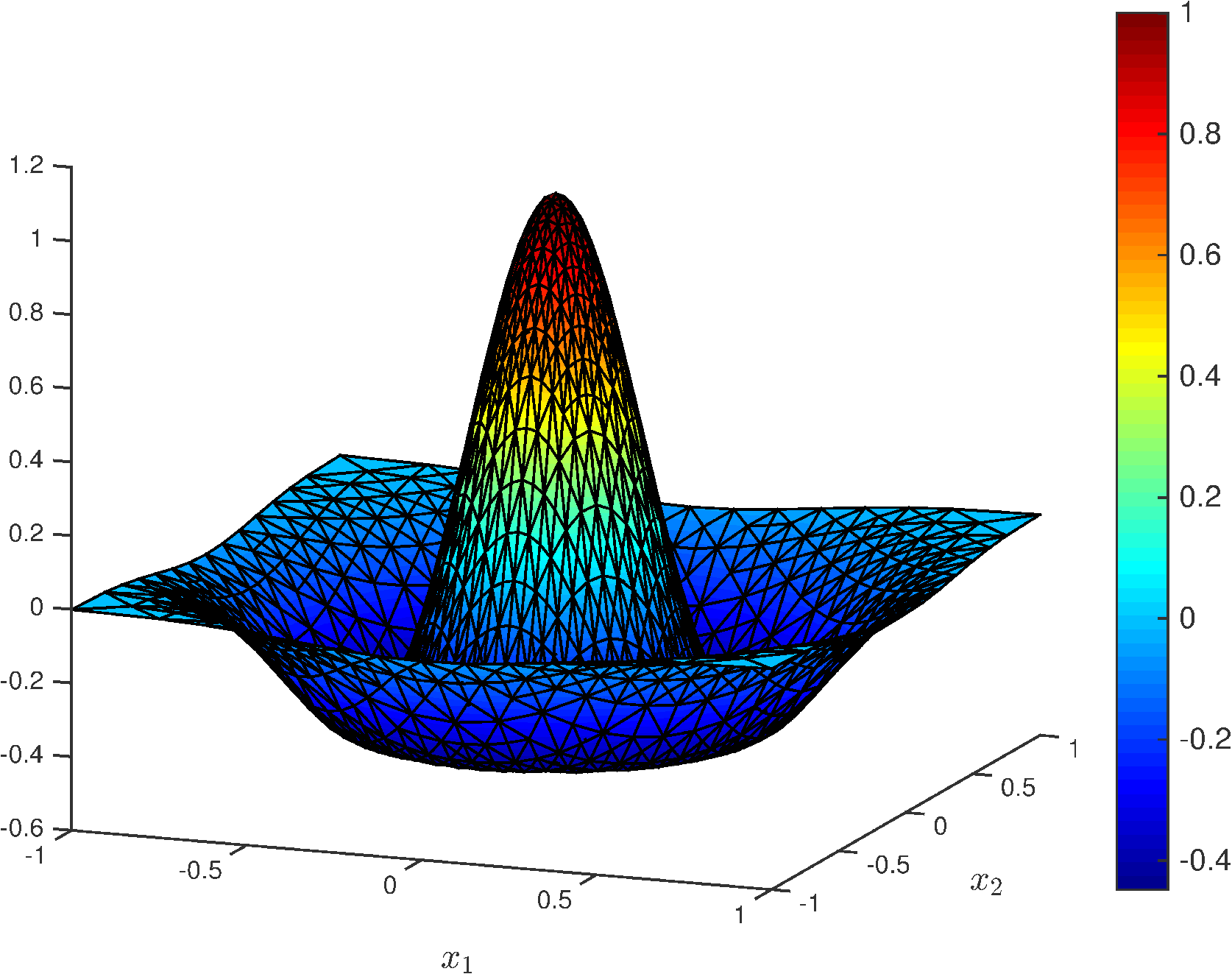}}
\end{center}
\caption{\label{fig:bsp1meshsol}%
Experiment with smooth solution from Section~\ref{ex:bsp1}: Initial triangulation $\TT_0$ with 
NVB reference edges as well as adaptively generated meshes 
$\TT_8$ resp.\ $\TT_{16}$, and discrete FVM solution calculated on $\TT_{16}$.}
\end{figure}
\begin{figure}
\begin{center}
\begin{psfrags}%
\psfragscanon%
%
\psfrag{s12}[l][l]{\scriptsize$1$}%
\psfrag{s04}[l][l]{\small $\eta_\ell$ (uni.)}%
\psfrag{s08}[l][l]{\small $\osc_\ell$ (ada.)}%
\psfrag{s14}[l][l]{\scriptsize$1$}%
\psfrag{s05}[l][l]{\small $E_\ell$ (uni.)}%
\psfrag{s07}[l][l]{\small $\eta_\ell$ (ada.)}%
\psfrag{s01}[l][l]{\small $\osc_\ell$ (uni.)}%
\psfrag{s02}[l][l]{\small $E_\ell$ (ada.)}%
\psfrag{s09}[l][l]{\scriptsize$1$}%
\psfrag{s10}[l][l]{\scriptsize$\; 1/2$}%
\psfrag{s06}[t][t]{\small number of elements}%
\psfrag{s11}[b][b]{\small error, estimator, oscillation}%
%
\color[rgb]{0.15,0.15,0.15}%
%
\psfrag{x01}[t][t]{\scriptsize${10^{1}}$}%
\psfrag{x02}[t][t]{\scriptsize${10^{2}}$}%
\psfrag{x03}[t][t]{\scriptsize${10^{3}}$}%
\psfrag{x04}[t][t]{\scriptsize${10^{4}}$}%
\psfrag{x05}[t][t]{\scriptsize${10^{5}}$}%
\psfrag{x06}[t][t]{\scriptsize${10^{6}}$}%
\psfrag{x07}[t][t]{\scriptsize${10^{7}}$}%
%
\psfrag{v01}[r][r]{\scriptsize${10^{-3}}$}%
\psfrag{v02}[r][r]{\scriptsize${10^{-2}}$}%
\psfrag{v03}[r][r]{\scriptsize${10^{-1}}$}%
\psfrag{v04}[r][r]{\scriptsize${10^{0}}$}%
\psfrag{v05}[r][r]{\scriptsize${10^{1}}$}%
\psfrag{v06}[r][r]{\scriptsize${10^{2}}$}%
\psfrag{v07}[r][r]{\scriptsize${10^{3}}$}%
%
\includegraphics[width=0.8\textwidth]{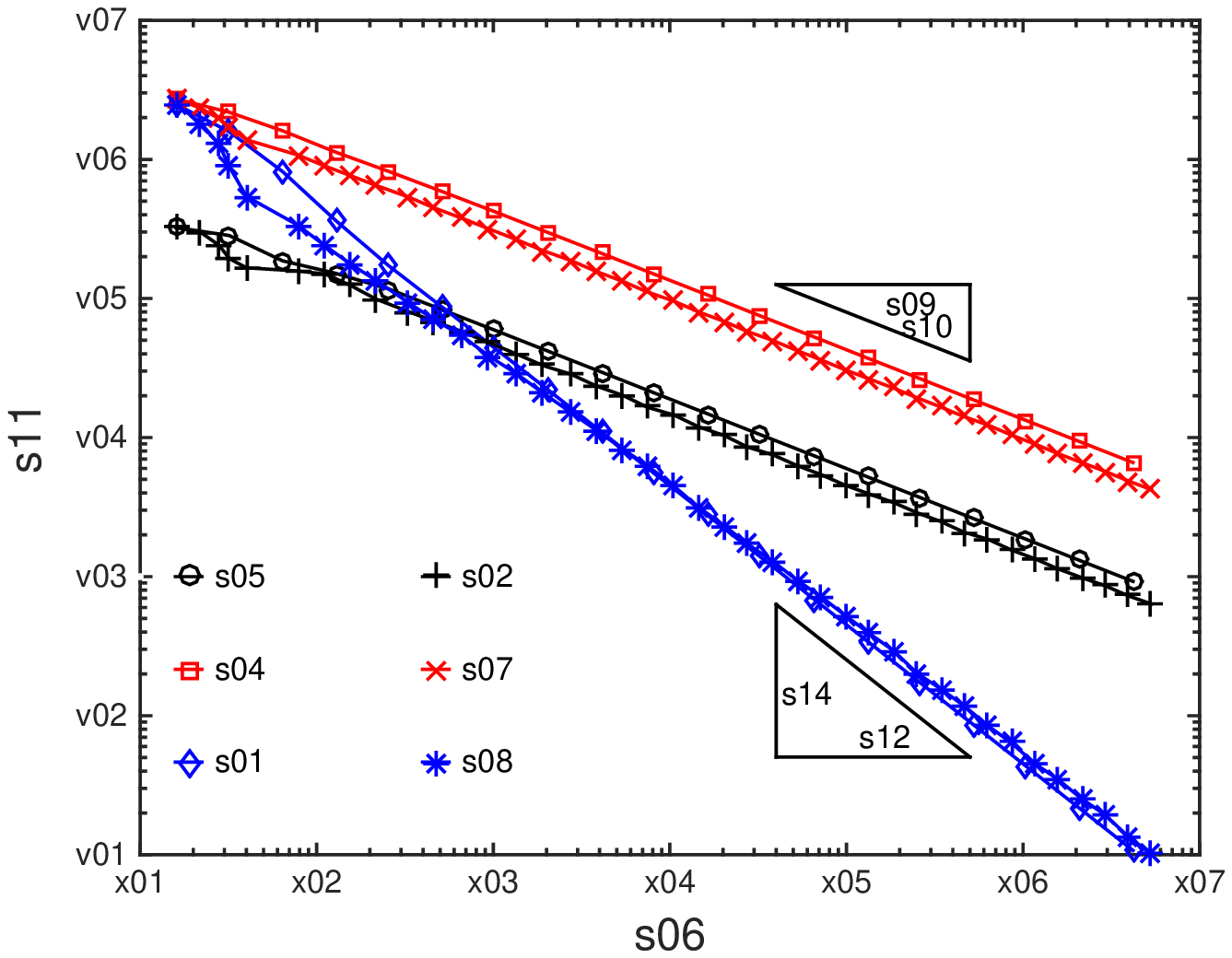}%
\end{psfrags}%

\end{center}
\caption{\label{fig:bsp1error}%
Experiment with smooth solution from Section~\ref{ex:bsp1}: 
Error in the energy norm $E_\ell:=\enorm{u-u_\ell}$, weighted-residual error estimator
$\eta_\ell$, and data oscillations $\osc_\ell$ for uniform and adaptive
mesh-refinement.}
\end{figure}
\begin{figure}
\begin{center}
	\subfigure[\label{subfig:bsp2meshT0}$\TT_0$ (12 elements).]
	{\includegraphics[width=0.3\textwidth]
	{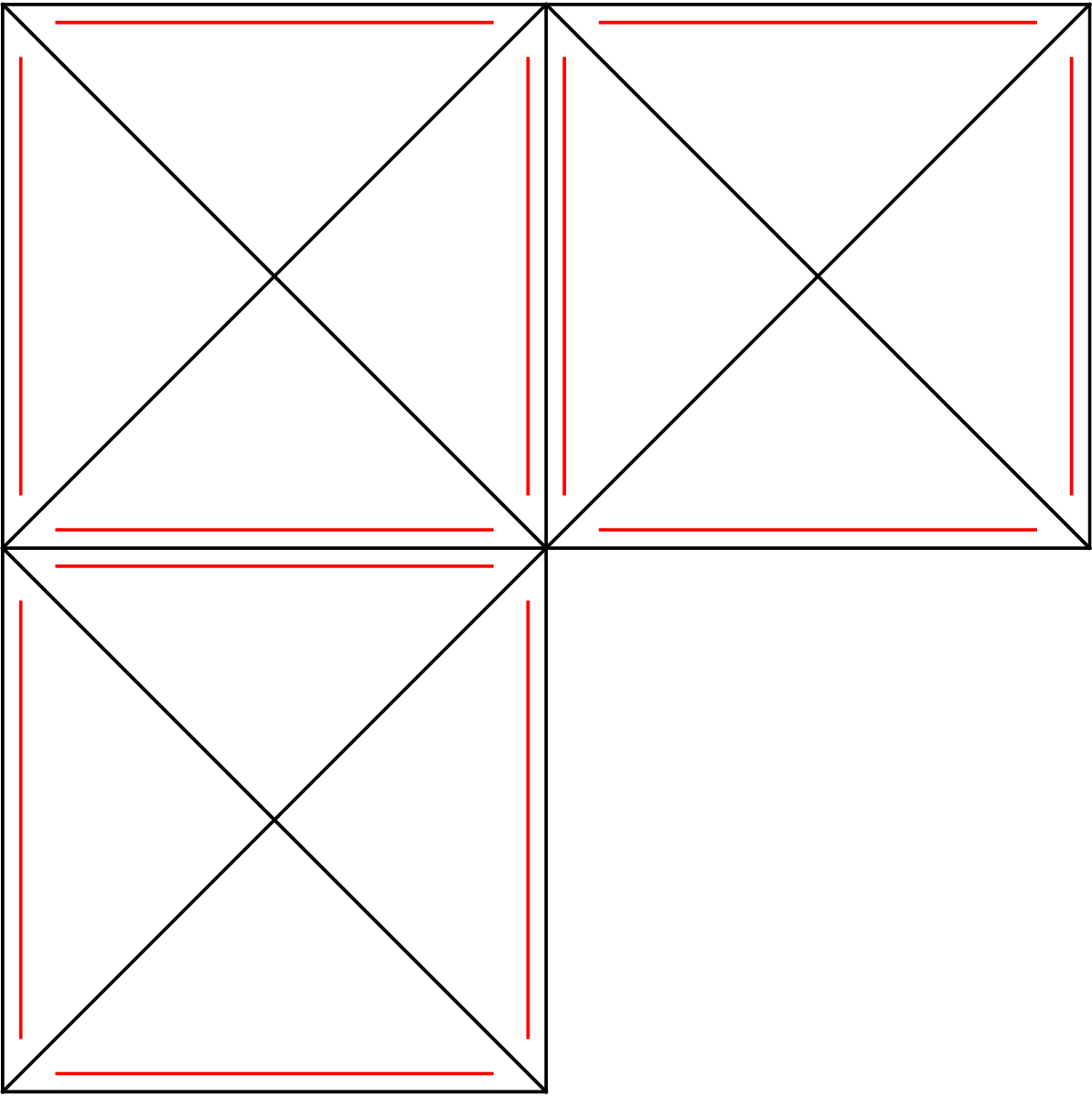}}
	\hspace{0.2\textwidth}
	\subfigure[\label{subfig:bsp2meshT8}$\TT_8$ (204 elements).]
	{\includegraphics[width=0.3\textwidth]
	{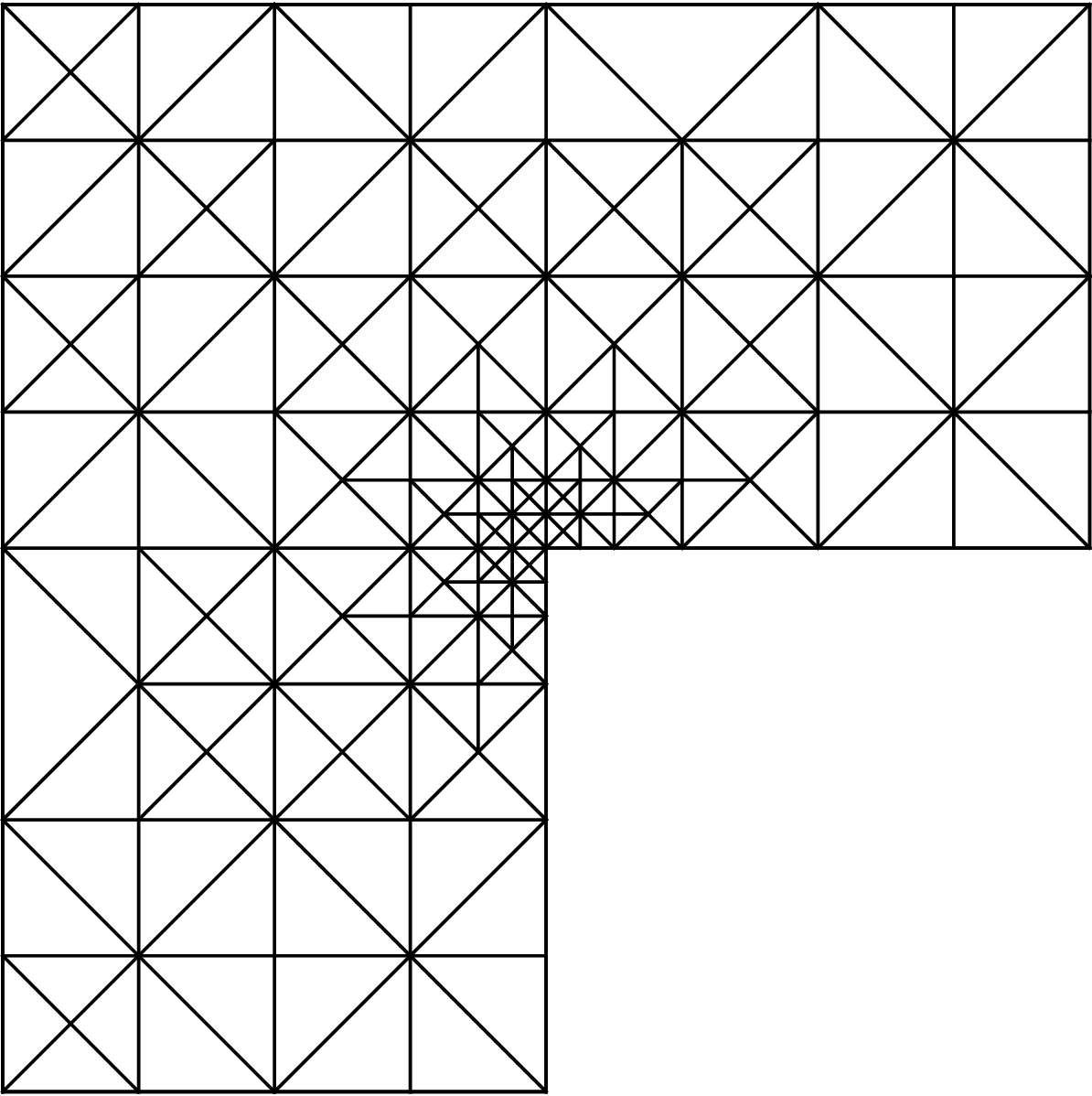}}
	\subfigure[\label{subfig:bsp2meshT16}$\TT_{16}$ (2451 elements).]
	{\includegraphics[width=0.3\textwidth]
	{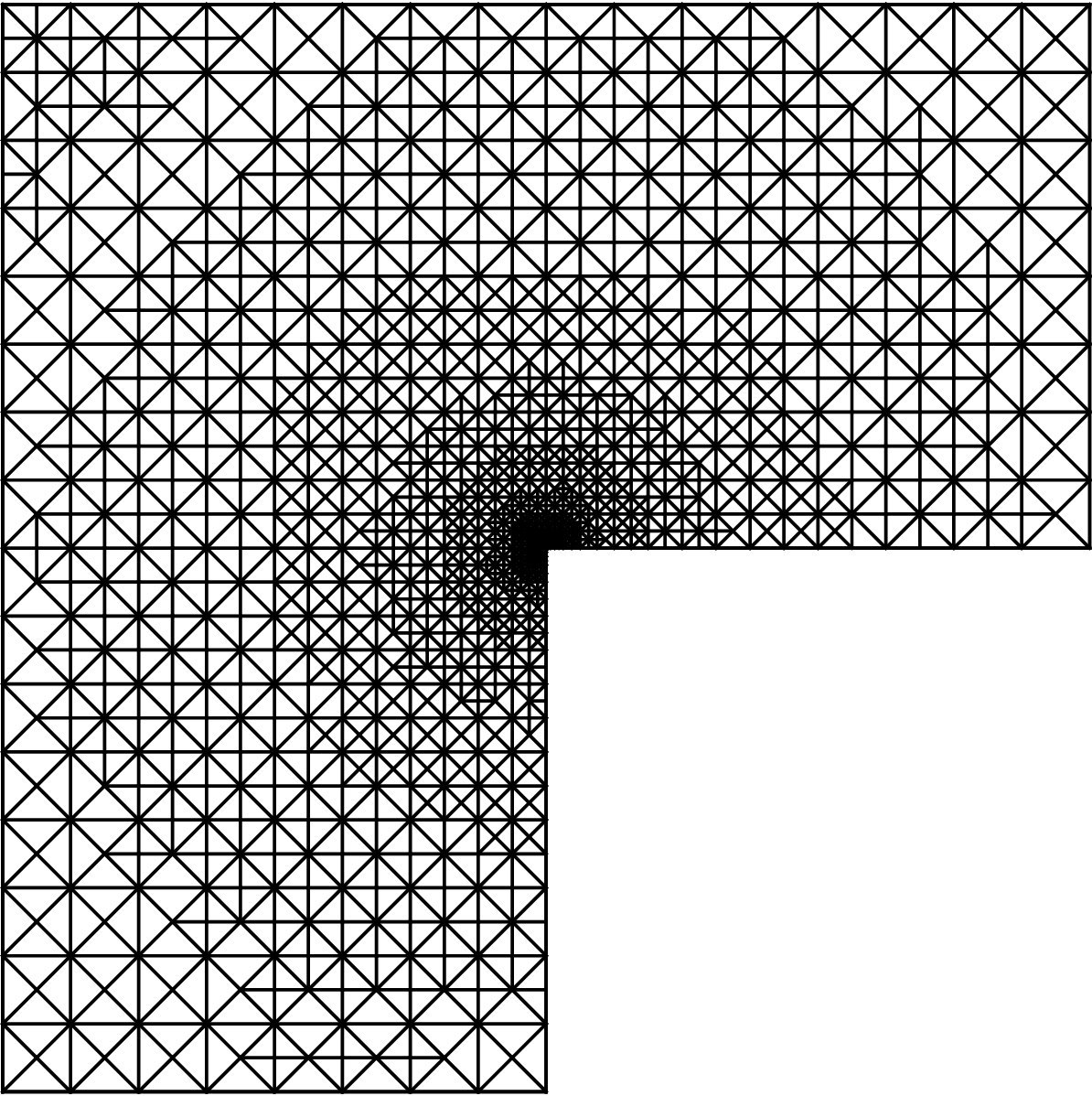}}
	\hspace{0.05\textwidth}
	\subfigure[\label{subfig:bsp2sol}Solution ($\TT_{16}$).]
	{\includegraphics[width=0.46\textwidth]
	{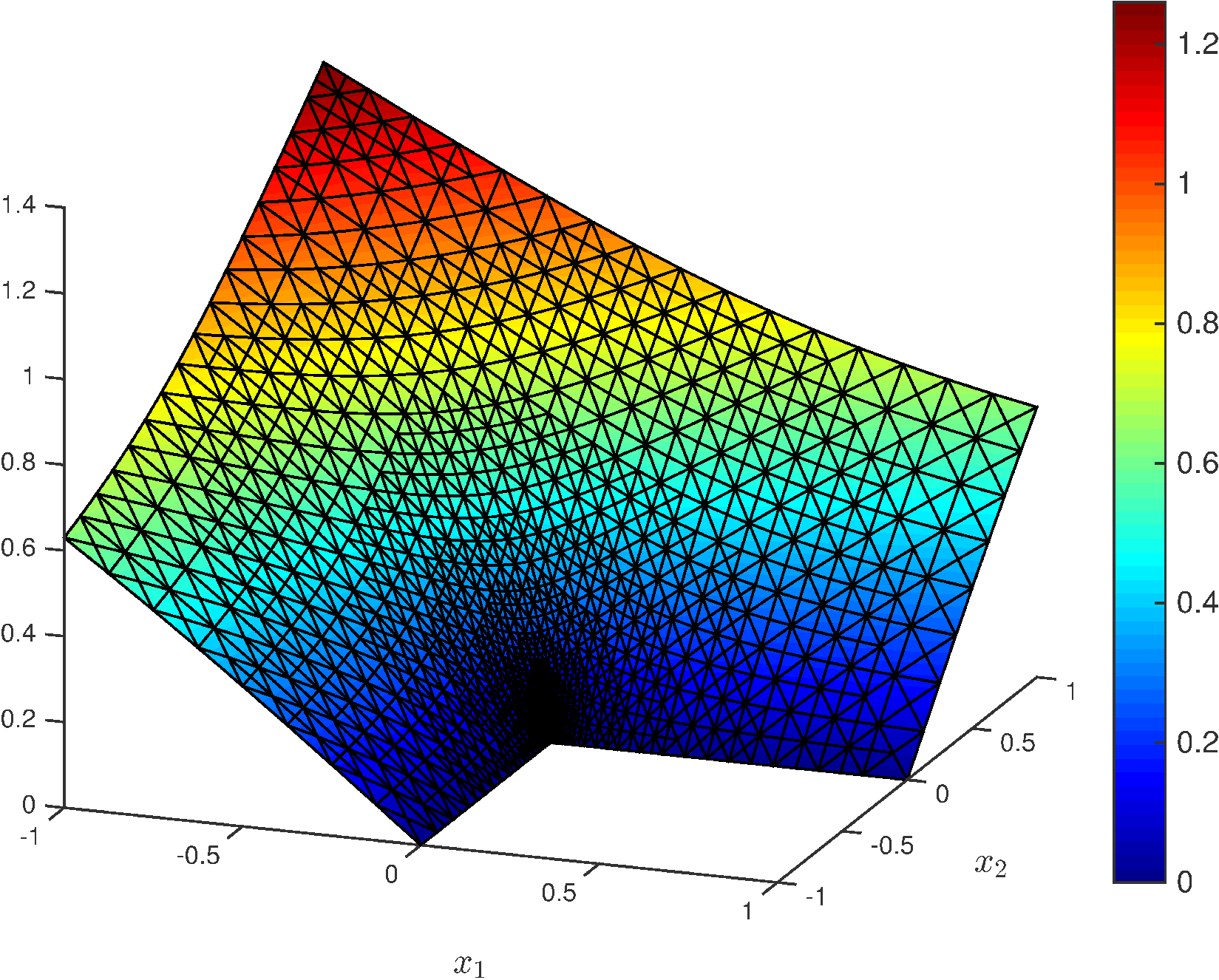}}
\end{center}
\caption{\label{fig:bsp2meshsol}%
Experiment with singular solution from Section~\ref{ex:bsp2}: Initial triangulation $\TT_0$ with 
NVB reference edges as well as adaptively generated meshes 
$\TT_8$ resp.\ $\TT_{16}$, and discrete FVM solution calculated on $\TT_{16}$.}
\end{figure}
\begin{figure}
\begin{center}
\begin{psfrags}%
\psfragscanon%
%
\psfrag{s14}[l][l]{\small $\osc_\ell$ (uni.)}%
\psfrag{s02}[t][t]{\small number of elements}%
\psfrag{s06}[l][l]{\small $E_\ell$ (uni.)}%
\psfrag{s08}[l][l]{\small $\eta_\ell$ (uni.)}%
\psfrag{s12}[l][l]{\small $E_\ell$ (ada.)}%
\psfrag{s03}[b][b]{\small error, estimator, oscillation}%
\psfrag{s09}[l][l]{\scriptsize $1/2$}%
\psfrag{s15}[l][l]{\small $\osc_\ell$ (ada.)}%
\psfrag{s01}[l][l]{\scriptsize $1$}%
\psfrag{s10}[l][l]{\scriptsize $1$}%
\psfrag{s13}[l][l]{\scriptsize $\;1/3$}%
\psfrag{s11}[l][l]{\small $\eta_\ell$ (ada.)}%
\psfrag{s16}[l][l]{\scriptsize $1$}%
\psfrag{s05}[l][l]{\scriptsize $1$}%
%
\color[rgb]{0.15,0.15,0.15}%
%
\psfrag{x01}[t][t]{\scriptsize${10^{1}}$}%
\psfrag{x02}[t][t]{\scriptsize${10^{2}}$}%
\psfrag{x03}[t][t]{\scriptsize${10^{3}}$}%
\psfrag{x04}[t][t]{\scriptsize${10^{4}}$}%
\psfrag{x05}[t][t]{\scriptsize${10^{5}}$}%
\psfrag{x06}[t][t]{\scriptsize${10^{6}}$}%
\psfrag{x07}[t][t]{\scriptsize${10^{7}}$}%
%
\psfrag{v01}[r][r]{\scriptsize${10^{-6}}$}%
\psfrag{v02}[r][r]{\scriptsize${10^{-5}}$}%
\psfrag{v03}[r][r]{\scriptsize${10^{-4}}$}%
\psfrag{v04}[r][r]{\scriptsize${10^{-3}}$}%
\psfrag{v05}[r][r]{\scriptsize${10^{-2}}$}%
\psfrag{v06}[r][r]{\scriptsize${10^{-1}}$}%
\psfrag{v07}[r][r]{\scriptsize${10^{0}}$}%
\psfrag{v08}[r][r]{\scriptsize${10^{1}}$}%
%
\includegraphics[width=0.8\textwidth]{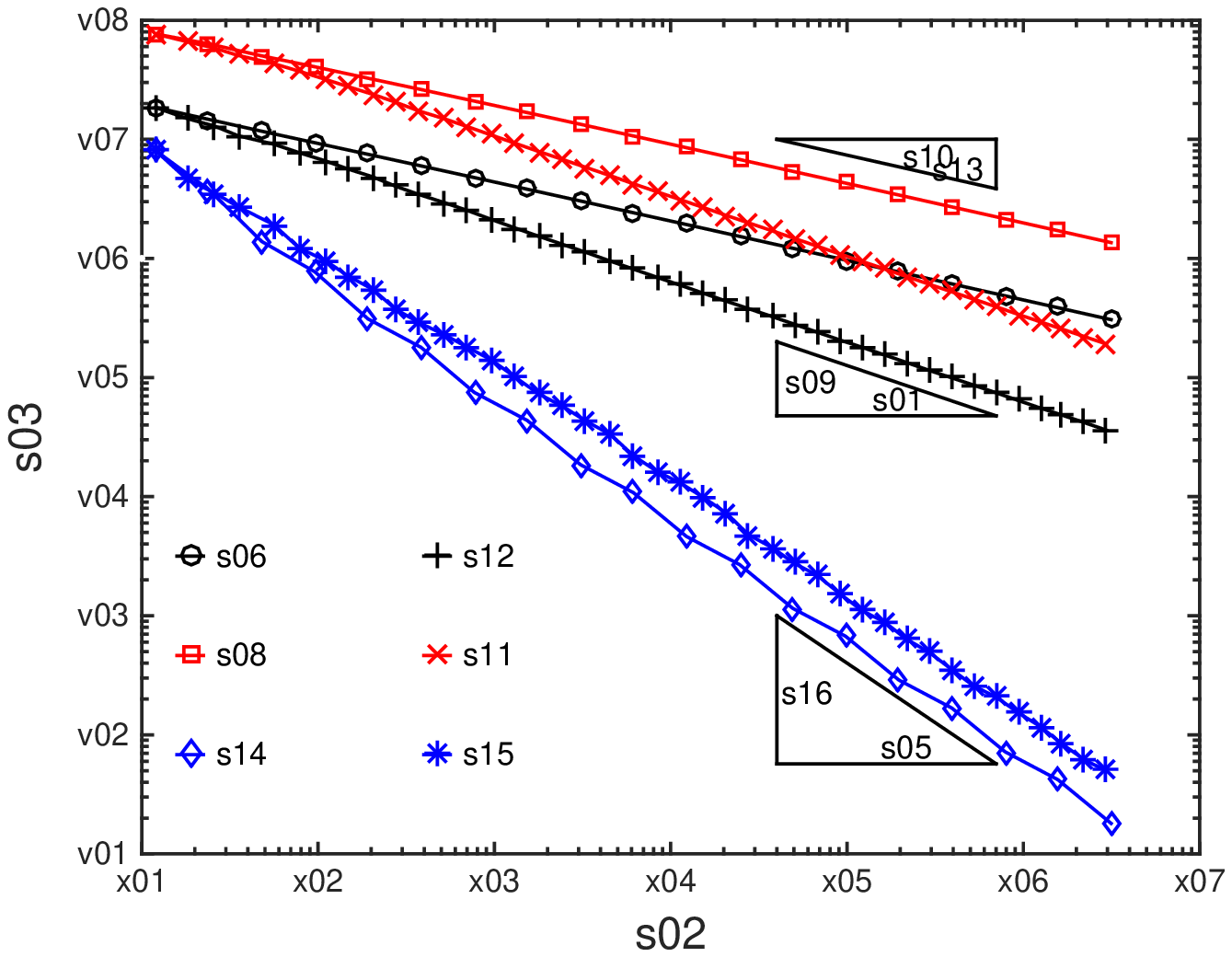}%
\end{psfrags}%
\end{center}
\caption{\label{fig:bsp2error}%
Experiment with smooth solution from Section~\ref{ex:bsp2}: 
Error in the energy norm $E_\ell:=\enorm{u-u_\ell}$, weighted-residual error estimator
$\eta_\ell$, and data oscillations $\osc_\ell$ for uniform and adaptive
mesh-refinement.}
\end{figure}

\section{Numerical experiments}

\noindent
In this section, we illustrate the performance of Algorithm~\ref{algorithm:mns}
with $\theta=0.5=\theta'$ for two examples. In extension of our theory, we consider
the model problem~\eqref{eq:model} with inhomogeneous Dirichlet boundary conditions. 
The numerical experiments are conducted in {\sc Matlab} on a standard
laptop with a dual core $2.8$ GHz processor and $16$ GB memory. 

\subsection{Experiment with smooth solution}
\label{ex:bsp1}
On the square $\Omega=(-1,1)^2$, 
we prescribe
the exact solution $u(x_1,x_2) = (1-10x_1^2-10x_2^2)e^{-5(x_1^2+x_2^2)}$
with $x=(x_1,x_2)\in\R^2$.
We choose the diffusion matrix 
\begin{align*}
	\A=  \left ( \begin{array}{rr}
  10+\cos x_1 & 9\,x_1 x_2 \\
  9\,x_1 x_2 & \;10+\sin x_2
  \end{array}\right),
\end{align*}
so that~\eqref{eq:A} holds with $\lambda_{\min}=0.82293$ and
$\lambda_{\max}=10.84096$.
The right-hand side $f$ is calculated appropriately.
The uniform initial mesh $\TT^{(0)}$ consists of $16$ triangles; see
Figure~\ref{fig:bsp1meshsol}\subref{subfig:bsp2meshT0}. 
In Figure~\ref{fig:bsp1meshsol}\subref{subfig:bsp1meshT8}
and~\ref{fig:bsp1meshsol}\subref{subfig:bsp1meshT16} we see
an adaptively generated mesh after $8$ and $16$ refinements, respectively.
Figure~\ref{fig:bsp1meshsol}\subref{subfig:bsp1sol} plots the 
smooth solution on the mesh $\TT_{16}$.
As $u$ is smooth, uniform and adaptive mesh-refinement lead to the optimal convergence order $\OO(N^{-1/2})$ with respect to the number $N$ of elements; see 
Figure~\ref{fig:bsp1error}.
The oscillations are of higher order and decrease with $\OO(N^{-1})$.
Table~\ref{tab:bsp}\subref{tab:bsp1} shows the experimental validation of the additional assumption in Theorem~\ref{theorem:mns}~{\rm(ii)} that marking for the data oscillations is negligible.

\subsection{Experiment with generic singularity}
\label{ex:bsp2}
On the L-shaped domain $\Omega=\linebreak(-1,1)^2\backslash \big([0,1]\times[-1,0]\big)$,
we prescribe the exact solution 
$u(x_1,x_2) = r^{2/3}\sin(2\varphi/3)$ 
in polar coordinates $r\in\R_0^+$, $\varphi\in[0,2\pi[$,
and $(x_1,x_2) = r(\cos\varphi,\sin\varphi)$.
Then, $u$ has a generic singularity at the reentrant corner $(0,0)$, which
leads to $u\in H^{1+2/3-\varepsilon}(\Omega)$ for all
$\varepsilon>0$.
We choose the diffusion matrix 
\begin{align*}
	\A=  \left ( \begin{array}{rr}
  5+(x_1^2+x_2^2)\cos x_1 & (x_1^2+x_2^2)^2 \\
  (x_1^2+x_2^2)^2 & \;5+(x_1^2+x_2^2)\sin x_2
  \end{array}\right)
\end{align*}
so that~\eqref{eq:A} holds with $\lambda_{\min}=0.46689$ and
$\lambda_{\max}=5.14751$.
The right-hand side $f$ is calculated appropriately.
The uniform initial mesh $\TT^{(0)}$ consists of $12$ triangles. Some further adaptively generated meshes together with a plot of the discrete solution are shown in Figure~\ref{fig:bsp2meshsol}.

For uniform mesh refinement, we observe the expected suboptimal convergence order of $\OO(N^{-1/3})$, while adaptive mesh-refinement regains  the optimal convergence order of $\OO(N^{-1/2})$;
see Figure~\ref{fig:bsp2error}. As in the experiment of Section~\ref{ex:bsp1}, the oscillations are of higher order $\OO(N^{-1})$.
See Table~\ref{tab:bsp}\subref{tab:bsp2} for the experimental validation of the additional assumption in Theorem~\ref{theorem:mns}~{\rm(ii)} that marking for the data oscillations is negligible.

\begin{table}[!t]\small
\begin{center}
\subfigure[\label{tab:bsp1}Section~\ref{ex:bsp1}.]{%
\begin{tabular}{r|rcc}
\hline 
$\ell$  &  $\#\TT_\ell$  & $\#\MM_\ell/\#\MM_\ell^\eta $ & $ \osc(\MM_\ell^\eta)^2 / \osc_\ell^2 $ \\
\hline 
\hline 
  0 &       16 &        1.000 &      0.634 \\  
  1 &       22 &        1.000 &      0.613 \\  
  2 &       28 &        1.000 &      0.704 \\  
  3 &       32 &        1.000 &      0.769 \\  
  4 &       40 &        1.214 &      0.338 \\  
  5 &       78 &        1.111 &      0.449 \\  
  6 &      112 &        1.133 &      0.292 \\  
  7 &      156 &        1.119 &      0.410 \\  
  8 &      216 &        1.062 &      0.393 \\  
  9 &      331 &        1.198 &      0.263 \\  
 10 &      460 &        1.014 &      0.474 \\  
 11 &      660 &        1.049 &      0.371 \\  
 12 &      944 &        1.027 &      0.430 \\  
 13 &     1,340 &       1.025 &      0.404 \\  
 14 &     1,914 &       1.019 &      0.383 \\  
 15 &     2,752 &       1.026 &      0.374 \\  
 16 &     3,838 &       1.015 &      0.358 \\  
 17 &     5,428 &       1.003 &      0.449 \\  
 18 &     7,430 &       1.013 &      0.359 \\  
 19 &    10,572 &       1.003 &      0.445 \\  
 20 &    14,462 &       1.019 &      0.322 \\  
 21 &    20,264 &       1.004 &      0.431 \\  
 22 &    27,532 &       1.004 &      0.455 \\  
 23 &    38,402 &       1.010 &      0.323 \\  
 24 &    52,366 &       1.000 &      0.539 \\  
 25 &    72,386 &       1.007 &      0.401 \\  
 26 &    98,144 &       1.000 &      0.509 \\  
 27 &   135,076 &       1.004 &      0.445 \\  
 28 &   184,006 &       1.000 &      0.605 \\  
 29 &   251,668 &       1.002 &      0.475 \\  
 30 &   341,940 &       1.001 &      0.488 \\  
 31 &   461,354 &       1.000 &      0.616 \\  
 32 &   634,922 &       1.004 &      0.415 \\  
 33 &   852,264 &       1.000 &      0.663 \\  
 34 &  1,171,426 &      1.002 &      0.465 \\  
 35 &  1,567,542 &      1.000 &      0.611 \\  
 36 &  2,150,232 &      1.000 &      0.521 \\  
 37 &  2,893,626 &      1.000 &      0.652 \\  
 38 &  3,932,562 &      1.000 &      0.593 \\  
 39 &  5,335,740 &      1.000 &      0.493 \\  
\hline 
\end{tabular}
}
\hfill
\subfigure[\label{tab:bsp2}Section~\ref{ex:bsp2}.]{%
\begin{tabular}{r|rcc}
\hline 
$\ell$  &  $\#\TT_\ell$ & $\#\MM_\ell/\#\MM_\ell^\eta $ & $ \osc(\MM_\ell^\eta)^2 / \osc_\ell^2 $ \\
\hline 
\hline 
  0 &       12 &        1.667 &      0.143 \\  
  1 &       18 &        1.750 &      0.115 \\  
  2 &       26 &        1.400 &      0.108 \\  
  3 &       35 &        1.222 &      0.062 \\  
  4 &       56 &       1.200 &      0.104 \\  
  5 &       78 &        1.643 &      0.028 \\  
  6 &      110 &          1.350 &      0.135 \\  
  7 &      148 &          1.161 &      0.290 \\  
  8 &      204 &        1.111 &      0.268 \\  
  9 &      274 &        1.048 &      0.423 \\  
 10 &      370 &        1.168 &      0.223 \\  
 11 &      525 &       1.069 &      0.324 \\  
 12 &      704 &       1.063 &      0.296 \\  
 13 &      961 &      1.015 &      0.442 \\  
 14 &     1,314 &       1.003 &      0.475 \\  
 15 &     1,784 &         1.037 &      0.345 \\  
 16 &     2,451 &        1.000 &      0.639 \\  
 17 &     3,305 &      1.015 &      0.417 \\  
 18 &     4,562 &      1.000 &      0.595 \\  
 19 &     6,161 &     1.001 &      0.482 \\  
 20 &     8,344 &       1.011 &      0.440 \\  
 21 &    11,316 &        1.000 &      0.635 \\  
 22 &    15,249 &       1.000 &      0.528 \\  
 23 &    20,631 &        1.000 &      0.577 \\  
 24 &    27,742 &       1.014 &      0.451 \\  
 25 &    37,566 &    1.000 &      0.655 \\  
 26 &    50,139 &       1.011 &      0.437 \\  
 27 &    67,722 &       1.000 &      0.571 \\  
 28 &    90,543 &       1.000 &      0.523 \\  
 29 &   121,136 &      1.005 &      0.471 \\  
 30 &   163,221 &      1.000 &      0.715 \\  
 31 &   216,681 &      1.025 &      0.361 \\  
 32 &   292,527 &       1.000 &      0.545 \\  
 33 &   389,411 &     1.000 &      0.582 \\  
 34 &   521,975 &    1.013 &      0.437 \\  
 35 &   699,195 &     1.000 &      0.678 \\  
 36 &   928,417 &    1.012 &      0.418 \\  
 37 &  1,246,972 &     1.000 &      0.561 \\  
 38 &  1,658,877 &     1.000 &      0.585 \\  
 39 &  2,224,754 &      1.003 &      0.481 \\  
 40 &  2,959,035 &      1.000 &      0.659 \\  
\hline 
\end{tabular}
}
\end{center}
\caption{Experimental results on marking strategy: We compute $\widetilde{C}_{MNS}:=\#\MM_\ell/\#\MM_\ell^\eta \le 2$ and see that the additional assumption in Theorem~\ref{theorem:mns}~{\rm(ii)} is experimentally verified. In addition, we compute $\widetilde\theta':=\osc_\ell(\MM_\ell^\eta)^2/\osc_\ell^2 \ge 0.02$, i.e., the choice $\theta = 0.5$, $\theta' = 0.02$ would guarantee $\MM_\ell = \MM_\ell^\eta$ in Algorithm~\ref{algorithm:mns}.}
\label{tab:bsp}
\end{table}

\bibliographystyle{alpha}
\bibliography{afvm}  

\end{document}